\documentclass[a4paper,11pt, reqno]{amsart}
\usepackage[T1]{fontenc}
\usepackage[english]{babel}
\usepackage[margin=1in]{geometry}
\usepackage{amsmath,amssymb,amsthm}
\usepackage{bbm,csquotes}
\usepackage{multicol}
\usepackage{mathrsfs}
\usepackage{enumitem}
\usepackage{mathtools}
\usepackage{graphicx,xtab}
\usepackage{makecell}
\usepackage{multirow}
\usepackage{tabularx}
\usepackage{calc}
\usepackage{array, booktabs, caption, ragged2e}
\usepackage[dvipsnames]{xcolor}
\usepackage[colorlinks=true,linkcolor=NavyBlue,urlcolor=RoyalBlue,citecolor=PineGreen,bookmarks=true,bookmarksopen=true,bookmarksopenlevel=2,unicode=true,linktocpage]{hyperref}
\usepackage{float}
\usepackage{tikz}
\usepackage{comment}
\usepackage{mathtools}
\mathtoolsset{showonlyrefs}

\usepackage[style=ext-numeric-comp,backend=biber,giveninits=true,isbn=false,url=false,doi=false,articlein=false,maxbibnames=99]{biblatex}
\DeclareFieldFormat[article]{title}{#1} 
\DeclareFieldFormat[phdthesis]{title}{#1} 
\AtEveryBibitem{%
\clearfield{number}\clearfield{series}\clearfield{number}\clearfield{month}}

\usepackage{pifont}
\newcommand{\cmark}{\ding{51}}%
\newcommand{\xmark}{\ding{55}}
 
\newcommand{\resample}[2]{{}^{#2}\mkern -0.5mu {#1}}

\newtheorem{theorem}{Theorem}[section]
\newtheorem{corollary}[theorem]{Corollary}
\newtheorem{proposition}[theorem]{Proposition}
\newtheorem{lemma}[theorem]{Lemma}

\numberwithin{equation}{section}

\theoremstyle{definition}
\newtheorem{definition}[theorem]{Definition}
\newtheorem{examplex}[theorem]{Example}

\newenvironment{example}
  {\pushQED{\qed}\examplex}
  {\popQED\endexamplex}

\theoremstyle{remark}
\newtheorem{remark}[theorem]{Remark}
\newtheorem{remarks}[theorem]{Remarks}
\newtheorem*{remark*}{Remark}

\newcommand{\1}[1]{{\mathbbm{1}\mkern -1.5mu}{\{#1\}}}
\newcommand{\ind}[1]{{\mathbbm{1}\mkern -1.5mu}_{\{#1\}}}

\newcommand{\R}{{\mathbb R}}

\newcommand{\N}{{\mathbb N}}
\newcommand{\ZP}{{\mathbb Z}_+}
\newcommand{\RP}{{\mathbb R}_+}
\newcommand{\Sp}[1]{{\mathbb S}^{#1}}
\newcommand{\rhoH}{\rho_{\mathrm{H}}}

\newcommand{\rhoone}{\rho_{1}}
\newcommand{\eps}{\varepsilon}

\DeclareMathOperator*{\Exp}{\mathbb{E}}
\let\Pr\relax
\DeclareMathOperator*{\Pr}{\mathbb{P}}

\DeclareMathOperator*{\Var}{\mathbb{V}ar}
\DeclareMathOperator*{\Cov}{\mathbb{C}ov}

\DeclareMathOperator*{\argmax}{arg \mkern 1mu \max}
 
\DeclareMathOperator*{\diam}{diam} 
\DeclareMathOperator*{\cl}{cl} 

\DeclareMathOperator*{\perim}{perim}

\DeclareMathOperator*{\hull}{hull}

\newcommand{\tra}{{\scalebox{0.6}{$\top$}}}
\newcommand{\per}{{\mkern -1mu \scalebox{0.5}{$\perp$}}}

\newcommand{\spara}[1]{\sigma^2_{#1}}

\newcommand{\re}{{\mathrm{e}}}

\newcommand{\ud}{{\mathrm d}}

\newcommand{\cA}{{\mathcal A}}
\newcommand{\cB}{{\mathcal B}}
\newcommand{\cC}{{\mathcal C}}
\newcommand{\cD}{{\mathcal D}}
\newcommand{\cE}{{\mathcal E}}
\newcommand{\cF}{{\mathcal F}}
\newcommand{\cG}{{\mathcal G}}
\newcommand{\cH}{{\mathcal H}}
\newcommand{\cI}{{\mathcal I}}
\newcommand{\cJ}{{\mathcal J}}
\newcommand{\cK}{{\mathcal K}}

\newcommand{\cN}{{\mathcal N}}

\newcommand{\cR}{{\mathcal R}}
\newcommand{\cS}{{\mathcal S}}

\newcommand{\cU}{{\mathcal U}}
\newcommand{\cV}{{\mathcal V}}

\newcommand{\tZ}{{\widetilde{Z}}}

\newcommand{\as}{\ \text{a.s.}}

\newcommand{\toP}[1]{\xrightarrow[#1 \to \infty]{\mathrm{prob.}}}

\newcommand{\tod}[1]{\xrightarrow[#1 \to \infty]{\mathrm{dist.}}}
\newcommand{\toL}[2]{\xrightarrow[#1 \to \infty]{L^{{#2}}}}

\newcommand{\overbar}[1]{\mkern 1.5mu\overline{\mkern-1.5mu#1\mkern-1.5mu}\mkern 1.5mu}

\newcommand{\ba}{{\boldsymbol{a}}}

\newcommand{\be}{{\boldsymbol{\rm e}}}
\newcommand{\0}{{\boldsymbol{0}}}
\newcommand{\bmu}{{\boldsymbol{\mu}}}
\newcommand{\hbmu}{\skew{1.5}\widehat{\boldsymbol{\mu}}}
\newcommand{\hbmuperp}{\skew{1.5}\widehat{\boldsymbol{\mu}}_\per}

\newcommand{\oM}{\skew{2}\overbar{M}}
\newcommand{\oMperp}{\oM^\perp}

\newcommand{\eqd}{\overset{\mathrm{d}}{=}}

\makeatletter
\def\namedlabel#1#2{\begingroup
    (#2)%
    \def\@currentlabel{#2}%
    \phantomsection\label{#1}\endgroup
}
\makeatother

\newlist{myenumi}{enumerate}{10}
\setlist[myenumi]{leftmargin=0pt, labelindent=\parindent, listparindent=\parindent, labelwidth=0pt, itemindent=!, itemsep=1pt, parsep=4pt}

\newlist{thmenumi}{enumerate}{10}
\setlist[thmenumi]{leftmargin=0pt, labelindent=\parindent, listparindent=\parindent, labelwidth=0pt, itemindent=!}

\title[Fluctuations for convex hulls of multiple random walks]{Fluctuations for diameter and perimeter of convex hulls of multiple random walks} 

\date{\today}

\allowdisplaybreaks
\begin{document}

\author[W.~Cygan]{Wojciech Cygan}
\address{Institute of Mathematics, University of Wrocław}
\email{wojciech.cygan@uwr.edu.pl}

\author[T.~Kralj]{Tomislav Kralj}
\address{University of Zagreb Faculty of Science, Department of Mathematics}
\email{tomislav.kralj@math.pmf.unizg.hr}

\author[N.~Sandri\'c]{Nikola Sandri\'c}
\address{University of Zagreb Faculty of Science, Department of Mathematics}
\email{nikola.sandric@math.pmf.unizg.hr}

\author[S.~\v Sebek]{Stjepan \v Sebek}
\address{University of Zagreb Faculty of Electrical Engineering and Computing}
\email{stjepan.sebek@fer.unizg.hr}

\author[A.~Wade]{\\ Andrew Wade}
\address{Department of Mathematical Sciences, Durham University}
\email{andrew.wade@durham.ac.uk}

\author[M.D.~Wong]{Mo~Dick Wong}
\address{Department of Mathematics, The University of Hong Kong}
\email{mdwong@hku.hk}

\begin{abstract}
We study the diameter and perimeter of the convex hull generated by finitely many independent planar random walks whose increments have finite second moments. The large-time fluctuations are governed by the geometry of the polygon formed by the drift vectors. 
We develop an $L^2$-approximation framework, based on Wald-type maximal central limit theorems, which reduces the asymptotic analysis of the hull to a finite collection of endpoint, maximal-projection, and Brownian support-function terms. 
For the diameter, we obtain general max-type limit theorems, Gaussian in the case of a unique extremal diametrical pair and typically non-Gaussian when several extremal pairs compete. 
For the perimeter, we prove a general distributional limit: non-zero extremal drifts contribute maxima of Gaussian projections, while zero-drift extremal walks contribute Brownian support-function terms. The results recover the previously known Gaussian regimes (the case of one or two walks) and identify the non-Gaussian limits in the degenerate and boundary cases left open (even for two walks). 
We also give $L^2$ approximations of the convex hull by simpler random sets, under Hausdorff and $\ell_1$ metrics on compact convex sets.
Our proofs work under the optimal finite second moment assumption.
\end{abstract}

\maketitle

 \medskip


\tableofcontents

\section{Introduction}
\label{sec:introduction}

\subsection{Convex hulls of planar random walks}
\label{sec:convex-hulls}

This paper  develops an $L^2$-approximation framework  to study the fluctuations of  geometric quantities associated with multiple random walks.
Given the $n$-step trajectories of $N \geq 1$ independent random walks in $\R^2$ whose increments have finite second moments, denote by $D_n$ and $L_n$ the diameter and perimeter, respectively, of the convex hull $\cH_n$ generated by the union of their trajectories. 
The polygon $\cC_\bmu$ generated by the drift vectors of the random walks governs the first-order asymptotic behaviour of these quantities, via a  geometrical law of large numbers that says $n^{-1} D_n \to \diam \cC_\bmu$ and $n^{-1} L_n \to \perim \cC_\bmu$. We show that the polygon $\cC_\bmu$ also influences the fluctuations of these quantities in a delicate way. In addition to  explicit limit theorems for   arbitrary numbers and configurations of random walks, we also give companion results on optimal approximation of the convex hull by simpler stochastic objects in appropriate metrics. 
The framework is motivated by, and may be viewed as a multidimensional geometric analogue of, an $L^2$
 version of Wald's maximal central limit theorem for one-dimensional random walk.

For the diameter, 
we establish, for arbitrary $N\geq 1$, an explicit distributional limit  
\begin{equation}
\label{eq:diameter-limit}
\frac{D_n-n \diam \cC_\bmu}{\sqrt n}
\tod{n} \Delta;
\end{equation}
see Theorem~\ref{thm:many-walks-diameter} below for a formal statement. 
The  limit $\Delta$ in~\eqref{eq:diameter-limit} is generically Gaussian, but becomes non-Gaussian whenever several random walks are competing to determine the diameter. The accompanying $L^2$ approximation, Theorem~\ref{thm:diameter-approximation-general}, identifies exactly which endpoints and zero-drift trajectories can affect the diameter at this scale.

For the perimeter, our main result (Theorem~\ref{thm:perim_main_theorem} below) gives, for arbitrary $N \geq 1$, an explicit distributional limit, as an independent sum,
\begin{equation}
\label{eq:perimeter-limit}
\frac{L_n-n \perim \cC_\bmu}{\sqrt n}
\tod{n}
 \Pi^+ +  \Pi^0,
\end{equation}
where $\Pi^+$ is an explicit Gaussian functional determined by the non-zero extremal drift vectors, and $\Pi^0$ is an explicit Brownian support-function functional associated with zero-drift walks when the origin is an extremal point of the drift polygon.

In~\eqref{eq:diameter-limit} and~\eqref{eq:perimeter-limit}, the geometry of the drift polygon $\cC_\bmu$ determines both the type of the limiting distribution and its precise form, which also involves the increment covariance matrices of the walks. In generic cases the limits are Gaussian, recovering earlier results on $N=1$ or $N=2$ walks, such as~\cite{wx-drift,mcrw,ikss}; in boundary or degenerate configurations, non-Gaussian limits appear naturally, and here even for $N=2$ our results provide novel non-Gaussian limits.
Figure~\ref{fig0} shows some examples of convex hulls for $N \in \{2, 3\}$ walks.


\begin{figure}
	\scalebox{0.48}{\begin{minipage}{0.28\textwidth}
			\centering
			\includegraphics[width=0.6\textwidth]{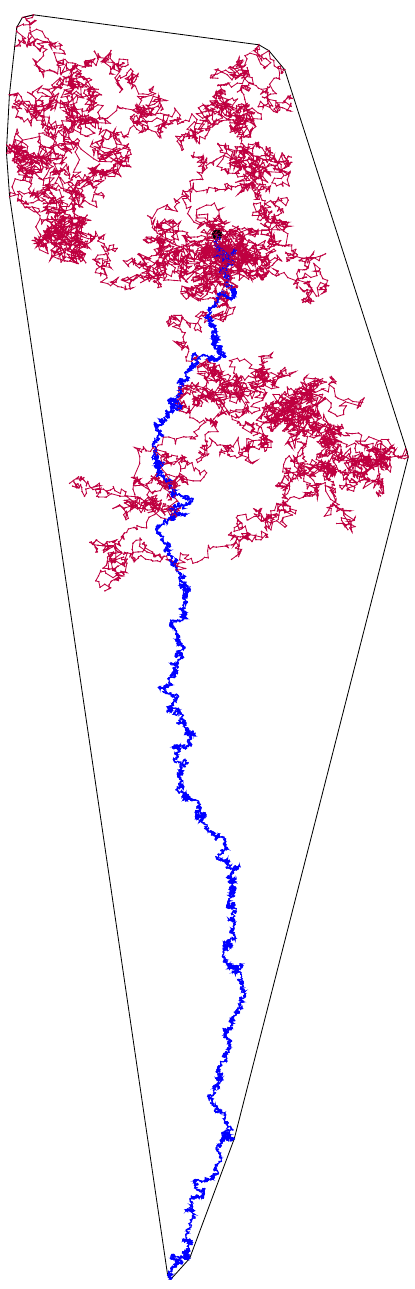} 
		\end{minipage} \hfill
		\begin{minipage}{0.66\textwidth}
			\centering
			\includegraphics[width=0.95\textwidth]{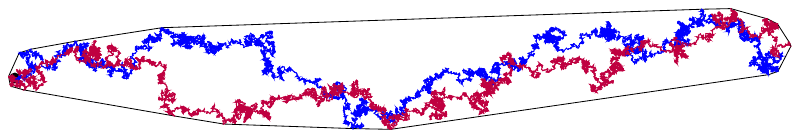}
			
			\vspace{1cm}
			
			\includegraphics[angle=90,width=0.5\textwidth]{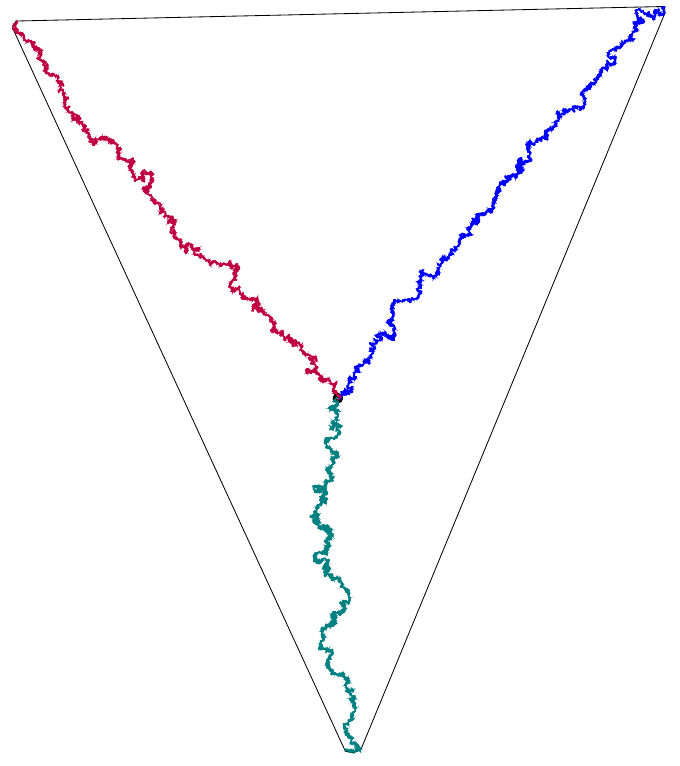}
	\end{minipage}}
	\caption{Schematic of random walk configurations. \emph{Left.} One walk with zero drift and one non-zero drift: perimeter (Example~\ref{ex:perim_ice-cream}) and diameter (Corollary~\ref{cor:diam_two_walks}\ref{cor:diam_two_walks-iv}) have non-Gaussian limits involving functionals of Brownian motions. \emph{Top right.} Two walks with the same non-zero drift: the perimeter has a non-Gaussian limit (Example~\ref{ex:perim_two-the-same}) while the diameter limit is the maximum of independent Gaussians (Theorem~\ref{thm:many-walks-diameter}). \emph{Bottom right.} Three non-zero drifts forming an equilateral triangle: the perimeter has a Gaussian limit described in Corollary~\ref{cor:zero-in-interior-clt}, while the diameter limit is described by  Theorem~\ref{thm:many-walks-diameter}. }
	\label{fig0}
\end{figure}

We give formal statements of~\eqref{eq:diameter-limit}--\eqref{eq:perimeter-limit} later, after introducing appropriate notation. To set the scene, in~\S\ref{sec:convex-hulls-background} we describe known results for a  single random walk. The  narrative of this paper sets these existing results, and our new contributions, into an $L^2$-approximation framework cast as a multidimensional analogue of Wald's maximal central limit theorem, and we introduce this perspective in~\S\ref{sec:wald-framework}. Notation and background for convex hulls of multiple random walks is in~\S\ref{sec:known-gaussian}, and the outline of our contributions and the organization of the bulk of the paper is in~\S\ref{sec:results}. In~\S\ref{sec:two-walk-intro} we focus on  $N=2$ where our results complete the picture from earlier literature~\cite{ikss,tomislav}. First we comment briefly on motivation and related literature.

The study of convex hulls of stochastic processes is motivated by understanding extremal geometry of processes,
by seeking multidimensional extensions of one-dimensional fluctuation or record-value theory, and by applications in
ecology (see the next paragraph) or by statistical set estimation. 
Early work on the convex hull of planar Brownian motion goes back to L\'evy~\cite{levy} and, for random walk, to Spitzer \& Widom~\cite{spitzer-widom}; we refer to~\cite{majumdar-survey} for a survey, and additional background and motivation. 
In the last decade or so, activity has increased significantly; for a sample  of classical and more modern work, we refer to~\cite{levy,snyder-steele,khoshnevisan,av,akmv,css,bgm,klm,kvz,vz,wx-drift}.

In ecology, an animal's, or group of animals',  territory or \emph{home range}~\cite[p.~26]{seton} is an important indicator of how that animal or group interacts with the environment, resources, and other animals, and is a central consideration in conservation efforts. A standard way to estimate the home range from location data is by the convex hull of the data points. Early variants of this approach were used for rabbits~\cite{dalke-sime}, rodents~\cite{blair} and  birds~\cite{odum}; we refer for developments and recent perspectives to~\cite{fieberg,nilsen,moorcroft,worton}. Random walks have remained popular models of  animal movement since Karl Pearson's work in the early 1900s: see e.g.~\cite{morales,codling}. With such applications in mind, convex hulls of multiple random walks or Brownian motions have attracted attention recently and there are several active strands of research; see, e.g.~\cite{rfmc,dewenter,majumdar-survey,ikss,rfz,rf}.

\subsection{Limit theorems for one random walk}
\label{sec:convex-hulls-background}

The context for our contributions needs us to describe existing results for a single random walk. 
Let $Z, Z_1, Z_2, \ldots$ be i.i.d.~$\R^2$-valued random variables, with $\Exp [ \| Z\|^2] < \infty$,
where $\|  \cdot  \|$ denotes the Euclidean norm. Then the 
increment mean vector and increment covariance matrix we denote by
\begin{equation}
	\label{eq:mu-sigma}
	\bmu := \Exp [Z], \quad \text{and}\quad \Sigma := \Exp [ ( Z -\bmu) (Z-\bmu)^\tra ];
\end{equation}
we view vectors in $\R^2$ as column vectors when they appear in formulas, but sometimes write them in-line as row vectors for notational convenience.  
Define the associated random walk $(S_n)_{n \in \ZP}$ by $S_0 := \0$ (the origin) and $S_n := \sum_{i=1}^n Z_i$ for $n \in \N$. (Here $\N := \{1,2,3,\ldots\}$ and $\ZP := \{ 0 \} \cup \N$.) 
The \emph{convex hull} of the first $n$ steps of the random walk is $\cH_n := \hull \{ S_0, \ldots, S_n \}$, the smallest convex set to contain $S_0, \ldots, S_n$. For each $n$, the set $\cH_n$ is a (random) non-empty closed convex
subset of $\R^2$, with polygonal boundary, 
and we denote by $D_n := \diam \cH_n$ and $L_n := \perim \cH_n$ its diameter and perimeter length, respectively.

The random set $\cH_n$ and its associated random variables (like $D_n, L_n$) have been
studied for several decades, and many of their properties are now well understood;
see e.g.~\cite{spitzer-widom,snyder-steele,majumdar-survey,av,vz,akmv} for a selection of the literature.
It is a consequence of the functional form of the strong law
of large numbers (SLLN) that $\lim_{n \to \infty} n^{-1} \cH_n = \hull \{ \0 , \bmu\}$ (line segment), a.s.,
in the space $\cK_0$ of compact, convex subsets of $\R^2$ containing~$\0$, endowed with the Hausdorff metric~$\rhoH$. Hence, by continuous mapping, $\lim_{n\to \infty} n^{-1} D_n = \| \bmu \|$ and $\lim_{n\to \infty} n^{-1} L_n = 2\| \bmu \|$, a.s. See~\cite[\S3.2--\S3.3]{james} and~\cite{mcrw,lmw} for these, and related, results, and further literature. 

In the case $\bmu = \0$ (\emph{zero drift}) it has been shown~\cite{wx-scaling,mcrw}, using
Donsker's theorem, that $n^{-1/2} D_n$ and $n^{-1/2} L_n$ have $n\to \infty$ distributional 
limits corresponding to  diameter and perimeter of the unit-time trajectory of planar Brownian motion (up to a linear transformation to account for the covariance $\Sigma$). These limits are non-Gaussian, since they are a.s.~positive and have positive variances. The perhaps more subtle case is when $\bmu \neq \0$ (\emph{non-zero drift}),
in which   Gaussian limit results are known, for \emph{centred} and scaled quantities. Writing  
\begin{equation}
	\label{eq:spara-def}
	\hbmu := \bmu / \| \bmu \|, \quad \text{and}\quad
	\spara{} := \Var ( \hbmu^\tra Z ) = \hbmu^\tra \Sigma \hbmu, 
\end{equation}
it holds for the 
diameter~\cite[Corollary~1.10]{mcrw} and for the 
perimeter~\cite{wx-drift} that 
\begin{equation}
	\label{eq:known-D-L-CLT} \frac{D_n -   n \| \bmu \|}{\sqrt{n}}  \tod{n} \cN (0, \spara{}), 
	\quad \text{and}\quad  \frac{L_n - 2 n \| \bmu \|}{\sqrt{n}}  \tod{n} \cN (0, 4\spara{}).
\end{equation}
The degenerate case~$\spara{} =0$ is permitted, in which case~\eqref{eq:known-D-L-CLT} 
says  that $n^{-1/2}( D_n -  n \| \bmu \| )$ and $n^{-1/2}( L_n - 2 n \| \bmu \| )$ both converge in probability to~$0$ as $n \to \infty$; finer results exist in this special case, 
which, under some mild additional moments hypotheses,  say   that fluctuations for $D_n$ are on scale~$O(1)$~\cite[Theorem~1.11]{mcrw} and for $L_n$
are on scale $\sqrt{ \log n }$~\cite{akmv}.

\subsection{An \texorpdfstring{$L^2$}{L-2}-approximation framework}
\label{sec:wald-framework}
The Gaussian limit results~\eqref{eq:known-D-L-CLT} 
are proved in~\cite{wx-drift} and~\cite{mcrw} using a martingale-difference idea
that yields $L^2$ approximations with interesting geometrical interpretations. For $\bmu \neq \0$, the $L^2$ approximation results for 
diameter~\cite[Theorem~1.9]{mcrw} 
and perimeter (from \cite{wx-drift}, reformulated in~\cite[Theorem~1.7]{mcrw}) are
\begin{equation}
	\label{eq:known-D-L-approx}  \frac{D_n - \hbmu^\tra S_n}{\sqrt{n}} \toL{n}{2} 0, 
	\quad \text{and}\quad 
	\frac{L_n - 2 \hbmu^\tra S_n}{\sqrt{n}} \toL{n}{2} 0. \end{equation}
The $L^2$ approximations~\eqref{eq:known-D-L-approx}, together with the classical central limit theorem (CLT) for $\hbmu^\tra S_n$, directly yield~\eqref{eq:known-D-L-CLT}, and, indeed, a joint convergence result for $(D_n,L_n)$. The  approximation also yields the variance asymptotics $n^{-1} \Var D_n \to \spara{}$ and
$n^{-1} \Var L_n \to 4\spara{}$ 
as $n \to \infty$.

A relatively easy result  
(Lemma~\ref{lem:norm-projection} below) says that $n^{-1/2} ( \| S_n \| - \hbmu^\tra S_n ) \to 0$ in $L^2$, which means that~\eqref{eq:known-D-L-approx}
can be given the striking geometrical interpretations
\begin{equation}
	\label{eq:line-segment-limits}
	\frac{D_n - \diam \ell_n}{\sqrt{n}} \toL{n}{2} 0,\quad \text{and}\quad \frac{ L_n - \perim \ell_n }{\sqrt{n}} \toL{n}{2} 0, \end{equation}
where $\ell_n := \hull \{ \0, S_n\}$ is the line segment between the two endpoints of the random walk. Re-casting the results of~\cite{wx-drift,mcrw} in the form~\eqref{eq:line-segment-limits} emphasizes that the line-segment approximation goes much farther than the SLLN when $\bmu \neq \0$. 

There is a subtlety in the convergence in~\eqref{eq:line-segment-limits}. Indeed, despite~\eqref{eq:line-segment-limits},  sets  $\ell_n$ and $\cH_n$ differ with $\sqrt{n}$-scale fluctuations, i.e., $n^{-1/2} \rhoH ( \cH_n , \ell_n )$ will not vanish in probability. To see this, fix a unit vector $\hbmuperp$ orthogonal to $\hbmu$. The continuous functional $F_\per (K) := \sup_{x \in K} | \hbmuperp^\tra x |$ on $\cK_0$ has $F_\per (\ell_n ) = |\hbmuperp^\tra S_n|$
while $F_\per ( \cH_n ) = \max_{0 \leq i \leq n} | \hbmuperp^\tra S_i |$, and so Donsker's theorem shows that~$n^{-1/2} ( F_\per (\cH_n) - F_\per (\ell_n) )$ has a non-trivial limit distribution. 
On the other hand, there is a metric~$\rhoone$ (see \S\ref{sec:hausdorff}) for which $n^{-1/2} \rhoone (\cH_n , \ell_n) \to 0$ (see~Theorem~\ref{thm:perimeter-deviation-metric} in~\S\ref{sec:perimeter_deviation_metric} below), which implies the perimeter part of~\eqref{eq:line-segment-limits}, but not the diameter part.
We make some further observations on set approximation in~\S\ref{sec:hausdorff} below, in the context of multiple random walks.

\subsection{Multiple random walks}
\label{sec:known-gaussian}

Consider $N \in \N$ independent
random walks, $S^{(1)}, \ldots, S^{(N)}$,
where for each $k \in \{1,\ldots,N\}$, 
$S^{(k)} := ( S^{(k)}_n )_{n\in\ZP}$, 
for all $n \in \ZP$, 
$S^{(k)}_n := \sum_{i=1}^n Z_{k,i}$, and, for each $k$, the increments $Z_{k,1}, Z_{k,2},\ldots$ are i.i.d.
Analogously to~\eqref{eq:mu-sigma}, for the case of multiple walks we will adopt the following assumptions and notation.

\begin{description}[topsep=3pt,itemsep=3pt]
	\item
	[\namedlabel{ass:many-walks}{M}]
	Suppose that $\sup_{1 \leq k \leq N} \Exp [ \| Z_{k,1} \|^2 ] < \infty$, and denote
	$\bmu_k := \Exp [Z_{k,1}] \in \R^2$ (mean drift) and
	$\Sigma_k := \Exp [ ( Z_{k,1} -\bmu_k) (Z_{k,1} -\bmu_k)^\tra ]$ (increment covariance matrix) associated with each walk~$k \in \{1,\ldots,N\}$. Also define 
	\[
	\cC_\bmu := \hull \{ \0, \bmu_1, \ldots \bmu_N \}.
	\]
\end{description}

Consider the convex hull generated by the $N$ random walks defined by 
\begin{equation}
	\label{eq:multiple-hull}
	\cH_n := \hull \bigl\{ S^{(k)}_{i} : 1 \leq k \leq N, \, 0 \leq i \leq n \bigr\},
\end{equation}
and its diameter and perimeter
\begin{equation}
	\label{eq:multiple-hull-diam-perim}
	D_n := \diam \cH_n,\quad \text{and}\quad  L_n := \perim \cH_n.
\end{equation}
The polygon~$\cC_\bmu$ defined in~\eqref{ass:many-walks} 
governs
the first-order shape of $\cH_n$ through the SLLN 
\[ \lim_{n \to \infty} \rhoH (  n^{-1} \cH_n , \cC_\bmu ) = 0, \as, \]
where $\rhoH$ is the Hausdorff metric on compact, convex subsets of $\R^2$ containing~$\0$ (see~Theorem~2.1 of~\cite{ikss} for~$N=2$ and~Theorem~2.2.1 of~\cite{tomislav} for the general case). 
Hence follow strong laws 
$\lim_{n \to \infty} n^{-1} D_n = \diam  \cC_\bmu $ and
$\lim_{n \to \infty} n^{-1} L_n = \perim \cC_\bmu $, a.s. 
Fluctuations are now more complicated than for a single walk, and this is our interest in the present paper.

The two-walk case has its own useful geometry and was the starting point for several of the questions addressed here. We return to it in \S\ref{sec:two-walk-intro}, after summarizing the main results.

\subsection{Overview of  results}
\label{sec:results}

We summarize the main contributions of the paper and indicate some of the ideas underlying the proofs.

\subsubsection*{$L^2$ approximation and the Wald CLT} We present an $L^2$-approximation framework in which the existing $N=1$ CLT results~\eqref{eq:known-D-L-approx} for $L_n$ and $D_n$ are interpreted as higher-dimensional analogues of \emph{Wald's maximal CLT}~\cite{wald1947}; see \S\ref{sec:wald}--\S\ref{sec:diam-one-walk}. This framework is then extended to convex hulls generated by multiple random walks.
	The starting point is the one-dimensional Wald-type estimate in \S\ref{sec:wald}: for a random walk with positive drift, the running maximum differs from the endpoint by $o(\sqrt n)$ in $L^2$, while for negative drift the maximum itself is $o(\sqrt n)$. This is the probabilistic input behind
    later reductions where we approximate the convex hull by the convex hull of a smaller set in which certain random walks can be replaced by their endpoints. A first geometrical application of this idea is to the diameter of one random walk, in~\S\ref{sec:diam-one-walk}.
	
\subsubsection*{Set approximation for the convex hull.} As part of the approximation framework, we exhibit  approximations of the convex hulls themselves, in appropriate metrics. In particular, Theorem~\ref{thm:general-hausdorff} gives an essentially optimal $L^2$ approximation for $\cH_n$ in Hausdorff metric; see \S\ref{sec:hausdorff}.
This result also shows why the Hausdorff metric is sometimes too strong for perimeter-type approximations: transverse excursions can survive on the $\sqrt n$ scale even when they do not affect the perimeter or diameter at that order, as explained below~\eqref{eq:line-segment-limits}.
  We show that   an alternative metric on convex compact sets (the $\rho_1$ metric described in \S\ref{sec:perimeter_deviation_metric} below) gives a finer $L^2$ set approximation result which is optimal from the point of view of the perimeter: this is Theorem~\ref{thm:perimeter-deviation-metric} below.

\subsubsection*{Approximation and limit theory for the diameter} For the diameter, our main general approximation result is Theorem~\ref{thm:diameter-approximation-general}. It gives an essentially complete $L^2$ approximation for all $N \geq 1$, from which are derived the corresponding distributional limits, of the form of~\eqref{eq:diameter-limit}, given in  Theorem~\ref{thm:many-walks-diameter} below.  The limits are Gaussian for generic configurations of the drift polygon, but max-type non-Gaussian limits appear whenever more than one random walk, or any zero-drift random walk, contributes to the limiting fluctuations. The approximation result, extending Wald's one-dimensional limit theorem, shows that for the diameter one can reduce the complexity of the convex hull, retaining only endpoints of walks corresponding to diametrical chords of $\cC_\bmu$, and trajectories of zero-drift walks if the origin contributed to a diametrical chord. The resulting reduced object is then amenable to fluctuation analysis by classical limit theorems.

	
\subsubsection*{Approximation and limit theory for the perimeter} For the perimeter, the main contribution is Theorem~\ref{thm:perim_main_theorem}, which  gives a general distributional limit for $L_n$ for arbitrary $N \geq 1$ of the form~\eqref{eq:perimeter-limit}.
The proof uses the Cauchy  formula of integral geometry to express perimeter
as an integral over maxima of one-dimensional projected random walks, and then again uses $L^2$ approximation ideas related to the Wald theorem (replacing maxima over random walk trajectories by  endpoints) to make a series of complexity reductions. It is shown that the reduced convex hull needs to retain only the endpoints of random walks corresponding to extreme points of $\cC_\bmu$, as well as trajectories of zero-drift random walks if $\0$ is an extreme point of $\cC_\bmu$. 
The   proof of Theorem~\ref{thm:perim_main_theorem} in \S\ref{sec:perim-proofs} is the longest part of the paper, 
organized as a series of approximation steps, and considerable work is needed to cast the limit objects in an accessible form.
    
	
\subsubsection*{Partial Cauchy formulae} The perimeter limit 
in~\eqref{eq:perim-main-limit} has a geometric interpretation, given in Appendix~\ref{sec:semi-cauchy}. There we show that the support-function integrals appearing in $\Pi^+ + \Pi^0$ in~\eqref{eq:perimeter-limit} can be read as partial Cauchy formulae over the normal cones of the vertices of the drift polygon. Thus each term in the limit corresponds to a boundary arc of the appropriate Gaussian or Brownian convex body, with terms also coming from the endpoints of the arc.
	
\subsubsection*{Examples for two walks.} 
Part of our original motivation was to settle the case $N=2$ for both diameter and perimeter. As mentioned, for generic cases Gaussian limits were known~\cite{ikss,tomislav}. Our main limit theorems (Theorem~\ref{thm:many-walks-diameter} and Theorem~\ref{thm:perim_main_theorem}) contain these known results, but also establish non-Gaussian limit theorems for the cases not covered in~\cite{ikss,tomislav}. The $N=2$ examples give a flavour of the behaviour for general $N \geq 2$, and we present them in some detail in~\S\ref{sec:max-Gaussian-examples} (for diameter) and \S\ref{subsec:perim-applic} (for perimeter). An explanation of how these results complement the existing literature is in \S\ref{sec:two-walk-intro}, and the two-walk picture is also summarized in Table~\ref{table1}.

	
\subsubsection*{Variance asymptotics.} 
Our distributional convergence results are accompanied by variance asymptotics, which either follow from the $L^2$ approximations that we develop or by generic  uniform square-integrability estimates.

\begin{remarks}
\label{rems:intro}
\begin{myenumi}[label=(\alph*)]
\item\label{rems:intro-a}
	Throughout, we work under the minimal moment hypotheses on the increments for which the main results can hold, namely finiteness of second moments, as in~\eqref{ass:many-walks}. Stronger assumptions, such as a uniform $p$th moment bound for some $p>2$, would shorten some technical arguments but would forfeit optimality of the hypotheses.
    \item\label{rems:intro-b}
     We restrict attention to the planar case, where 
       tools such as Cauchy's perimeter formula are available and give a particularly clean description of the perimeter fluctuations. Extension to higher dimensions is of interest, but would demand additional technical work, especially for the perimeter (surface measure, in higher dimensions).
	 \item\label{rems:intro-c}
    Although we assume the $N$ random walks are \emph{independent},
	the methods of this paper would admit, with appropriate modifications, relaxing this to an assumption that the joint increments $(Z_{1,i}, \ldots,
	Z_{N,i})_{i \in \N}$ are i.i.d.~$\R^{2N}$-valued random vectors, permitting dependence among different walks. 
	Our $L^2$ approximation framework extends, and resulting distributional limit results would be modified, e.g., with correlations in Gaussian or Brownian limits for individual walks.
    \end{myenumi}
\end{remarks}

The starting point for the framework sketched above goes back to one dimensional random walk. The Wald CLT says that for a one-dimensional random walk whose increments have positive mean and finite variance, the running maximum satisfies the same CLT as does the endpoint of the walk. We will find it useful to formulate an $L^2$ approximation of the type represented in~\eqref{eq:known-D-L-approx} above from which the Wald result can be easily deduced: this is the aim of \S\ref{sec:wald}.  This is also the simplest setting in which to present some proof elements that 
will be repeated in more involved settings later in the paper. First, in~\S\ref{sec:two-walk-intro}, we spell out how our results in the case $N=2$ complete the picture in the literature, which previously considered only the cases corresponding to Gaussian limits. 

\subsection{The case of two random walks}
\label{sec:two-walk-intro}

Our general results cover $N$ multiple random walks for any $N \geq 2$.
Since part of our original motivation was to complete the picture for even $N=2$ walks, we describe that case in more detail here. 
The first work on fluctuations for $D_n$ and $L_n$ given by~\eqref{eq:multiple-hull-diam-perim}
for the convex hull~\eqref{eq:multiple-hull} of $N=2$ random walks was~\cite{ikss} (see also~\cite{tomislav}).
In~\cite{ikss} it was shown 
that similar $L^2$ approximations to~\eqref{eq:known-D-L-approx}, and hence CLTs,
hold for most cases of the relative orientation and magnitudes of the drifts~$\bmu_1$, $\bmu_2$ of the two random walks. The cases omitted from~\cite{ikss} for the perimeter are a subset of those omitted for the diameter, and all are special for one or other form of degeneracy, for example, because one of the drifts is zero, or both drifts are the same, or at least of the same magnitude. The two $N=2$ examples in   Figure~\ref{fig0} are representative configurations of the exceptional types.

A consequence of our general results is to settle the remaining two-walk cases, which all turn out to have \emph{non-Gaussian} limits, complementing the Gaussian cases given in~\cite{ikss}. Define $\cD (\bmu_1, \bmu_2 ) := \{  \| \bmu_1\| , \| \bmu_2 \|, \| \bmu_1 - \bmu_2 \| \}$, the set of lengths of edges of the (possibly degenerate) triangle $\cC_\bmu$ generated by the two drift vectors. In~\cite{ikss}, Gaussian limits were obtained for the diameter and perimeter in most of the $\bmu_1, \bmu_2$ parameter space, as described by the following two conditions on the triangle of drifts $\cC_\bmu$.

\begin{description}[topsep=3pt,itemsep=3pt]
	\item
	[\namedlabel{ass:no-zero}{${\Delta}_{0}$}]
	We have $0 \notin \cD (\bmu_1, \bmu_2 )$.
	\item
	[\namedlabel{ass:unique-max}{${\Delta}_{1}$}]
	There is a unique maximal edge for $\cC_\bmu$.
\end{description}

Note that (i) \eqref{ass:unique-max} implies~\eqref{ass:no-zero}; and (ii) sufficient for~\eqref{ass:no-zero} is that the triangle $\cC_\bmu$ has positive two-dimensional area, but~\eqref{ass:no-zero} permits ``degenerate'' triangles in which $0 < \bmu_1 =\lambda \bmu_2$ for some $\lambda \in \R$, $\lambda \neq 1$.

\begin{table}
	\setcellgapes{7pt}
	\makegapedcells
	\small
	\scalebox{0.89}{
		\begin{tabularx}{1.08\textwidth}{|l|X|X|c|c|}
			\hline
			{\bf\normalsize Conditions on $\bmu_1, \bmu_2$} 
			& {\bf\normalsize Diameter limit} 
			& {\bf\normalsize Perimeter limit} 
			& \eqref{ass:no-zero} 
			& \eqref{ass:unique-max}  \\
			\hline
			
			\eqref{ass:no-zero} \& \eqref{ass:unique-max}
			& G: see~\cite{ikss}; also  Theorem~\ref{thm:many-walks-diameter}
			& G: see~\cite{ikss}; also  Theorem~\ref{thm:perim_main_theorem}
			& \cmark
			& \cmark  \\
			\hline
			
			$\bmu_1 = \bmu_2 = \0$
			& nG: BF($4$); see Appendix~\ref{sec:donsker}
			& nG: BF($4$); see Appendix~\ref{sec:donsker}
			& \xmark
			& \xmark   \\
			\hline
			
			$\| \bmu_1 \| >0 = \| \bmu_2 \|$
			& nG: G\,$+$\,BF($1$); see Corollary~\ref{cor:diam_two_walks}
			& nG: G\,$+$\,BF($2$); see Theorem~\ref{thm:perim_main_theorem}
			& \xmark
			& \xmark \\
			\hline
			
			$\| \bmu_1 \| = \| \bmu_2 \| > \| \bmu_1-\bmu_2 \| >0$
			& nG: G\,$\vee$\,G; see Corollary~\ref{cor:diam_two_walks}
			& G: see~\cite{ikss}; also Theorem~\ref{thm:perim_main_theorem}
			& \cmark
			& \xmark
			\\
			\hline
			
			$\bmu_1=\bmu_2\neq \0$
			& nG: G\,$\vee$\,G; see Corollary~\ref{cor:diam_two_walks}
			& nG: see Example~\ref{ex:perim_two-the-same}
			& \xmark
			& \xmark
			\\
			\hline
			
			$\| \bmu_1 \| = \| \bmu_1-\bmu_2 \| > \| \bmu_2 \| >0$
			& nG: G\,$\vee$\,G (cor.); see Corollary~\ref{cor:diam_two_walks}
			& G:  see~\cite{ikss}; also  Theorem~\ref{thm:perim_main_theorem}
			& \cmark
			& \xmark \\
			\hline
			
			$\| \bmu_1 \| = \| \bmu_2 \| = \| \bmu_1-\bmu_2 \| >0$
			& nG: G\,$\vee$\,G\,$\vee$\,G (cor.); see Corollary~\ref{cor:diam_two_walks}
			& G:  see~\cite{ikss}; also Theorem~\ref{thm:perim_main_theorem}
			& \cmark
			& \xmark \\
			\hline
	\end{tabularx}}
\smallskip
	\caption{\label{table1}
		Schematic of results for diameter and perimeter of two random walks. ``G'' stands for Gaussian, ``nG'' for non-Gaussian.
		The table describes loosely  each non-Gaussian limit,
		possibly as a sum (``$+$'') or maximum (``$\vee$'') of more elementary random variables, where  components, in increasing order of complexity,
		are ``G'' (Gaussian, again), ``BF($m$)'' (some elementary geometrical functional
		of an $m$-dimensional Brownian motion).
		Components in a sum or maximum are independent, unless denoted ``cor.'' (for ``correlated''). 
		Each abbreviation may represent a different object of the class
		it describes on every appearance in the table; for example,
		in the  row of results for case $\bmu_1 =\bmu_2 =\0$, which are straightforward applications of Donsker's
		theorem recorded in Appendix~\ref{sec:donsker}, the first ``BF($4$)'' 
		is a diameter of the convex hull of two independent planar Brownian motions,
		while the second ``BF(4)'' is a perimeter of the same type.}
\end{table}

Table~\ref{table1} gives a schematic summary of the two-walk picture. In particular, the diameter results are organized by which edge or edges of the drift triangle can realize the maximal length, whereas the perimeter results are governed by the support directions of the drift polygon.

Write $S^{(1,2)}_n := S^{(1)}_{n} - S^{(2)}_{n}$, $\bmu_{1,2} :=  \bmu_1 - \bmu_2$ and $\hbmu_{1,2} = \bmu_{1,2} / \| \bmu_{1,2} \|$ whenever $\bmu_{1,2} \neq \0$.
It is shown in~\cite{ikss} that, for the perimeter,
provided~\eqref{ass:no-zero} holds,
\begin{equation}
	\label{eq:ikss-perimeter}
	\frac{ L_n - ( \hbmu_1^\tra S^{(1)}_{n}   + \hbmu_2^\tra S^{(2)}_{n}   +  \hbmu_{1,2}^\tra S^{(1,2)}_n  )  }{\sqrt{n}} \toL{n}{2} 0,
\end{equation}
and, for the diameter,
provided~\eqref{ass:unique-max} holds,
\begin{equation}
	\label{eq:ikss-diameter}
	\frac{ D_n -   \hbmu^\tra_\star S^{(\star)}_{n}   }{\sqrt{n}} \toL{n}{2} 0,
\end{equation}
where $\hbmu_\star$ is the unit vector corresponding to whichever of $\bmu_1$, $\bmu_2$, $\bmu_{1,2}$ provides the unique maximum
side of $\cC_\bmu$ as specified by~\eqref{ass:unique-max}, and $S^{(\star)}_{n}$ the corresponding walk.

\begin{example}[Perimeter, collinear drifts]
	\label{ex:colinear}
	Suppose that $\0 \neq \bmu_1 =\lambda \bmu_2$ for some $\lambda \in \R$. Then~\eqref{ass:no-zero} holds provided $\lambda \neq 1$, and it follows from~\eqref{eq:ikss-perimeter} that
	\[ \frac{ L_n - 2\hbmu_\star^\tra S^{(\star)}_{n}     }{\sqrt{n}} \toL{n}{2} 0,   \]
	where $\star$ indicates $1$ for $\lambda >1$,
	$2$ for $\lambda \in (0,1)$, and $1,2$ for $\lambda <0$.  
	Consequently, for $\lambda \neq 1$, $n^{-1/2} ( L_n - \Exp [L_n] )$
	converges to a Gaussian limit. 
	We show  that when $\lambda =1$ the CLT fails and the limit is a non-Gaussian, see Example~\ref{ex:perim_two-the-same} below.
\end{example}


\section{One-dimensional results related to the Wald limit theorem}
\label{sec:wald}

Let $X, X_1, X_2, \ldots$ be i.i.d.\ $\R$-valued
random variables with $\Exp [ X^2 ]< \infty$; denote their mean by $\mu := \Exp [X ] \in \R$ and variance by $\sigma^2 := \Var (X) \in \RP := [0, \infty)$. Define the
associated random walk\footnote{We alert the reader that throughout \S\ref{sec:wald} we use $S$ to denote a random walk on $\R$, unlike in the bulk of the paper, where (as in \S\ref{sec:convex-hulls-background}) it denotes a random walk on $\R^2$.} $S = (S_n)_{n \in \ZP}$ by $S_0 :=0$ and
$S_n := \sum_{k=1}^n X_k$ for $n \in \N$. 
We use the notation 
\begin{equation}
    \label{eq:max-min}
M_n := \max_{0 \le k \le n} S_k, \quad \text{and}\quad m_n := \min_{0 \le k \le n} S_k,
\end{equation}
for the associated partial maxima and minima.

A classical maximal CLT, due to Wald~\cite{wald1947}, says that, provided $\mu >0$, 
\begin{equation}
\label{eq:wald-clt-classic}
     \frac{M_n - n\mu }{\sqrt{n}} \tod{n} \cN (0, \sigma^2),
\end{equation}
where $\cN (0, \sigma^2)$ is the normal distribution with mean~$0$
and variance~$\sigma^2$; see~\cite[\S2.12, \S4.4]{gut-stopped} for a modern presentation.
Note that~\eqref{eq:wald-clt-classic} is precisely the same CLT that~$S$ itself satisfies, and that $\mu >0$ is necessary. 
The focus of this section is the following $L^2$-approximation
result between $M_n$ and $S_n$, which complements the Wald CLT: indeed~\eqref{eq:wald-clt-classic}
follows directly from part~\ref{thm:WaldCLT-ii} of the result.
We will see later that Theorem~\ref{thm:WaldCLT} serves as a one-dimensional prototype
for  approximations for convex-hull statistics  that are our main objects of interest.

\begin{theorem}
    \label{thm:WaldCLT}
    Suppose that $\Exp [ X^2] < \infty$ and $\mu >0$. 
The following statements hold.
\begin{thmenumi}[label=(\roman*)]
\item 
    \label{thm:WaldCLT-i}
    $\displaystyle n^{-1/2} m_n \toL{n}{2} 0$.
\item 
    \label{thm:WaldCLT-ii}
    $\displaystyle n^{-1/2} (M_n - S_n) \toL{n}{2} 0$. 
\end{thmenumi}
\end{theorem}
\begin{remarks}
    \label{rem:waldCLT}
    \begin{myenumi}[label=(\alph*)]
        \item  Although we did not find  
Theorem~\ref{thm:WaldCLT} stated explicitly in the literature,
it follows directly from combining some standard ingredients,
that one can find in the excellent book~\cite{gut-stopped}, for example.
We  give a derivation here because, firstly,
the   strategy is mirrored in some of the
subsequent, more complicated arguments, and, secondly,
we also wish to state a related result in the case~$\mu =0$ that will be essential in \S\ref{sec:diam-one-walk} below. This auxiliary uniform integrability result is Lemma~\ref{lem:max-mean-ui} below (of course, Theorem~\ref{thm:WaldCLT} itself does not hold when $\mu =0$, since $n^{-1/2} m_n$ and $n^{-1/2} (M_n - S_n)$ have non-trivial distributional limits).
\item Note that  Theorem~\ref{thm:WaldCLT}\ref{thm:WaldCLT-i} and Theorem~\ref{thm:WaldCLT}\ref{thm:WaldCLT-ii} are equivalent statements,
as follows from classical time-reversal duality (see \cite[Lemma 11.2]{gut-stopped}):
\begin{equation}
    \label{eq:duality}
(M_n, M_n - S_n) \eqd (S_n - m_n, -m_n),\quad \text{for every } n \in \ZP.
\end{equation}
        \item A useful $L^1$ relative of Theorem~\ref{thm:WaldCLT} that we will also need later, and which can be found in Theorem~4.4.8 of~\cite[p.~141]{gut-stopped}, says that, if 
$\Exp [ X^2] < \infty$ and $\mu >0$,
\begin{equation}
    \label{eq:L1-gut}
    \sup_{n \in \ZP} \Exp [| M_n - S_n |] = \sup_{n \in \N} \Exp [| m_n |] < \infty.
\end{equation}
    \end{myenumi}
\end{remarks}

Our proof of Theorem~\ref{thm:WaldCLT} uses the following fact about uniform square-integrability, which we prove later in this section.

\begin{lemma}
\label{lem:max-sum-ui}
Suppose that $\Exp [ X^2] < \infty$ and $\mu \geq 0$. Then
 both $(n^{-1} ( M_n - S_n)^2)_{n\in\N}$ and $(n^{-1} m_n^2)_{n\in\N}$  are uniformly integrable.
\end{lemma}

\begin{proof}[Proof of Theorem~\ref{thm:WaldCLT}]
As explained above, it suffices to prove
Theorem~\ref{thm:WaldCLT}\ref{thm:WaldCLT-i}, since time-reversal duality~\eqref{eq:duality} then yields part~\ref{thm:WaldCLT-ii}. If $\mu >0$, then $\lim_{n \to \infty} m_n = m_\infty \in (-\infty, 0]$ exists, a.s., by the SLLN, and so $\lim_{n\to\infty} n^{-1} m_n^2 = 0$, a.s.
Then uniform integrability (Lemma~\ref{lem:max-sum-ui}) shows that $\lim_{n \to \infty} n^{-1} \Exp [m_n^2] =0 $, as required.
\end{proof}

Regardless of the sign of~$\mu \in \R$, it is well known that 
when $\Exp [X^2] < \infty$ it holds that
    \begin{equation}
        \label{eq:sum-mean-ui}
     (n^{-1} (S_n - n\mu)^2)_{n\in\N} \text{ is uniformly integrable}; \end{equation}
 see e.g.~\cite[p.~20]{gut-stopped}, \cite[p.~176]{billingsley}.
Then, by  duality~\eqref{eq:duality}
and the fact that
$n^{-1} (M_n - S_n)^2 \leq 2 n^{-1} (M_n - n \mu)^2 + 2 n^{-1} (S_n - n\mu)^2$,
to establish Lemma~\ref{lem:max-sum-ui}, it suffices to prove the following.

\begin{lemma}
\label{lem:max-mean-ui}
If $\Exp [ X^2] < \infty$ and $\mu \geq 0$, 
  $(n^{-1} ( M_n - n \mu)^2)_{n\in\N}$ is uniformly integrable.
\end{lemma}

In the literature, we were able to find Lemma~\ref{lem:max-mean-ui}
only for $\mu > 0$, as the $r=2$ case of 
Theorem~4.4.3(ii) in~\cite[p.~139]{gut-stopped}
(the hypothesis $\mu>0$ applies throughout~\S4.4 of~\cite{gut-stopped}, 
where a wealth of  neighbouring results can be found). 
In the case $\mu \geq 0$, writing $\widetilde{S}_k := S_k - k \mu$,  
\[ S_n - n \mu \leq M_n - n\mu = \max_{0 \leq k \leq n} \bigl( \widetilde{S}_k + k \mu \bigr) - n \mu \leq  \max_{0 \leq k \leq n}   \widetilde{S}_k .\]
Thus from~\eqref{eq:sum-mean-ui}, we see that the general $\mu \geq 0$ case of Lemma~\ref{lem:max-mean-ui} follows from the case $\mu =0$. The latter being the case we did not find elsewhere, and important in \S\ref{sec:diam-one-walk},
we give the proof here.
An important ingredient is the following maximal inequality for zero-mean, finite-variance random walks, which we obtain  via Etemadi's inequality and  symmetrization.

\begin{lemma}
    \label{lem:etemadi-plus}
      Suppose that $\Exp [ X^2] < \infty$ and $\mu = 0$.
      Then there exists $C \in \RP$ such that, 
      \begin{equation}
          \label{eq:max-sum-inequality}
\Pr ( |M_n| > 12t ) 
\leq 48 \Pr ( |S_n|  \geq t ),\quad \text{for all } t > C \text{ and all } n \in \ZP.
      \end{equation}
\end{lemma}
\begin{proof}
First we recall  Etemadi's inequality (see e.g.~\cite[p.~256]{billingsley})
which says that
\begin{equation}
    \label{eq:etemadi}
    \Pr \left( |M_n| \geq  x \right) \leq 3 \max_{0 \leq k \leq n} \Pr \left( | S_k | \geq x / 3 \right), \quad \text{for all } x \geq 0.
\end{equation}
We thus need to estimate the right-hand side of the inequality in \eqref{eq:etemadi}.  
For a random variable $X$, let $X'$ denote an independent copy of $X$ and write $X^\star := X- X'$ for the symmetrization of $X$. Then~\cite[p.~149]{feller2}   $\Pr ( |X^\star| \geq 2t ) \leq 2 \Pr (|X| > t)$
and, if $a \in \RP$ is such that $\Pr ( X \leq a ) \geq 1/4$ and $\Pr ( X \geq -a ) \geq 1/4$, then $\Pr ( |X| > t +a) \leq 4 \Pr ( |X^\star| \geq t )$. 

The CLT shows  $\lim_{n \to \infty} \Pr ( S_n \geq 0) = \lim_{n \to \infty} \Pr ( S_n \leq 0) = 1/2$, so there exists $n_0 \in \N$ such that both $\Pr (S_n \leq 0 ) \geq 1/4$ and $\Pr ( S_n \geq 0) \geq 1/4$ for all $n \geq n_0$.
It follows that there exists $0<C < \infty$ such that $\inf_{n \in \ZP} \Pr ( S_n \leq C  ) \geq 1/4$
and $\inf_{n \in \ZP} \Pr (S_n \geq - C ) \geq 1/4$.
Hence
\begin{align*} 
\max_{0 \leq k \leq n}
\Pr ( |S_k| > t + C   ) 
& \leq
\max_{0 \leq k \leq n}
 4 \Pr ( |S_k^\star | \geq t ). \end{align*}
 By an inequality of L\'evy (Corollary~5 of~\cite[p.~72]{chow-teicher}) 
 it holds that $\Pr ( | M_n^\star| \geq t ) \leq 2 \Pr ( |S_n^\star | \geq t)$, and, since $\Pr ( |M_n^\star| \geq x ) = \Pr ( \cup_{0 \leq k \leq n} \{ | S^\star_k | \geq x \} ) \geq \max_{0 \leq k \leq n} \Pr (|S^\star_k| \geq x)$, we get
 $\max_{0 \leq k \leq n}
\Pr ( |S_k| > t + C   ) 
 \leq 8 \Pr ( |S_n^\star| \geq t )$. Thus,
 for all $t > 2 C $, 
\begin{align*} 
\max_{0 \leq k \leq n}
\Pr ( |S_k| > t   ) \leq 8   \Pr ( |S_n^\star| \geq t -  C  ) \leq 16 \Pr ( |S_n| > \tfrac{t - C}{2}  ) \leq 16 \Pr ( |S_n| > \tfrac{t}{4} ) .\end{align*} 
It follows that, on re-defining the constant~$C$, $\max_{0 \leq k \leq n}
\Pr ( |S_k| > 4t ) 
\leq 16 \Pr ( |S_n|  \geq t )$, for all $t > C$.
 Combined with~\eqref{eq:etemadi} we obtain~\eqref{eq:max-sum-inequality}.
 \end{proof}

Now we can complete the proof of Lemma~\ref{lem:max-mean-ui}. Set  $x^+ := x\1{x>0}$.

\begin{proof}[Proof of Lemma~\ref{lem:max-mean-ui}]
As explained, it suffices to suppose that $\mu =0$.
Then~\eqref{eq:sum-mean-ui} says that $(n^{-1} S_n^2 )_{n \in \N}$ is uniformly integrable. Recall that, for  $\zeta \geq 0$ and every  $z \geq 0$,
\begin{align}
\label{eq:by-parts}
\Exp [ \zeta \ind{ \zeta \geq z} ] - z \Pr ( \zeta \geq z) & = \Exp [ ( \zeta - z )^+ ] 
\nonumber\\
&  = \int_0^\infty \Pr ( ( \zeta - z )^+ \geq y ) \ud y   = \int_{z} ^\infty \Pr (  \zeta \geq y  ) \ud y .
\end{align}
In particular, taking $\zeta = M_n^2$ and $z=Kn$ for $K >0$ in~\eqref{eq:by-parts}, we get
\[ \Exp \left[ M_n^2 \ind { M_n^2 \geq K n} \right]
= \int_{K n}^\infty \Pr (  M_n^2 \geq y  ) \ud y 
+ K n \Pr ( M_n^2 \geq K n ) .
\]
Hence, by two applications of Lemma~\ref{lem:etemadi-plus} (taking $K > C$),
\begin{align*}
\Exp \left[ M_n^2 \ind { M_n^2 \geq K n} \right]
& \leq 48  \int_{K n}^\infty \Pr (  S_n^2 \geq y /12  ) \ud y 
+ 48 K n   \Pr ( S_n^2 \geq K n /12 ) \\
& \leq 576 \Exp \left[ S_n^2 \ind { S_n^2 \geq K n / 12} \right] + 48 K n   \Pr ( S_n^2 \geq K n /12),
\end{align*}
using~\eqref{eq:by-parts} once more. Hence
\begin{align*}
  n^{-1} \Exp \left[ M_n^2 \ind { M_n^2 \geq K n} \right]
&  \leq 576   n^{-1} \Exp \left[ S_n^2 \ind { S_n^2 \geq K n / 12} \right] + 48 K   \Pr ( S_n^2 \geq K n /12 ).
\end{align*}
Here, by the argument for Markov's inequality, 
\[  48 K \Pr ( S_n^2 \geq K n /12 )
\leq 576  n^{-1} \Exp [ S_n^2  \ind { S_n^2 \geq K n / 12} ].
\]
So we conclude that
\begin{align*}
\sup_{n \in \ZP} n^{-1} \Exp \left[ M_n^2 \ind { M_n^2 \geq K n} \right]
&  \leq 1152 \sup_{n \in \ZP} n^{-1} \Exp \left[ S_n^2 \ind { S_n^2 \geq K n / 12} \right] ,
\end{align*}
which tends to~$0$ as $K \to \infty$, by uniform integrability of $(n^{-1} S_n^2)_{n\in\N}$. 
\end{proof}

\section{Diameter for one walk: A geometrical approach}
\label{sec:diam-one-walk}

\subsection{Notation and main result}
\label{sec:diam-one-walk-result}
This section presents a short geometrical proof of  (a slight extension of) 
the diameter result from~\cite{mcrw}, quoted in~\eqref{eq:known-D-L-approx} above,
using the one-dimensional Wald-type results from~\S\ref{sec:wald}. 
This serves as an ingredient
for the case of multiple walks that we approach by a related
method in~\S\ref{sec:wald-based} and \S\ref{sec:max-Gaussian}.
As in \S\ref{sec:convex-hulls-background}, let $Z, Z_1, Z_2, \ldots$ be i.i.d.~$\R^2$-valued random variables with $\Exp [ \| Z \|^2 ] < \infty$
and $\bmu := \Exp[ Z]  \neq \0$. Denote the associated random walk $S = (S_n)_{n \in \ZP}$
by $S_0 = \0$ and $S_n = \sum_{i=1}^n Z_i$, $n \in \N$, and 
its diameter $D_n := \diam \{ S_0, S_1, \ldots, S_n \}$. Recall also that $\hbmu = \bmu / \| \bmu \|$.
We give a geometrical proof of the following result.

\begin{theorem}
\label{thm:one-walk-diameter}
    Suppose that $\Exp [ \| Z \|^2 ] < \infty$
and $\bmu   \neq \0$. Then
\[ \frac{D_n - \hbmu^\tra S_n }{\sqrt{n}} \toL{n}{2} 0,\quad  \text{and}\quad \sup_{n \in \ZP} \Exp \left[| D_n - \hbmu^\tra S_n|\right] < \infty.\]
Moreover,  as $n \to \infty$, $\Exp [D_n] = n \| \bmu \| + O(1)$ and $n^{-1} \Var ( D_n ) \to \spara{}$ 
defined at~\eqref{eq:spara-def}.
\end{theorem}

The $L^2$ result in Theorem~\ref{thm:one-walk-diameter}
was obtained in~\cite{mcrw} via a martingale-difference argument. 
The shorter proof that we give here makes transparent the 
link to the $L^2$ version of the  Wald CLT (Theorem~\ref{thm:WaldCLT})
via a one-dimensional approximation, and will serve as a warm-up for~\S\ref{sec:max-Gaussian} and \S\ref{sec:wald-based}. The $L^1$ bound in Theorem~\ref{thm:one-walk-diameter} is new (but straightforward).

\subsection{Proofs}
\label{sec:diam-one-walk-proof}

To formulate the approximation result to an appropriate one-dimensional process,
let
$\hbmuperp$ denote a unit vector orthogonal to $\hbmu$,
and define 
$S^\bmu_n := \hbmu^\tra S_n$ and
$S^\per_n := \hbmuperp^\tra S_n$ 
for the projections of the walk in the $\hbmu, \hbmuperp$ directions, respectively.
Also set  
$M^\bmu_n := \max_{0 \leq k \leq n} S^\bmu_k$ and $m^\bmu_n := \min_{0 \leq k \leq n} S^\bmu_k$.
Then set $D_n^{\bmu} := M^\bmu_n - m^\bmu_n$.

\begin{proposition}
\label{prop:two-to-one-diamater}
    Suppose that $\Exp [ \| Z \|^2 ] < \infty$
and $\bmu   \neq \0$. Then
\begin{equation}
    \label{eq:two-to-one-diamater-L2}
     \frac{ D_n - D_n^{\bmu} }{\sqrt{n}} \toL{n}{2} 0 , \text{ and } \sup_{n \in \ZP}   \Exp \left[| D_n - D_n^{\bmu} |\right] < \infty.
\end{equation}  
%
\end{proposition}

Before giving the proof of Proposition~\ref{prop:two-to-one-diamater}, we give one simpler result of a similar flavour, which 
is needed in the derivation of~\eqref{eq:line-segment-limits} from the Introduction.
The proof of $L^2$ convergence is similar to that in~\S\ref{sec:wald}: combine uniform square-integrability with convergence in probability.

\begin{lemma}
\label{lem:norm-projection}    
 Suppose that $\Exp [ \| Z \|^2 ] < \infty$
and $\bmu   \neq \0$. Then
\begin{equation}
    \label{eq:norm-projection}
     \frac{ \| S_n \| - \hbmu^\tra S_n }{\sqrt{n}} \toL{n}{2} 0 .
\end{equation}  
\end{lemma}
\begin{proof}
    Let $\hbmuperp^\tra$ denote a unit vector orthogonal to~$\hbmu$. 
Then we can write
    \[ \| S_n\| - |\hbmu^\tra S_n| = \frac{n^{-1}  | \hbmuperp^\tra S_n |^2}{n^{-1} \| S_n\| + n^{-1} |\hbmu^\tra S_n |},
    \]
    where the SLLN shows that both $n^{-1} \| S_n\|$ and $n^{-1} |\hbmu^\tra S_n |$ tend to~$\| \bmu\|$ as $n \to \infty$, a.s., while the CLT shows that $(n^{-1}  | \hbmuperp^\tra S_n |^2)_{n\in\N}$ is tight. Hence
   \begin{equation}
    \label{eq:norm-projection-prob} \frac{ \| S_n\| - |\hbmu^\tra S_n| }{\sqrt{n}} \toP{n} 0. \end{equation}
Moreover,    since $0 \leq \| S_n\| -  |\hbmu^\tra S_n| \leq | \hbmuperp^\tra S_n |$, and $(n^{-1}  | \hbmuperp^\tra S_n |^2)_{n\in\N}$ is uniformly integrable~\cite[p.~32]{gut-stopped}, it follows that~\eqref{eq:norm-projection-prob} extends to~$L^2$ convergence. Finally, note that $0 \leq |\hbmu^\tra S_n| - \hbmu^\tra S_n \leq  |\hbmu^\tra S_n| \ind {  \hbmu^\tra S_n < 0}$, which tend to~$0$ in probability by the SLLN, and so another uniform integrability argument yields~\eqref{eq:norm-projection}.
\end{proof}

Now we turn to the proof of Proposition~\ref{prop:two-to-one-diamater}, which, compared to Lemma~\ref{lem:norm-projection}, requires control of the full trajectories rather than just the endpoints, and so goes along the lines of the Wald-type Theorem~\ref{thm:WaldCLT}. Write $\oMperp_n := \max_{0 \leq k \leq n} | S^{\per}_k |$.
We will use the geometrical bounds
\begin{equation}
\label{eq:box-bound}
D_n^{\bmu} \leq D_n \leq \sqrt{ ( { D_n^\bmu } )^2 + 4 ( \oMperp_n )^2 } ,\end{equation}
where the upper bound comes from the diagonal of a rectangle (see Figure~\ref{fig1}).

\begin{figure}
    \includegraphics[width=0.2\textwidth,angle=270]{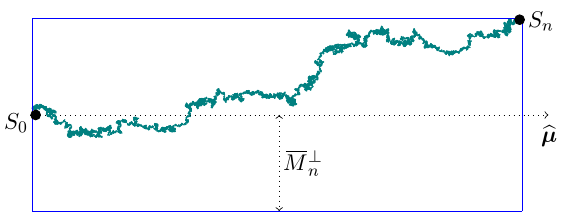}
    \caption{Bounding the diameter of one random walk with non-zero drift by the diagonal of a bounding rectangle. The length of the box (in the drift direction) is order $n$ while the width is about $\sqrt{n}$, so the triangle inequality is a poor bound. Similar geometrical bounds are used for $N \geq 2$ walks in \S\ref{sec:max-Gaussian} and \S\ref{sec:wald-based} below.}
    \label{fig1}
\end{figure}

\begin{remark}
\label{rem:two-to-one-diamater}
If we apply the triangle inequality, we get
\begin{equation}
\label{eq:triangle-bound}
D_n^\bmu \leq 
D_n \leq   D_n^\bmu  + 2  \oMperp_n .
\end{equation}
A consequence of Doob's $L^2$ inequality~\cite[p.~505]{gut-graduate} (or of Lemma~\ref{lem:max-mean-ui} above) is  
\begin{equation}
    \label{eq:max-variance}
\sup_{n \in \ZP}    n^{-1} \Exp \bigl[ \bigl( \oMperp_n \bigr)^2 \bigr] < \infty,
\end{equation}
and   $\Exp [\oMperp_n] \approx \sqrt{n}$, so the crude bound~\eqref{eq:triangle-bound} is insufficient for either part of Proposition~\ref{prop:two-to-one-diamater}. However, we expect $D_n^\bmu \gg  \oMperp_n$, so the triangle inequality~\eqref{eq:triangle-bound} is a poor estimate.
\end{remark}

\begin{proof}[Proof of Proposition~\ref{prop:two-to-one-diamater}]
First, we use the crude bound~\eqref{eq:triangle-bound} to obtain
\begin{equation}
\label{eq:diam-approx-bound-1}
\limsup_{n \to \infty} n^{-1}
\Exp \left[ | D_n - D_n^\bmu |^2 \ind{ \oMperp_n > B \sqrt{n} } \right] \leq 4 \sup_{n \in \N} n^{-1} \Exp \left[ (  \oMperp_n )^2 \ind{ \oMperp_n > B \sqrt{n} } \right] ,
\end{equation}
which tends to $0$ as $B \to \infty$, by Lemma~\ref{lem:max-mean-ui}. Similarly, again using~\eqref{eq:triangle-bound}, we have,  for every fixed $B < \infty$ and $\delta >0$, 
\begin{equation}
\label{eq:diam-approx-bound-2}
\limsup_{n \to \infty} n^{-1}
\Exp \left[ | D_n - D_n^\bmu |^2 \ind{ \delta D_n^\bmu \leq \oMperp_n \leq B \sqrt{n} } \right] \leq 4 B^2 \limsup_{n \to \infty}  \Pr \left(   D_n^\bmu \leq B \sqrt{n} / \delta \right) ,
\end{equation}
where, by Chebyshev's inequality and the fact that $D_n^\bmu \geq M^\bmu_n  \geq S^\bmu_n$, for all $n$ large enough so that $2 B   <\sqrt{n} \| \bmu \| \delta$,
\[  \Pr \bigl(   D_n^\bmu \leq B \sqrt{n} / \delta \bigr) 
\leq \Pr \bigl(   S_n^\bmu \leq B \sqrt{n} / \delta \bigr) 
\leq \Pr \Bigl(  \bigl| S_n^\bmu - n \| \bmu \| \bigr| \geq (n/2) \| \bmu \| \Bigr) = O (1/n),
\]
as $n \to \infty$. Hence from~\eqref{eq:diam-approx-bound-2} we obtain,
 for every fixed $B < \infty$ and $\delta >0$, 
\begin{equation}
\label{eq:diam-approx-bound-2b}
\lim_{n \to \infty} n^{-1}
\Exp \left[ | D_n - D_n^\bmu |^2 \ind{ \delta D_n^\bmu \leq \oMperp_n \leq B \sqrt{n} } \right]  = 0.
\end{equation}
By~\eqref{eq:diam-approx-bound-1}--\eqref{eq:diam-approx-bound-2b}, for every $\eps>0$ and $\delta >0$,
we may choose $B = B(\eps, \delta)$ large enough so that
\begin{equation}
\label{eq:diam-approx-bound-3}
\limsup_{n \to \infty}  n^{-1}
\Exp \left[ | D_n - D_n^\bmu |^2  \right]  \leq \eps 
+ \limsup_{n \to \infty}  n^{-1}
\Exp \left[ | D_n - D_n^\bmu |^2 \ind{    \oMperp_n < \delta D_n^\bmu  } \right] .
\end{equation}
The rightmost expression in~\eqref{eq:diam-approx-bound-3} is where we must improve on the crude triangle-inequality bound~\eqref{eq:triangle-bound} (see Remark~\ref{rem:two-to-one-diamater}). Indeed, observe from~\eqref{eq:box-bound} that 
\begin{align}
\label{eq:box-bound2}
    D_n   \leq \frac{( D_n^\bmu )^2 + 4 ( \oMperp_n )^2}{  \sqrt{ ( { D_n^\bmu } )^2 + 4 ( \oMperp_n )^2 } } 
     \leq  D_n^\bmu  +  \frac{4 ( \oMperp_n )^2}{ { D_n^\bmu }  } .
 \end{align}
Consequently, $0 \leq D_n - D_n^\bmu \leq 4 \delta \oMperp_n$, on $\{ \oMperp_n < \delta D_n^\bmu  \}$. 
It follows that
\begin{equation}
\label{eq:diam-approx-bound-4}
\limsup_{n \to \infty}  n^{-1}
\Exp \left[ | D_n - D_n^\bmu |^2 \ind{    \oMperp_n < \delta D_n^\bmu  } \right] 
\leq 16 \delta^2 \sup_{n \in \N}  n^{-1}
\Exp \left[  ( \oMperp_n )^2  \right] ,
\end{equation}
which tends to~$0$ as $\delta \to 0$, by~\eqref{eq:max-variance}. Combining~\eqref{eq:diam-approx-bound-3} and~\eqref{eq:diam-approx-bound-4}, we conclude that for every $\eps>0$, we can choose $\delta = \delta (\eps)$ small enough and then $B = B (\eps , \delta)$ large enough so that
\[ \limsup_{n \to \infty}  n^{-1}
\Exp \left[ | D_n - D_n^\bmu |^2  \right]  \leq 2 \eps. \]
Since $\eps>0$ was arbitrary, this completes the proof of the $L^2$ part of~\eqref{eq:two-to-one-diamater-L2}.

The proof  the $L^1$ part of~\eqref{eq:two-to-one-diamater-L2} is similar, but simpler.  First note that, by~\eqref{eq:triangle-bound},
\[ 
 \Exp \left[ | D_n -D_n^\bmu | \ind { \oMperp_n >  n} \right] \leq
  \frac{2}{n} \Exp \left[ ( \oMperp_n )^2 \ind { \oMperp_n > n} \right]
  \leq
  \frac{2}{n} \Exp \left[ ( \oMperp_n )^2 \right],
\]
which is bounded uniformly over $n \in \N$, by~\eqref{eq:max-variance}. Next, by Chebyshev's inequality, $\Pr ( D^\bmu_n \leq \| \bmu \| n/2 ) \leq C/n$ for some $C< \infty$, depending only on $\| \bmu \|$, and all $n \in \N$. 
Hence, by~\eqref{eq:triangle-bound},
\[ \Exp \left[ | D_n -D_n^\bmu | \ind { D_n^\bmu \leq \| \bmu \| n/2, \, \oMperp_n \leq n } \right] \leq 2  n \Pr ( D^\bmu_n \leq \| \bmu \| n/2 ) 
\leq 2C .
 \]
Moreover,
by~\eqref{eq:box-bound2},
\[ \Exp \left[ | D_n -D_n^\bmu | \ind { D_n^\bmu > \| \bmu \| n/2} \right] 
\leq 
\frac{8}{n \| \bmu\|} \Exp [ (\oMperp_n)^2 ] 
. \]
Then another application of~\eqref{eq:max-variance} completes the proof.
\end{proof}

\begin{proof}[Proof of Theorem~\ref{thm:one-walk-diameter}]
Recall that $D_n^\bmu = M_n^\bmu - m_n^\bmu$. Here, Theorem~\ref{thm:WaldCLT} shows that $(n^{-1/2} m_n^\bmu)_{n\in\N}$ and $(n^{-1/2} (M_n^\bmu - S_n^\bmu))_{n\in\N}$ both tend to~$0$ in $L^2$, and hence $n^{-1/2} ( D_n^\bmu - \hbmu^\tra S_n) \to 0$, as $n\to\infty$, in $L^2$ also. Combined with the $L^2$ part of Proposition~\ref{prop:two-to-one-diamater}, this gives the $L^2$ convergence part of the theorem. Moreover, from~\eqref{eq:L1-gut} we have $\sup_n \Exp [| M_n^\bmu - S_n^\bmu |] < \infty$ and $\sup_n \Exp [| m_n^\bmu |] < \infty$, from which it follows that
$\sup_n \Exp[ | D_n^\bmu - \hbmu^\tra S_n |] < \infty$. Now use the $L^1$ part of Proposition~\ref{prop:two-to-one-diamater}.
\end{proof}

\section{Hausdorff approximation of the convex hull}
\label{sec:hausdorff}

\subsection{Notation and main result}
\label{sec:hausdorff-result}

As in~\S\ref{sec:known-gaussian}, consider $N$ random walks, $S^{(1)} ,\ldots, S^{(N)}$, 
satisfying~\eqref{ass:many-walks}, where walk $S^{(k)}$ has i.i.d.~increments with mean $\bmu_k$ and covariance 
matrix $\Sigma_k$,
with joint convex hull $\cH_n$ given by~\eqref{eq:multiple-hull}
and diameter $D_n$ and perimeter $L_n$ given by~\eqref{eq:multiple-hull-diam-perim}.
Recall that the polygon $\cC_\bmu = \hull \{ \0, \bmu_1, \ldots \bmu_N \}$ 
is the law-of-large-numbers limit of $n^{-1} \cH_n$. 
The aim of this section is to extend the framework of \S\ref{sec:wald}--\S\ref{sec:diam-one-walk} to encompass $L^2$ approximation results for \emph{multiple} random walks.
The main result of the present section is Theorem~\ref{thm:general-hausdorff}, which gives an $L^2$ approximation, in Hausdorff distance, of $\cH_n$ by a simpler (smaller) set. This will also serve as an intermediate result on our way to limit theorems for the diameter (in \S\ref{sec:max-Gaussian} below).

Denote by $\cJ$ the set of all $i \in \{0,1,\ldots,N\}$ such that the point at $\bmu_i$
is on the boundary~$\partial \cC_\bmu$ of~$\cC_\bmu$; 
here we set $\bmu_0 := \0$. Note that $\cC_\bmu = \hull \{ \bmu_k : k \in \cJ \}$. 
We call $\{ \bmu_k : k \in \cJ\}$
the set of \emph{boundary drifts} of $\cC_\bmu$,
and we call walks $S^{(k)}$ with $0 < k \in \cJ$  \emph{boundary walks}. 
Roughly speaking, the boundary walks are the potentially significant ones for the macroscopic  asymptotics of~$\cH_n$.

For $x,y \in \R^2$, $x \neq y$,
write $(x,y) := \{ x + \theta (y-x) : \theta \in (0,1)\}$
for the 
line segment between $x$ and $y$, excluding the endpoints. Each face of $\cC_\bmu$ is $(\bmu_i, \bmu_j)$ for some $i , j \in \cJ$ (recall $\bmu_0 = \0$). 
If $\0 \in \cJ$, then 
there are at most two faces of $\cC_\bmu$ incident to $\0$, 
although there may be several $\bmu_k$ for which the line segment $(\0,\bmu_k)$ is contained
in such a face.  
Partition $\cJ$ into
 $\cJ^0 := \{ k \in \cJ :  \bmu_k = \0\}$,
 $\cJ^\text{f} := \{ k \in \cJ : \bmu_k \neq \0, \, (\0,\bmu_k) \text{ is contained in a face} \}$
 and $\cJ^+ := \{ k \in \cJ : \bmu_k \neq \0, \, (\0,\bmu_k) \text{ is not contained in a face}\}$.
 Note that if $\0$ is in the interior of $\cC_\bmu$ (so $0 \notin \cJ)$, then $\cJ^0 = \cJ^\text{f} = \emptyset$, while if $0 \in \cJ$, then $0 \in \cJ^0$. On the other hand,
 if $\cC_\bmu$ has empty interior (i.e., is a line segment), then $\cJ^+ = \emptyset$.
 Let
 \begin{align*}
 \cG^0_n & :=  \bigl\{ S^{(k)}_i : 0 < k \in \cJ^0, \, i \in \{0,1,\ldots,n\} \bigr\} , \\
 \cG_n^\text{f}  & :=  \bigl\{ S^{(k)}_i : k \in \cJ^\text{f}, \, i \in \{0,1,\ldots,n\} \bigr\}, \\
 \cG^+_n & := \bigl\{ S^{(k)}_n : k \in \cJ^+ \bigr\} ;\end{align*}
 note that $\cG^0_n$ and $\cG^\text{f}_n$ contain \emph{trajectories} of relevant random walks, 
 and $\0 \in \cG^\text{f}_n$ whenever $\cJ^{f} \neq \emptyset$, 
 while $\cG^+_n$ contains only \emph{endpoints} of non-zero-drift walks.
For $n \in \ZP$, define
\begin{equation}
\label{eq:g-def}
\cG_n := \hull \bigl( \cG_n^0 \cup \cG_n^\text{f} \cup \cG_n^+ \bigr) \subseteq \cH_n.\end{equation}

The main result of this section is  Theorem~\ref{thm:general-hausdorff} on Hausdorff approximation, below.
While we anticipate that this result will be useful in other contexts, in the present
paper it acts as an intermediate result in our study of the diameter 
for multiple random walks.
\begin{theorem}
\label{thm:general-hausdorff}
    Suppose that \eqref{ass:many-walks} holds.
Then, with $\cG_n$ defined at~\eqref{eq:g-def}, 
    \[ \frac{ \rhoH \bigl( \cH_n, \cG_n )}{\sqrt{n}} \toL{n}{2} 0 .\]
\end{theorem}

\begin{remarks}
\label{rems:general-hausdorff}
\begin{myenumi}[label=(\alph*)]
\item\label{rems:general-hausdorff-a}
 Theorem~\ref{thm:general-hausdorff} can already lead to new functional approximations. 
 If $F: \cK_0  \to \RP$ is Lipschitz (with respect to the $\rhoH$ metric),
 then Theorem~\ref{thm:general-hausdorff} shows that
 \[ \frac{ F ( \cH_n ) - F(\cG_n) }{\sqrt{n}} \toL{n}{2} 0. \]
We already observed that~\eqref{eq:line-segment-limits} suggests a better approximation for $\cH_n$, i.e., by a simpler (smaller) set than $\cG_n$, for perimeter and diameter functionals, 
 but one cannot expect an approximation to $\cH_n$ in Hausdorff distance by a set essentially smaller than $\cG_n$. 
 To obtain an approximation by a smaller set than  $\cG_n$, and hence better functional approximation, one needs a different metric on convex compact sets: see Theorem~\ref{thm:perimeter-deviation-metric} below for such a result.
\item\label{rems:general-hausdorff-b}
If $\cC_\bmu$ has empty    interior (i.e., is a line segment), then $\cJ^+ = \emptyset$
    and so $\cG_n = \cH_n$, and  Theorem~\ref{thm:general-hausdorff} gives no non-trivial information in that case. This is the case, for example, if  $N=1$ and $\bmu_1 = \bmu \neq \0$, 
    a single random walk with non-zero drift.
\end{myenumi}
\end{remarks}

\subsection{Proofs}
\label{sec:hausdorff-proofs}
The main goal of this section is the proof of Theorem~\ref{thm:general-hausdorff}. The strategy is similar to that in~\S\ref{sec:wald}--\S\ref{sec:diam-one-walk}: we first show that
$(n^{-1/2} \rhoH ( \cH_n, \cG_n))_{n\in\N}$ is uniformly square-integrable, and then we prove convergence in probability. 
For the first part, define 
\[ \cU_n := \hull \bigl( \{ \0 \} \cup \{ S^{(k)}_n : k \in \{1, \ldots, N\} \} \bigr), \]
and \[ \cV_n := \begin{cases} \hull \bigl(  S^{(k)}_n  :  0 < k \in \cJ \bigr) & \text{if } 0 \notin \cJ, \\
\hull \bigl(  \{ \0 \} \cup \{ S^{(k)}_n  :  0 < k \in \cJ  \} \bigr) & \text{if } 0 \in \cJ. \end{cases} 
 \]
Then $\cV_n \subseteq \cU_n$ are polygons using only the endpoints of some or all of the random walks; note $\0 \in \cV_n$ whenever $0 \in \cJ$. 
The following uniform integrability result includes not only the result required for Theorem~\ref{thm:general-hausdorff}, but several related approximations that are needed later.

 \begin{lemma}
\label{lem:H-V-UI}
    Suppose that~\eqref{ass:many-walks} holds. 
    Then each of the following is uniformly integrable:
\begin{alignat*}{3}    
& \text{(i)}~(n^{-1} \rhoH ( \cH_n, \cU_n )^2)_{n\in\N}, 
&& ~~\text{(ii)}~(n^{-1} \rho_H ( \cU_n, n \cC_\bmu)^2)_{n\in\N}, 
&& ~~\text{(iii)}~(n^{-1} \rho_H ( \cV_n, n \cC_\bmu)^2)_{n\in\N}, \\ 
& \text{(iv)}~(n^{-1} \rhoH (\cU_n, \cV_n)^2)_{n\in\N}, 
&& ~~\text{(v)}~(n^{-1} \rhoH ( \cH_n, \cV_n )^2)_{n\in\N}, 
&& ~~\text{(vi)}~(n^{-1} \rhoH ( \cH_n, \cG_n)^2)_{n\in\N}, \end{alignat*} 
and 
     (vii)~$(n^{-1} \rhoH ( \cH_n, n \cC_\bmu )^2)_{n\in\N}$.
   \end{lemma}

\begin{definition}[Bounding rectangles]
\label{def:rectangle}
For every $k \in \{1,\ldots,N\}$, let $\hbmu_k^\per$ denote a unit vector orthogonal to $\hbmu_k$, and define
\begin{equation}
\label{eq:rectangle-quantities}
\oM_n^{(k)} := \max_{0 \leq i \leq n} \bigl| (\hbmu_k^\per)^\tra S^{(k)}_i \bigr|, ~  M^{(k)}_n := 
\max_{0 \leq i \leq n}  \hbmu_k^\tra S^{(k)}_i, ~ \text{and} ~ 
m^{(k)}_n := -\min_{0 \leq i \leq n}  \hbmu_k^\tra S^{(k)}_i.
\end{equation}
The trajectory $\{ S_i^{(k)} : 0 \leq i \leq n\}$
is contained in a rectangle 
\[
	R_n^{(k)} := \{ a \hbmu_k + b \hbmu_k^\perp : -m_n^{(k)} \le a \le M_n^{(k)}, \ |b| \le \overline{M}_n^{(k)} \},
\]
that is, a rectangle with major axis from $-m_n^{(k)} \hbmu_k$ to $M_n^{(k)}  \hbmu_k$  and of  width $2 \oM_n^{(k)}$. 
\end{definition}

In the following proof, and later on, we will make use of the fact that if $A, B$ are non-empty, compact subsets of $\R^2$, then
\begin{equation}
    \label{eq:rhoH-hull}
    \rhoH ( \hull A, \hull B) \leq \rhoH (A, B). 
\end{equation}
 
 \begin{proof}[Proof of Lemma~\ref{lem:H-V-UI}]
We first prove (i), i.e., we show that $(n^{-1} \rhoH ( \cH_n, \cU_n)^2)_{n\in\N}$ is uniformly integrable. 
We start by showing that for every $x \in R_n^{(k)}$, we have
		\begin{align}\label{eq:rect-pre-bound}
		\operatorname{dist}\bigl(x, \mathrm{hull}\{ \0, S_n^{(k)} \}\bigr)
		\le 2\oM_n^{(k)}  + \bigl( M_n^{(k)} - \hbmu_k^\top S_n^{(k)} \bigr) + m_n^{(k)},\quad \text{a.s.}
		\end{align}
		Indeed, write $u := \hbmu_k$ and $v := \hbmu_k^\perp$. Then, any $x \in R_n^{(k)}$ can be written as
		$x = a u + b v$ with $-m_n^{(k)} \le a \le M_n^{(k)}$ and $|b| \le \overline{M}_n^{(k)}$. We also write
		$S_n^{(k)} = \alpha u + \beta v,$, where $\alpha = u^\top S_n^{(k)}$ and $\beta = v^\top S_n^{(k)}$.
		Consider the distance from $x$ to the segment $\mathrm{hull}\{\0,S_n^{(k)}\} = \{ t S_n^{(k)} : t \in [0,1]\}$. The transverse components satisfy $|b| \le \overline{M}_n^{(k)}$ and $|\beta| \le \overline{M}_n^{(k)}$, so the total transverse discrepancy is at most $2\overline{M}_n^{(k)}$.
		Further, if $a > \alpha$, then the projection of $x$ in direction $u$ exceeds that of $S_n^{(k)}$. This excess is bounded by
		\[
		a - \alpha \le M_n^{(k)} - \alpha = M_n^{(k)} - \hat{\mu}_k^\top S_n^{(k)}.
		\]
		If $a < 0$, then $x$ lies behind the origin in direction $u$, and this contributes at most $	|a| \le m_n^{(k)}$. This verifies~\eqref{eq:rect-pre-bound}.

Let $\cR_n := \hull ( \cup_{1 \leq k \leq N} R^{(k)}_n )$ denote the convex hull of all the bounding rectangles described in Definition~\ref{def:rectangle}. Then $\cH_n \subseteq \cR_n$ and, by \eqref{eq:rect-pre-bound}, every $x \in  \cup_{1 \leq k \leq N} R^{(k)}_n$ is within distance  
\[ \oM^\star_n := \sum_{1 \leq k \leq N}\left\{ 2 \oM_n^{(k)} + \bigl( M^{(k)}_n - \hbmu_k^\tra S^{(k)}_n \bigr) 
+   m_n^{(k)}\right\} 
\]
of one of the line segments from $\cup_{1 \leq k \leq N} \hull \{ \0, S_n^{(k)} \}$. 
Thus $$\rhoH ( \cup_{1 \leq k \leq N} R^{(k)}_n , \cup_{1 \leq k \leq N} \hull \{ \0 , S^{(k)}_n \} ) \leq \oM^\star_n,\quad \text{a.s.}$$ Since $\cU_n \subseteq \cH_n \subseteq \cR_n$, it follows that, for all $n \in \ZP$, by~\eqref{eq:rhoH-hull}, 
\begin{equation}
    \label{eq:multiple-box-bound}
 \rhoH ( \cH_n , \cU_n ) \leq \rhoH ( \cR_n, \cU_n ) \leq 
\oM^\star_n,\quad \as
\end{equation}
Combining~\eqref{eq:multiple-box-bound} and Lemma~\ref{lem:max-sum-ui} shows $(n^{-1} \rhoH ( \cH_n, \cU_n)^2)_{n\in\N}$ is uniformly integrable.

For part (ii), observe that by~\eqref{eq:rhoH-hull}, 
 \[\rhoH ( \cU_n, n \cC_\bmu ) \leq 
 \rhoH (  \{ \0 , S_n^{(1)}, \ldots, S_n^{(N)} \}  ,   \{ \0 , n \bmu_1 ,\ldots, n \bmu_N\} ) 
 \leq \max_{1 \leq k \leq N}
 \| S_n^{(k)} - n \bmu_k \| .\]
 By the fact~\eqref{eq:sum-mean-ui} applied to each coordinate,
 we have that $(n^{-1}  \| S_n^{(k)} - n \bmu_k \|^2)_{n\in\N}$ is uniformly integrable, for each $k$, and part~(ii) follows. Similarly,
since $\cC_\bmu = \hull \{ \bmu_k : k \in \cJ \}$,
 \[\rhoH ( \cV_n, n \cC_\bmu ) \leq 
 \rhoH ( \{ S_n^{(k)} : k \in \cJ \} ,   \{ n \bmu_k :   k \in \cJ \} )
\leq \max_{k \in \cJ} \| S_n^{(k)} - n \bmu_k \| ,\]
where $S_n^{(0)} := \0$ if $0 \in \cJ$, from which we obtain~(iii). 
 
The triangle inequality together with parts (ii) and (iii) immediately gives~(iv), while parts~(i) and~(iv) give~(v). 
Since $\cV_n \subseteq \cG_n \subseteq \cH_n$,
we have $\rhoH ( \cH_n, \cG_n ) \leq \rhoH ( \cH_n, \cV_n )$,
and so (vi) is immediate from (v). 
 Parts (iii) and (v) together imply part (vii).
 \end{proof}

\begin{proof}[Proof of Theorem~\ref{thm:general-hausdorff}]
In view of Remark~\ref{rems:general-hausdorff}\ref{rems:general-hausdorff-b} and Lemma~\ref{lem:H-V-UI}(vi), it remains to show   
\begin{equation}
    \label{eq:G-H-probab}
\frac{ \rhoH \bigl( \cH_n, \cG_n )}{\sqrt{n}} \toP{n} 0, \end{equation}
and one can suppose that $\cC_\bmu$ has non-empty interior. 
Let $\cJ^\circ := \{0,1,\ldots, N\} \setminus \cJ$ be the set of all indices $k$ such that $\bmu_k$ (including, possibly, $\bmu_0 =\0$) lies in the interior of $\cC_\bmu$.
Then there exists $\delta >0$ such that for all convex, compact $C$ with $\rhoH (C, \cC_\bmu) < \delta$, for every  $k \in \cJ^\circ$, it holds that $\bmu_k \in C$ and $\bmu_k$ is distance at least $\delta$ from $\partial C$ (cf.~\cite[p.~67]{schneider_book}). By the SLLN, 
\begin{equation}
\label{eq:small-ball}
\lim_{n \to \infty} \Pr ( \cap_{1 \leq k \leq N} \{ \| S_n^{(k)} - n \bmu_k \| < \delta n \} ) =
\lim_{n \to \infty} \Pr ( \rhoH ( n^{-1}  \cV_n , \cC_\bmu ) < \delta ) = 1.
\end{equation}
In particular, with probability $1 - o_n(1)$,
for every   $k \in \cJ^\circ$ it holds that 
$n \bmu_k \in \cV_n \subseteq \cG_n$ and $n \bmu_k$ is distance at least $\delta n$ from $\partial  \cG_n$ (for sufficiently large $n$). 

Recall from Definition~\ref{def:rectangle} that
 trajectory $\{ S_i^{(k)} : 0 \leq i \leq n\}$
is contained in a rectangle $R^{(k)}_n$ with major axis from $-m_n^{(k)} \hbmu_k$ to $M_n^{(k)}  \hbmu_k$ of  width $2 \oM_n^{(k)}$, with the notation at~\eqref{eq:rectangle-quantities}.
Suppose that $\0$ and $\bmu_k$ ($k > 0$) are both in the interior of $\cC_\bmu$ (i.e., $0,k \in \cJ^\circ$). 
In both cases $\bmu_k = \0$ and $\bmu_k \neq \0$, 
sequences $(n^{-1/2} m_n^{(k)})_{n\in\N}$, $(n^{-1/2} ( M_n^{(k)} - \hbmu_k^\tra S_n^{(k)} ))_{n\in\N}$, and $(n^{-1/2} \oM_n^{(k)})_{n\in\N}$ are tight, and so, with probability $1-o_n(1)$, the rectangle  $R^{(k)}_n$
is contained in the interior of~$\cG_n$, and the trajectory of walk $S^{(k)}$ plays no part in $\rhoH ( \cH_n, \cG_n)$.

The previous argument shows that it suffices to consider (i) $k \in \cJ$, or (ii) $k \in \cJ^\circ$ and $0 \in \cJ$. In case (ii), necessarily $\bmu_k \neq \0$.
Extend the faces of $\cC_\bmu$ incident to $\0$ to form a wedge with apex at $\0$; the drift $\bmu_k$ is directed into this wedge, and the SLLN implies that, a.s., for all but finitely many $n \in \N$, the infinite trajectory of~$(S^{(k)}_n)_{n \in \ZP}$ remains in this wedge. Consequently, using~\eqref{eq:small-ball} once more, with probability $1-o_n(1)$ there are fewer than $n^{\eps}$ (for any fixed $\eps>0$) points of $S_1^{(k)}, \ldots, S_n^{(k)}$ outside $\cV_n$ (hence also outside $\cG_n$). A similar argument applies in case~(i) provided that $(\0,\bmu_k)$ is not contained in a face, considering the time-reversal of the random walk. 
Thus we conclude, for every $\eps > 0$, 
$\lim_{n \to \infty} \Pr ( \rhoH ( \cH_n , \cG_n ) < n^\eps ) = 1$,
which implies~\eqref{eq:G-H-probab}.
\end{proof}

\section{Fluctuations of the diameter for multiple random walks}
\label{sec:max-Gaussian}

\subsection{Notation and main results}
\label{sec:max-Gaussian-overview}

In this section, we establish $L^2$-approximation (Theorem~\ref{thm:diameter-approximation-general}) and distributional convergence (Theorem~\ref{thm:many-walks-diameter}) results for the diameter $D_n$ of the convex hull of $N \geq 1$ random walks  under hypothesis~\eqref{ass:many-walks}, covering all cases and with quite explicit limit distributions, which are generically Gaussian but non-Gaussian in certain cases. As an illustration, the case $N=2$ is presented in~\S\ref{sec:max-Gaussian-examples}. For the present section, it is convenient  (without loss of generality) to
\begin{equation}
\label{eq:mus-increasing}
\text{assume that $\| \bmu_1 \| \leq \| \bmu_2 \| \leq \cdots \leq \| \bmu_N \|$} ; \end{equation}
in particular, this means that any of the $\bmu$s that are zero  will  appear at the start of the list.

Theorem~\ref{thm:general-hausdorff} approximates $\cH_n$ by $\cG_n$, which, for example, includes all endpoints of those random walks whose drift is on the boundary of $\cC_\bmu$. 
For the diameter, not all such boundary walks are significant,
since $\cC_\bmu$ governs the global shape of $\cH_n$ and, with high probability, only those walks that contribute to diametrical chords of $\cC_\bmu$ will contribute to~$D_n$. Theorem~\ref{thm:diameter-approximation-general} below formalizes this intuition and removes extraneous elements from $\cG_n$ for approximating $D_n$. The theorem has two related, but different statements: first, we exhibit a set $\cA_n$ such that $\diam \cA_n$ approximates $D_n$, and second (stronger), by concretely identifying an approximation for $D_n$ in terms of the underlying walks (direct computation of $\diam \cA_n$ may need other redundant comparisons). We set up the notation for both statements. 

Every chord in $\cC_\bmu$ that attains the diameter has as endpoints a pair of extreme points;
let $\cE$ denote the extreme points of $\cC_\bmu$.
Recall the definition of $\cJ$ from \S\ref{sec:hausdorff-result}.
Let 
\[ J_\bmu := \{ k \in \cJ : \bmu_k \in \cE, \, \| \bmu_k - \bmu_j \| = \diam \cC_\bmu \text{ for some } j \in \cJ \} , \]
the set of indices that are included in at least one diametrical chord of $\cC_\bmu$
(note that $J_\bmu$ always contains at least two elements, and can include $0$ if $\bmu_0 = \0$ is an endpoint of a diametrical chord).
 We partition  $J_\bmu$ into
 $J_\bmu^0 := \{ k \in J_\bmu : \bmu_k = \0\}$ and $J^+_\bmu := \{ k \in J_\bmu : \bmu_k \neq \0\}$.
 Let
 \[ \cS^0_n := \begin{cases}
     \bigl \{ S^{(k)}_i : 0 < k \in J^0_\bmu, \, i \in \{0,1,\ldots,n\} \bigr\},& \text{if } 0 \notin J_\bmu; \\
     \{\0\} \cup \bigl \{ S^{(k)}_i : 0 < k \in J^0_\bmu, \, i \in \{0,1,\ldots,n\} \bigr\},& \text{if } 0 \in J_\bmu,
 \end{cases}   \]
 and
 \[
 \cS^+_n := \{ S^{(k)}_n : k \in J^+_\bmu \}.
 \]
 Note that $\cS^0_n$ contains \emph{trajectories} of relevant zero-drift walks, while $\cS^+_n$ contains only \emph{endpoints} of non-zero-drift walks.
Define
\begin{equation}
    \label{eq:cAn-def}
 \cA_n := \hull \bigl(   \cS_n^0 \cup \cS_n^+ \bigr) .\end{equation}
Note that if there are \emph{no} $0 < k \in J_\bmu$ for which $\bmu_k = \0$, then
$\cS_n^0$ is either empty (if $0 \notin J_\bmu$) or equal to $\{ \0\}$ (if $0 \in J_\bmu$). 

Every possible diametrical chord of $\cC_\bmu$
corresponds to a line segment between $\bmu_i$ and $\bmu_j$ where $i, j$
are two distinct elements from $J_\bmu$. 
Represent these by the set $\cD_\bmu \subseteq J_\bmu^2$  of ordered pairs $( i,j )$, $i < j$, of distinct elements of $J_\bmu$: \[ \cD_\bmu := \{ ( i ,j ) \in J_\bmu^2 :  i <j, \,  \| \bmu_i - \bmu_j \| = \diam  \cC_\bmu \}. \]
Distinguish  
those diametrical chords that use two non-zero $\bmu$s,
\begin{equation}
    \label{eq:D+}
 \cD_\bmu^+ := \{ (i,j ) \in \cD_\bmu  : \| \bmu_i \| \cdot \| \bmu_j \| > 0 \}, 
 \end{equation}
those that use a single non-zero $\bmu_j$ and a zero $\bmu_i$ (recall~\eqref{eq:mus-increasing}),
\begin{equation}
    \label{eq:D-times}
 \cD_\bmu^\times := \{ (i,j ) \in \cD_\bmu  :  \bmu_i = \0 \neq \bmu_j  \}, 
 \end{equation}
and those that use a single non-zero $\bmu_j$ and the origin
\begin{equation}
    \label{eq:D-circ}
 \cD_\bmu^\circ := \{ (i,j ) \in \cD_\bmu  : i = 0, \, \bmu_j \neq \0 \}. \end{equation}
 Note that if $\bmu_i = \0$ for some $i \in \{1,\ldots,N\}$,
then for every $(0,j) \in \cD^\circ_\bmu$ 
(with $j >i$ by~\eqref{eq:mus-increasing})
it holds that $(i,  j) \in \cD^\times_\bmu$ also. Conversely, every $(i  , j) \in \cD_\bmu^\times$ means that $(0,j) \in \cD_\bmu^\circ$ also. For $i,j \in \{1,\ldots,N\}$, write $S^{(i,j)}_n := S^{(i)}_n - S^{(j)}_n$.

\begin{theorem}
\label{thm:diameter-approximation-general}
    Suppose that \eqref{ass:many-walks} holds.
    The following statements hold.
\begin{thmenumi}[label=(\roman*)]
\item
\label{thm:diameter-approximation-general-i}
With~$\cA_n$ as defined at~\eqref{eq:cAn-def}, we have
\[ \frac{ D_n - \diam \cA_n }{\sqrt{n}}\toL{n}{2} 0 . \]
\item
\label{thm:diameter-approximation-general-ii}
Using the notation at~\eqref{eq:D+}, \eqref{eq:D-times}, \eqref{eq:D-circ}, we have
    \[
\frac{{D_n}-\max \left( 
    \max\limits_{(i,j) \in \cD^+_\bmu} \| S_n^{(i,j)} \|,
    \max\limits_{(0,j) \in \cD^\circ_\bmu} \| S_n^{(j)} \| ,
    \max\limits_{(i,j) \in \cD^\times_\bmu} \max_{0 \leq k \leq n} \|   S_n^{(j)} -  S_k^{(i)} \|
             \right)  }{\sqrt{n}} \toL{n}{2} 0.
        \]
\end{thmenumi}
\end{theorem}
 
For $N=1$ walk with non-zero drift, $\cA_n = \hull \{ \0, S_n \}$, $\cD^+_\bmu = \cD^\times_\bmu = \emptyset$, $\cD^\circ_\bmu = \{(0,1)\}$, 
and Theorem~\ref{thm:diameter-approximation-general}
recovers the diameter result in~\eqref{eq:line-segment-limits} from~\cite{mcrw}. 
For $N=2$ walks, Theorem~\ref{thm:diameter-approximation-general}   recovers the result~\eqref{eq:ikss-diameter} from~\cite{ikss} under hypotheses~\eqref{ass:no-zero} and~\eqref{ass:unique-max}.
However, the main value of Theorem~\ref{thm:diameter-approximation-general}
is that it applies to an arbitrary number $N \geq 1$ of walks, and 
is the geometrical component in our new fluctuation results in Theorem~\ref{thm:many-walks-diameter} below.
To describe the explicit distributional limits we need some more notation;
in~\S\ref{sec:max-Gaussian-examples} we illustrate the  case of $N=2$ walks, where one or both of the hypotheses~\eqref{ass:no-zero}--\eqref{ass:unique-max} introduced in~\ref{sec:two-walk-intro} fail, and the limits are non-Gaussian, in contrast to the Gaussian cases previously observed in the literature~\cite{mcrw,ikss}.

For $i \in \{1,\ldots,N\}$ set $\bmu_{i,j} := \bmu_i - \bmu_j$. 
Let $\Sigma_{i,j} := \Sigma_i + \Sigma_j$ and $\tZ_{k,1} := Z_{k,1} - \bmu_k$; then $\Exp [ (\tZ_{i,1} - \tZ_{j,1} )( \tZ_{i,1} - \tZ_{j,1})^\tra ] = \Sigma_{i,j}$. We also define, 
analogously to~\eqref{eq:spara-def},  for $i, j \in\{1,\ldots,N\}$, 
\begin{equation}
\label{eq:spara-def-two}
\spara{i,j} := \hbmu_{i,j}^\tra \Sigma_{i,j} \hbmu_{i,j} \quad \text{and} \quad \spara{0,j} :=
\spara{j} := \hbmu_j^\tra \Sigma_j \hbmu_j.\end{equation} 
We are going to describe the limit distribution for the centred and scaled diameter in terms of a particular multivariate Gaussian collection $\zeta$ and a non-Gaussian collection $\xi$, where $\zeta$ and $\xi$ are independent, and consist of
\begin{equation}
\label{eq:zeta-xi-notation}
    \zeta = (\zeta_e : e \in \cD_\bmu^+ \cup \cD_\bmu^\circ), \text{ and } \xi = (\xi_e : e \in \cD^\times_\bmu).
\end{equation}
The marginals of $\zeta$ are, for each   $e \in \cD_\bmu^+ \cup \cD_\bmu^\circ$
and with $\spara{e}$ given by~\eqref{eq:spara-def-two}, 
\begin{equation}\label{eq:dist_of_zeta_e}
    \zeta_e \sim \cN \bigl( 0 , \spara{e} \bigr),
\end{equation}
and $\zeta_{e_1}$ and $\zeta_{e_2}$ are independent unless $e_1$ and $e_2$ share a common non-zero index. For such cases, we introduce the notation $e_1 \wedge e_2$ for the common index, and $e_1 \oplus e_2$ to be equal to $+1$ if the common index is in the same position in $e_1$ and $e_2$, and $-1$ otherwise. Then we have
\begin{equation}
    \label{eq:cov_of_zeta_e}
\Cov(\zeta_{e_1}, \zeta_{e_2}) = (e_1 \oplus e_2) \hbmu^\tra_{e_1} \Sigma_{e_1 \wedge e_2} \hbmu_{e_2}.
\end{equation}
The distribution of $\zeta$ is specified by the variances~\eqref{eq:dist_of_zeta_e} and covariances~\eqref{eq:cov_of_zeta_e}. On the other hand, the components of $\xi$ are all described in terms of a collection of independent planar Brownian motions $\{(W^{(k)}_t)_{t\in\RP} : 0 < k \in J_\bmu^0\}$ via
$\xi_{i,j} := - \inf_{0 \leq t \leq 1} \hbmu_j^\tra \Sigma^{1/2}_i W^{(i)}_t$.  The following theorem, the formal version of~\eqref{eq:diameter-limit},
 is the main result of this section.

\begin{theorem}
\label{thm:many-walks-diameter}
    Suppose that \eqref{ass:many-walks} holds and that $\bmu_k \neq \0$ for at least one~$k$. Then
    \[
\frac{{D_n}-\max \left( 
    \max\limits_{(i,j) \in \cD^+_\bmu} \hbmu^\tra_{i,j} S_n^{(i,j)} ,
    \max\limits_{(0,j) \in \cD^\circ_\bmu} \hbmu^\tra_{j} S_n^{(j)} ,
    \max\limits_{(i,j) \in \cD^\times_\bmu} \Bigl( \hbmu^\tra_{j} S_n^{(j)} 
    + \max\limits_{0 \leq k \leq n} ( - \hbmu^\tra_{j} S_k^{(i)} ) \Bigr)
             \right)  }{\sqrt{n}} \toL{n}{2} 0.
        \]
        Moreover, with $\zeta$ and $\xi$ the random elements from~\eqref{eq:zeta-xi-notation}, 
\[ \frac{D_n - n \diam \cC_\bmu }{\sqrt{n}} \tod{n} \Delta := \max \left( \max_{(i,j) \in \cD^+_\bmu}  \zeta_{i,j} , \max_{(0,j) \in \cD^\circ_\bmu} \zeta_{0,j} ,
\max_{(i,j) \in \cD^\times_\bmu} \left( \zeta_{0,j} + \xi_{i,j} \right) \right), \]
and $\lim_{n \to \infty} n^{-1}  \Var D_n = \Var \Delta$.
\end{theorem}
\begin{remarks}
    \label{rems:two-walks-diameter}
    \begin{myenumi}[label=(\alph*)]
\item\label{rems:two-walks-diameter-i} 
Correlations of $\xi_{i,j}$ and $\xi_{i,k}$ can be computed using~\cite{rs}.
\item \label{rems:two-walks-diameter-ii}
The case when $\bmu_k = \0$ for all $k$ can be settled by Donsker's theorem: see~Appendix~\ref{sec:donsker}. 
\end{myenumi}
\end{remarks}

If the $N$-walk generalization of~\eqref{ass:unique-max} holds, the following CLT generalizes the $N \in \{1,2\}$ cases~\cite{mcrw,ikss};
as indicated in~\cite[p.~108]{ikss},   the methods of~\cite{mcrw,ikss} adapt to this case also.

\begin{corollary}
\label{cor:unique-diameter}
    Suppose that $J_\bmu = \{ k_1, k_2\}$ with at least one of $\bmu_{k_1}, \bmu_{k_2}$ not $\0$ (i.e., $C_\bmu$ has a unique pair of distinct vertices that attain the diameter). Then, with $\spara{k_1, k_2}$   defined in \eqref{eq:spara-def-two},
    \[
    \frac{D_n -  n \diam \cC_\bmu}{\sqrt{n}} \tod{n} \cN(0,  \spara{k_1, k_2}), \text{ and } \lim_{n \to \infty} n^{-1} \Var (D_n) = \spara{k_1, k_2}. 
    \]
    \end{corollary}

\subsection{Applications for two random walks}
\label{sec:max-Gaussian-examples}

The following immediate consequence of Theorem~\ref{thm:many-walks-diameter} completes the picture for $N = 2$ walks filling in the cases not covered in~\cite{ikss}.
\begin{corollary}\label{cor:diam_two_walks}
\begin{thmenumi}[label=(\roman*)]
     \item
     \label{cor:diam_two_walks-i}
     (Isosceles drifts I.)
If $\| \bmu_1 \| = \| \bmu_2 \| > \| \bmu_1 - \bmu_2 \| \geq 0$, then
       \[ \frac{D_n - n \diam \cC_\bmu}{\sqrt{n}} \tod{n} \max (\zeta_{0, 1}, \zeta_{0, 2}) ,\]
    where $\zeta_{0, 1}$ and $\zeta_{0, 2}$ are independent Gaussians with distribution given in \eqref{eq:dist_of_zeta_e}.
    \item 
         \label{cor:diam_two_walks-ii}
    (Isosceles drifts II.) Suppose that $\| \bmu_1 \| = \| \bmu_1 - \bmu_2 \| > \| \bmu_2 \| >0$. Then
    \[ \frac{D_n - n \diam \cC_\bmu}{\sqrt{n}} \tod{n} \max (\zeta_{0, 1}, \zeta_{1,2}) ,\]
    where $\zeta_{0, 1}$ and $\zeta_{1,2}$ are correlated Gaussians with marginals~\eqref{eq:dist_of_zeta_e} and correlation~\eqref{eq:cov_of_zeta_e}.
    \item
         \label{cor:diam_two_walks-iii}
    (Equilateral drifts.) Suppose that $\| \bmu_1 \| = \| \bmu_2 \| = \| \bmu_1 - \bmu_2 \| > 0$. Then
    \[ \frac{D_n - n \diam \cC_\bmu}{\sqrt{n}} \tod{n} \max (\zeta_{0,1}, \zeta_{0,2}, \zeta_{1,2}),\]
    where $\zeta_{0,1}$, $\zeta_{0,2}$ and $\zeta_{1,2}$ are correlated Gaussians with marginals~\eqref{eq:dist_of_zeta_e} and correlations~\eqref{eq:cov_of_zeta_e}.
    \item
         \label{cor:diam_two_walks-iv}
    (Ice-cream.) Suppose that $\|\bmu_1\| = 0 < \|\bmu_2\|$. Then
    \[
    \frac{D_n - n \diam \cC_\bmu}{\sqrt{n}} \tod{n} \zeta_{0, 2} + \xi_{1, 2},
    \]
    where $\zeta_{0, 2}$ and $\xi_{1, 2}$ are independent; $\zeta_{0, 2}$ is Gaussian, and $\xi_{1, 2}$ is non-Gaussian.
\end{thmenumi}       
\end{corollary}

\subsection{Proofs}
Theorem~\ref{thm:general-hausdorff} shows that $D_n = \diam \cH_n$ is well approximated by  $\diam \cG_n$. 
To obtain Theorem~\ref{thm:diameter-approximation-general}, from which
we also deduce Theorem~\ref{thm:many-walks-diameter},
we show that $\diam \cG_n$ is well approximated by 
the diameter of a generically smaller set, and, moreover, that we can explicitly identify the
collection of pairwise distances between random walk locations that one must maximize to get the approximate diameter.
Introduce notation
\begin{align}
\Gamma_\bmu^+ (n , m) & := 
    \max\limits_{(i,j) \in \cD^+_\bmu} \max_{n - m \leq \ell \leq n} \| S_\ell^{(i,j)} \|, 
    \label{eq:gamma-def-1}
    \\
    \Gamma_\bmu^\circ (n, m) & := 
    \max\limits_{(0,j) \in \cD^\circ_\bmu} \max_{\ell \in \ZP : \min ( \ell, n -\ell) \leq m} \| S_m^{(j)} \| , 
        \label{eq:gamma-def-2}\\
\Gamma_\bmu^\times (n,m) & := 
    \max\limits_{(i,j) \in \cD^\times_\bmu} \max_{0 \leq k \leq n}  \max_{n - m \leq \ell \leq n} \|   S_\ell^{(j)} -  S_k^{(i)} \|, 
        \label{eq:gamma-def-3}\\
    \Gamma_\bmu (n , m) & := \max ( \Gamma_\bmu^+ (n , m), \Gamma_\bmu^\circ (n , m), \Gamma_\bmu^\times (n ,m) ).    \label{eq:gamma-def-4}
\end{align}
With the notation~\eqref{eq:gamma-def-1}--\eqref{eq:gamma-def-4}, the statement of Theorem~\ref{thm:diameter-approximation-general}\ref{thm:diameter-approximation-general-ii} is that $n^{-1/2} (D_n - \Gamma_\bmu (n,0)) \to 0$ in $L^2$.
The main ingredient in the proof of Theorem~\ref{thm:diameter-approximation-general} is the following.

\begin{proposition}
\label{prop:diam-approx}
     Suppose that \eqref{ass:many-walks} holds. Then, for every $\eps>0$, $\lim_{n \to \infty} \Pr ( | D_n -  \diam \cA_n | \leq n^{(1/4)+\eps} )
     = \lim_{n \to \infty} \Pr ( | D_n -  \Gamma_\bmu (n,0) | \leq n^{(1/4)+\eps} )= 1$.
\end{proposition}

The first step in proving Proposition~\ref{prop:diam-approx} is a deterministic lemma, which  says that for a  convex set generated by finitely many points, points of the set whose distance is close to the diameter must be correspondingly close to vertices that attain the diameter. 

\begin{lemma}
\label{lem:near-diameters}
Let $U = \{ u_1, \ldots, u_m \}$, $m \geq 3$, be the (distinct) vertices of a polygon,
which we denote by $\cC = \hull U$. Let $A := \{ (i,j) : 1 \leq i < j \leq m, \, \| u_i - u_j \| = \diam \cC \}$ be the set of all pairs of vertices (ordered by index) that attain the diameter. Then there exist $B, \delta_0 \in (0,\infty)$ such that for all   $\delta \in (0,\delta_0)$, and every $x, y \in \cC$ with $\| x - y \| >    \diam \cC - \delta $, there exist $(i,j) \in A$ such that either
    $\max ( \| x - u_i \|, \| y - u_j\| ) < B  \delta$
    or $\max ( \| y - u_i \|, \| x - u_j\| ) < B  \delta$. 
\end{lemma}
\begin{proof}
Let $V \subset U$ denote the vertices of $\cC$ that are represented in $A$,
i.e., every $(i,j) \in A$ has $u_i = v$ and $u_j = v'$ for some $v, v' \in V$.
By definition of $A$,
there is a constant $\eps_0 = \eps_0 (\cC)$ such that $\| u_i - u_j \| < \diam \cC - \eps_0$
for every $1 \leq i < j \leq m$ for which $(i,j) \notin A$.

First suppose that the polygon $\cC$ has non-empty interior. 
Let $\delta_1$ be the positive constant equal to 
one half the minimal face
length of $\cC$, and let $\delta \in (0,\delta_1)$.
Then to each $v \in V$ associate the points $\ell_\delta (v)$ and $r_\delta (v)$,
the two boundary points of $\cC$ at distance $\delta$ from $v$, one on each of the two faces that meet at $v$. The convex set 
\[\cC_\delta := \hull \bigl( (  U \setminus V ) \cup \{ \ell_\delta (v), r_\delta (v) : v \in V \} \bigr) \]
has $\cC_\delta \subset \cC$ and has the corners around each $v\in V$ `sliced off'. 

Fix $u \in U$. Then, by convexity, 
$\sup_{x \in \cC} \| u - x\|$ is attained by each $x \in V_u$ for some non-empty $V_u \subseteq U \setminus \{u\}$ (see~Theorem~5.6 of~\cite[p.~76]{gruber}).  
Moreover, for all $x \in \cC$ and any $v \in V_u$,
\[ \| u - x \|^2 = \| u - v\|^2 + \| x - v\| \bigl( \| x-v \| - 2 \| u-v \| \cos \theta (x,u,v) \bigr) ,\]
where $\theta (x,u,v)$ is the angle between $u-v$ and $x-v$; note 
$\theta_0 := \sup_{u \in U,\, v \in V_u,\, x \in \cC}  |\theta (x,u,v) | < \frac{\pi}{2}$. Also $\| u - v\| \geq \delta_2 := \inf_{u \in U,\, v \in V_u} \| u - v\|$. Thus for $\| x - v \| \leq \delta_3 :=\delta_2 \cos \theta_0$, we have $\| u - x \|^2 - \| u - v \|^2 \leq - \delta_3 \| x - v \|$. Hence
\begin{equation}
\label{eq:geometry} \| u - x \| - \| u - v \| = \frac{\| u - x \|^2 - \| u - v\|^2}{\| u -x \| + \| u - v \|} \leq - \alpha \| x - v \|, \quad \text{for all } \| x - v \| \leq \delta_3 ,\end{equation}
where $\alpha := \delta_3/(2 \diam \cC) >0$. 

Set $\delta_4 := \min(\delta_1, \delta_3)$ and suppose $\delta \in (0,\delta_4)$. 
We want an upper bound on $\| x - y\|$ over $x \in \cC_\delta$, $y \in \cC$.
For fixed $x$, the bound is achieved by some $u \in U$.
Then $\sup_{x \in \cC_\delta} \| x - u\| = \| v -u \|$
for some $v \in W_\delta$. Either $v \in U \setminus V$, in which case $\| u - v\| \leq \diam \cC - \eps_0$, or else $v \in W_\delta \setminus U$ is one of the vertices from `slicing off' the corners, which by construction is within distance $\delta$ of an element of $V_u$. In that case, by~\eqref{eq:geometry}, 
$\| x - u \| \leq \| u - v \| - \alpha \| x -v \|$ where $\| u -v \| \leq \diam \cC$.
There are constants $\delta_5 >0, \beta>0$ depending only on $\cC$ such that  $\| x - v\| > \beta \delta$ for all $x \in C_\delta$ and all $v \in V$ provided $\delta \in (0,\delta_5)$. Thus taking $\delta_0 \in (0,\min(\delta_4, \delta_5))$ small enough so that $\alpha \beta \delta_0 < \eps_0$, and setting $a:= \alpha \beta >0$, we obtain, for all $\delta \in (0,\delta_0)$,
\begin{equation}
    \label{eq:C-C0}
    \text{for every $x \in \cC_\delta$ and every $y \in \cC$, $\| x - y\| < \diam \cC - a \delta$.
}
\end{equation}

Suppose that $\delta \in (0,  \delta_0)$ and
$x, y \in \cC$ with $\| x - y \| > \diam \cC - a \delta$.
Then~\eqref{eq:C-C0} shows that $x \in \cC \setminus \cC_\delta$, and hence $\| x - u_i \| \leq \delta$ for some $u_i \in V$. 
By the same argument, $y \in \cC \setminus \cC_\delta$,
and $\| y - u_j \| \leq \delta$ for some $u_j \in V$. Moreover, we must have $(i,j) \in A$ or $(j,i) \in A$, since $a \delta < a \delta_0 < \eps_0$. 
This completes the proof in the case where $\cC$ has non-empty interior. Otherwise, $\cC$ is a line segment and the result is easily verified in that case.
\end{proof}

We will use Lemma~\ref{lem:near-diameters} to show that from the trajectories of the random walks retained in~$\cG_n$, relevant for the diameter are only points near the start or end of the trajectory. 

\begin{lemma}
\label{lem:early-late-approx} 
If~\eqref{ass:many-walks} holds then, for every $\eps>0$,
    $ \lim_{n \to \infty} \Pr ( D_n = \Gamma_\bmu (n, n^{(1/2)+\eps}) ) = 
    1$.
\end{lemma}
\begin{proof}
Fix $\eps \in (0,1/2)$, and define $\cC_\bmu^{(n)} \in \cK_0$ as an enlargement of $\cC_\bmu$ given by
$\cC_\bmu^{(n)} := \{ y \in \R^2 : \| x - y \| \leq n^{\eps - (1/2)} \text{ for some } x \in \cC_\bmu\}$.
Then $\rhoH ( \cC_\bmu^{(n)}, \cC_\bmu ) \leq n^{\eps -(1/2)}$ and   $0 \leq \diam \cC_\bmu^{(n)} - D_n \leq O ( n^{\eps -(1/2)})$.
The SLLN says that $n^{-1} \cH_n$ converges to $\cC_\bmu$.
Recall from Definition~\ref{def:rectangle} that
trajectory $\{ S_k^{(i)} : 0 \leq k \leq n\}$
is contained in a rectangle $R^{(i)}_n$ with major axis from $-m_n^{(i)} \hbmu_i$ to $M_n^{(i)}  \hbmu_i$ of  width $2 \oM_n^{(i)}$,
where $\oM_n^{(i)}$, $M_n^{(i)}$, and $m_n^{(i)}$ are defined at~\eqref{eq:rectangle-quantities}.
If $\bmu_i \neq \0$, then $-m_n^{(i)}$ converges a.s.~to $\min_{k \in \ZP} \hbmu_i^\tra S^{(i)}_k < \infty$, and so $m_n^{(i)}$ is tight,
as is $M_n^{(i)} - S^{(i)}_n$ by~\eqref{eq:duality}. 
Moreover, for every $\eps>0$, a.s., 
$\oM_n^{(i)} \leq n^{(1/2)+\eps}$ for all but finitely many $n \in \ZP$. 
If $\bmu_i = \0$, then also $m_n^{(i)} \leq n^{(1/2)+\eps}$
and $M_n^{(i)}  \leq n^{(1/2)+\eps}$. In any case, it follows that $\Pr ( \rhoH ( n^{-1} \cH_n , \cC_\bmu ) \leq n^{\eps-(1/2)} ) \to 1$ as $n \to \infty$.
Consequently,  
\begin{equation}
    \label{eq:scale-containment}
\lim_{n \to \infty} \Pr \bigl( n^{-1} \cH_n \subseteq \cC^{(n)}_\bmu \bigr) = 1 , \text{ for any } \eps >0.
 \end{equation}

Suppose $x, y \in n^{-1} \cH_n$ with $\| x -y \| >  \diam \cC_\bmu - n^{\eps-(1/2)}$;
by~\eqref{eq:scale-containment}, with probability $1-o(1)$, $n^{-1} D_n = \| x-y\|$ for such a pair $x,y$. 
Then there exist $x', y' \in \cC_\bmu$ with
$\| x - x'\| \leq n^{\eps-(1/2)}$, $\| y - y' \| \leq n^{\eps -(1/2)}$, and 
$\| x' - y ' \| \geq \diam \cC_\bmu - 3n^{\eps-(1/2)}$.
We may then apply the deterministic geometrical Lemma~\ref{lem:near-diameters} to the polygon $\cC_\bmu$
to conclude that $x, y$ are each within distance $Cn^{\eps -(1/2)}$ of opposite 
vertices corresponding to diametrical chords of $\cC_\bmu$. 

The possible diametrical chords are enumerated by $\cD_\bmu^+$, $\cD_\bmu^\circ$, and $\cD_\bmu^\times$.
Consider $(i,j) \in \cD_\bmu^+$. By definition $\bmu_i, \bmu_j \neq \0$, and the SLLN shows that, with high probability, only the steps  $n - n^{(1/2)+2\eps} \leq \ell \leq n$ (say) steps of the scaled random walk $n^{-1} S_\ell^{(i)}$ stay within distance $n^{\eps-(1/2)}$ of $\bmu_i$; similarly for walk~$j$. Thus
with high probability an upper bound for candidate diameters corresponding to $\cD_\bmu^+$
is given by $\Gamma_\bmu^+ (n, n^{(1/2)+2\eps})$ as defined at~\eqref{eq:gamma-def-1}. 
For $(0,j) \in \cD_\bmu^\circ$ a similar argument shows that an upper bound for candidate diameters corresponding to $\cD_\bmu^\circ$
is given by $\Gamma_\bmu^\circ (n, n^{(1/2)+2\eps})$ as defined at~\eqref{eq:gamma-def-2} (which includes early as well as late steps of the walk $j$, since these early steps can be close to $\0$).
Lastly, an upper bound for candidate diameters corresponding to $\cD_\bmu^\times$
is given by $\Gamma_\bmu^\times (n)$ as defined at~\eqref{eq:gamma-def-3},
since the whole $n$-step trajectory of any zero-drift walk can be close to $\0$.
Thus we get, for every $\eps>0$, $\Pr ( D_n \leq \Gamma_\bmu (n, n^{(1/2)+\eps} ) ) \to 1$ as $n \to \infty$,
which gives the result, since   $D_n \geq  \Gamma_\bmu (n, n^{(1/2)+\eps} )$.
\end{proof}

\begin{proof}[Proof of Proposition~\ref{prop:diam-approx}]
Lemma~\ref{lem:early-late-approx} shows that, with high probability,
$D_n = \Gamma_\bmu (n, n^{(1/2)+\eps})$ defined at~\eqref{eq:gamma-def-4}. 
We claim that, for every $\eps \in (0,1/2)$,
\begin{equation}
    \label{eq:gamma-approx-claim}
\lim_{n \to \infty} \Pr \bigl( \bigl| \Gamma_\bmu (n, n^{(1/2)+\eps}) - \Gamma_\bmu (n, 0) \bigr| \leq n^{(1/4)+\eps} \bigr) = 1.
\end{equation}
Given the claim~\eqref{eq:gamma-approx-claim}, from Lemma~\ref{lem:early-late-approx} we get that $| D_n - \Gamma_\bmu (n,0) | \leq n^{(1/4)+\eps}$ with high probability. Moreover, since $\Gamma_\bmu (n,0) \leq \diam \cA_n \leq D_n$, we also get $| D_n - \diam \cA_n | \leq n^{(1/4)+\eps}$ with high probability,
as required. 
Thus to prove
Proposition~\ref{prop:diam-approx}, it remains to verify~\eqref{eq:gamma-approx-claim}.
This is done using a similar approach to~\S\ref{sec:diam-one-walk}, with a modified
version of the bounding rectangles of Definition~\ref{def:rectangle}.
For example, to bound $| \Gamma_\bmu^+ (n,m) - \Gamma_\bmu^+ (n,0)|$
one can use a rectangle with axis along the line containing $S_n^{(i)}$
and $S_n^{(j)}$, which gives an $O(1)$ in probability bound in the axial direction and, in the orthogonal direction, a bound $O(n^{(1/4)+\eps})$ in probability, since one needs consider only $n^{(1/2)+\eps}$ steps near each endpoint. Together with similar arguments for the other terms, this verifies~\eqref{eq:gamma-approx-claim}.
\end{proof}

\begin{proof}[Proof of Theorem~\ref{thm:diameter-approximation-general}]
We first use a uniform square-integrability idea,  similar
to Lemma~\ref{lem:H-V-UI}. Indeed, take $j,k \in J_\bmu$ with $k \neq j$. Then
$\diam \cC_\bmu = \| \bmu_j - \bmu_k \|$. Also, by~\eqref{eq:sum-mean-ui} applied to each coordinate, 
$(n^{-1} \| S_n^{(k)} - S_n ^{(j)} - n ( \bmu_k - \bmu_j ) \|^2)_{n\in\N}$ is uniformly
integrable, and, by the triangle inequality $(\| a\| - \| b\|)^2 \leq \| a-b\|^2$, so that $(n^{-1} ( \diam  \{ S_n^{(k)}, S_n^{(j)} \} - n \diam \cC_\bmu )^2)_{n\in\N}$ is also uniformly integrable. But since $(n^{-1} \rhoH ( n \cC_\bmu, \cH_n)^2)_{n\in\N}$ is uniformly integrable, by Lemma~\ref{lem:H-V-UI}, and the diameter function is Lipschitz with respect to the Hausdorff metric,
 $(n^{-1} ( D_n - n \diam \cC_\bmu)^2)_{n\in\N}$ is also uniformly integrable, and hence 
so also is $(n^{-1} (  D_n - \diam  \{ S_n^{(k)}, S_n^{(j)} \} )^2)_{n\in\N}$,  by the triangle inequality.
But since
$\{ S_n^{(k)}, S_n^{(j)} \} \subseteq \cA_n \subseteq \cH_n$
and $\Gamma_\bmu (n,0) \geq \| S_n^{(k)} - S_n^{(j)} \|$ (recall the definition of $\Gamma_\bmu (n,0)$ from~\eqref{eq:gamma-def-4}), it follows that
\[ (n^{-1} (  D_n - \diam  \cA_n )^2)_{n\in\N}
\text{ and } 
(n^{-1} (  D_n - \Gamma_\bmu (n,0) )^2)_{n\in\N}
\text{ are both uniformly 
integrable.}\]
Combined with Proposition~\ref{prop:diam-approx}, which shows that, as $n \to \infty$, both $n^{-1/2} ( D_n - \diam \cA_n ) \to 0$ and  $n^{-1/2} ( D_n - \Gamma_\bmu (n,0) ) \to 0$ in probability, 
we verify Theorem~\ref{thm:diameter-approximation-general}\ref{thm:diameter-approximation-general-i}--\ref{thm:diameter-approximation-general-ii}. 
\end{proof}

\begin{proof}[Proof of Theorem~\ref{thm:many-walks-diameter}]
First we deduce the $L^2$ convergence statement in the theorem, involving one-dimensional projections, from that in Theorem~\ref{thm:diameter-approximation-general}\ref{thm:diameter-approximation-general-ii}, involving norms.   Lemma~\ref{lem:norm-projection} (and the fact that the sets $\cD_\bmu^+$ and $\cD^\circ_\bmu$ are finite) implies that
\[ 
\max\left( \frac{\max\limits_{(i,j) \in \cD^+_\bmu} \| S_n^{(i,j)} \| -
\max\limits_{(i,j) \in \cD^+_\bmu} \hbmu_{i,j}^\tra S_n^{(i,j)}}{\sqrt{n}} ,
    \frac{\max\limits_{(0,j) \in \cD^\circ_\bmu} \| S_n^{(j)} \| -\max\limits_{(0,j) \in \cD^\circ_\bmu} \hbmu_j^\tra S_n^{(j)}  }{\sqrt{n}} \right) \toL{n}{2} 0 ,
\]
and so the statement in the theorem will follow from the fact that, for every $(i,j) \in \cD^\times_\bmu$, 
\begin{equation}
    \label{eq:L2-ice-cream-diam}
 \frac{\max_{0 \leq k \leq n} \|   S_n^{(j)} -  S_k^{(i)} \|  - \max_{0 \leq k \leq n} \hbmu_j^\tra ( S^{(j)}_n - S_k^{(i)} )}{\sqrt{n}} \toL{n}{2} 0.
\end{equation}
The claim~\eqref{eq:L2-ice-cream-diam} is proved by a variation on the proof of Proposition~\ref{prop:two-to-one-diamater}.
Indeed, setting $D_n = \max_{0 \leq k \leq n} \|   S_n^{(j)} -  S_k^{(i)} \|$,
$D^\bmu_n = \max_{0 \leq k \leq n} \hbmu_j^\tra ( S_n^{(j)} - S_k^{(i)} )$,
and $\oMperp_n = \max_{0 \leq k \leq n} | \hbmu_\perp^\tra ( S_n^{(j)} - S_k^{(i)} ) |$, for $\hbmu_\perp$ a unit vector orthogonal to $\hbmu_j$, we have that the bounds in~\eqref{eq:box-bound} hold, and the rest of the proof of 
Proposition~\ref{prop:two-to-one-diamater} proceeds similarly with these new definitions. Thus we verify the $L^2$ limit statement. 

It remains to verify the distribution limit statement,
with   $\zeta$ and $\xi$ given at~\eqref{eq:zeta-xi-notation}.
First observe that, by definition, for every $(i,j) \in \cD^+_\bmu$, $\| \bmu_{i,j} \| = \diam \cC_\bmu$, for every $(i,j) \in \cD^\times_\bmu$, $\| \bmu_j \| = \diam \cC_\bmu$, and for every $(0,j) \in \cD^\circ_\bmu$, $\| \bmu_j \| = \diam \cC_\bmu$. Hence
\begin{align*}
    \hbmu_{i,j}^\tra S_n^{(i,j)} & = n \diam \cC_\bmu + \bigl( \hbmu_{i,j}^\tra ( S_n^{(i,j)} - \Exp S_n^{(i,j)} ) \bigr) , ~ (i,j) \in \cD^+_\bmu,\\
    \hbmu_j^\tra S_n^{(j)} & = n \diam \cC_\bmu + \bigl( \hbmu_{j}^\tra ( S_n^{(j)} - \Exp S_n^{(j)} ) \bigr) , ~ (i,j) \in \cD^\times_\bmu,\\
    \hbmu_j^\tra S_n^{(j)} + \max_{0 \leq k \leq n} ( - \hbmu_j^\tra S_k^{(i)} ) & = n \diam \cC_\bmu + \bigl( \hbmu_{j}^\tra ( S_n^{(j)} - \Exp S_n^{(j)} ) \bigr) 
    + \max_{0 \leq k \leq n} ( - \hbmu_j^\tra S_k^{(i)} ), ~ (i,j) \in \cD^\circ_\bmu.
\end{align*}
The result now follows from the $n \to \infty$ joint distributional convergence of the collection (for all $i,j$) of sequences $n^{-1/2} \hbmu^\tra_{i,j} ( S_n^{(i,j)} - \Exp S_n^{(i,j)})$, $n^{-1/2} \hbmu^\tra_j ( S_n^{(j)} - \Exp S_n^{(j)} )$ (by the CLT), and $n^{-1/2} \max_{0 \leq k \leq n} ( - \hbmu_j^\tra S_k^{(i)} )$ (by Donsker's theorem and continuous mapping). Finally, convergence of the variance follows from the $L^2$ approximation and the fact that the variance of the $L^2$ approximant, scaled by $n$, converges to $\Var \Delta$.
\end{proof}

\begin{proof}[Proof of Corollary~\ref{cor:unique-diameter}]
Let $k_1 < k_2$. If $k_1 = 0$, then $\mu_{k_1} = \0$ and $\mu_{k_2} \neq \0$, so $\cD_\bmu^\circ = \{ (k_1, k_2)\}$ and $\cD_\bmu^\times = \cD_\bmu^+ = \emptyset$. If $k_1 >0$ then we cannot have $\mu_{k_1} = \0$ (or else $J_\bmu$ would have to contain $0$ as well as $k_1, k_2$), so 
$\cD_\bmu^+ = \{ (k_1, k_2)\}$ and $\cD_\bmu^\times = \cD_\bmu^\circ = \emptyset$. Thus Theorem~\ref{thm:many-walks-diameter} implies that the standardized diameter converges in distribution to~$\zeta_{k_1,k_2} \sim \cN (0, \sigma_{k_1,k_2}^2)$.
\end{proof}

Finally, we can complete the proof of Corollary~\ref{cor:diam_two_walks} on non-Gaussian limits for the case $N=2$.

\begin{proof}[Proof of  Corollary~\ref{cor:diam_two_walks}.]
The result follows from Theorem~\ref{thm:many-walks-diameter} on noting the values of $\cD_\bmu^\circ, \cD_\bmu^\times, \cD_\bmu^+$ in each case. For case (i) $\cD_\bmu^\circ = \{ (0,1), (0,2)\}$, $\cD_\bmu^+ = \cD_\bmu^\times = \emptyset$; for case (ii) $\cD_\bmu^\circ = \{ (0,1) \}$, $\cD_\bmu^+ = \{ (1,2)\}$, $\cD_\bmu^\times = \emptyset$; for case (iii) $\cD_\bmu^\circ = \{ (0,1), (0,2) \}$, $\cD_\bmu^+ = \{ (1,2)\}$, $\cD_\bmu^\times = \emptyset$; and finally, for case (iv) $\cD_\bmu^\circ = \{ (0,2) \}$, $\cD_\bmu^\times = \{ (1,2)\}$,
$\cD_\bmu^+ = \emptyset$.
    \end{proof}

\section{Fluctuations of the perimeter for multiple random walks}
\label{sec:wald-based}

\subsection{Notation and main result}
In this section, we establish distributional convergence (Theorem~\ref{thm:perim_main_theorem}) results for the perimeter $L_n$ of the convex hull of $N \geq 1$ random walks  under hypothesis~\eqref{ass:many-walks}, covering all cases and with quite explicit limit distributions, which, as for the diameter, are generically Gaussian but non-Gaussian in certain cases.  As an illustration, the case $N=2$ is presented in~\S\ref{subsec:perim-applic}.
Before we formulate the main theorem, we have to introduce some additional notation. 

Recall from \S\ref{sec:diam-one-walk-result} that $\cE$ denotes the extreme points of $\cC_\bmu$.
Let 
\begin{equation}
\label{eq:cI-def}
\cI := \{1 \le k \le N : \bmu_k \in \cE\}
\end{equation}
be the set of indices of random walks with whose drift vectors are extreme points of $\cC_{\bmu}$. Observe that $\cI \neq \cJ$ in general, where $\cJ$ was introduced in Subsection \ref{sec:hausdorff-result}. 
We further split $\cI = \cI^+ \cup \cI^0$, where
$ \cI^+:=\{k \in \cI: \bmu_k\neq \0\}$ and $\cI^0:=\{k\in \cI: \bmu_k=\0\}$.
Define
\[
\Theta_0:=\Big\{\theta\in[0,2\pi]: \max_{k \in \cI^+} \be_\theta^\tra \bmu_k <0\Big\}.
\]
Note that geometrically $\Theta_0$ consists of directions for which the supporting point of \(\cC_\mu\) is the origin.
For each non-zero extreme drift $\bmu\in \cE$ define\footnote{For any $A\subseteq \R^2$, $\mathrm{int\,}A$ denotes the interior of $A$.}
\[
\Theta_\bmu:=\mathrm{int\,}\Big\{\theta\in[0,2\pi):\ \be_\theta^\tra \bmu=\max_{\bmu'\in \cE} \be_\theta^\tra \bmu'\Big\},
\qquad
\cI_\bmu:=\{k\in \cI:\ \bmu_k=\bmu\}.
\]
Note that evidently $[0,2\pi ]\setminus \bigl(\Theta_0 \cup \bigcup_{\bmu \in \cE \setminus\{0\}}\Theta_{\bmu}\bigr)$ is finite and so of Lebesgue measure zero. 

 Let $(X^{(k)})_{k=1}^N$ be a family of independent centred Gaussian vectors in $\R^2$ such that
$X^{(k)}\sim \cN(0,\Sigma_k)$,
where the covariance matrix $\Sigma_k$ comes from \eqref{ass:many-walks}.
Let $(B^{(k)})_{k \in \cI^0}$ be a family of independent planar Brownian motions with covariance matrices $\Sigma_k$, independent of the $X^{(k)}$. The following theorem, the formal version of~\eqref{eq:perimeter-limit},
 is the main result of this section.

\begin{theorem}\label{thm:perim_main_theorem}
   Under assumption \eqref{ass:many-walks} and  in the above notation, it holds
    \begin{equation}
        \label{eq:perim-main-limit}
\frac{L_n-n\perim \cC_\bmu}{\sqrt{n}}
 \tod{n} \Pi^+ + \Pi^0,
\end{equation}
where $\Pi^+$ and $\Pi^0$ are independent random variables given by
\begin{equation}
\label{eq:Pi-defs}
\Pi^+
:=
\sum_{\bmu\in \cE\setminus\{\0\}}
\int_{\Theta_\bmu}
\max_{k\in \cI_\bmu}\big( \be_\theta^\tra X^{(k)} \big)\,\ud\theta, \text{ and }
\Pi^0 := \int_{\Theta_0}
\max_{k\in \cI^0}\sup_{0\le t\le 1}\big(\be_\theta^\tra B^{(k)}_t \big)\,\ud\theta.
\end{equation}
Moreover, $\lim_{n\to \infty} n^{-1} \Var L_n = \Var \Pi^+ + \Var \Pi^0$.
\end{theorem}

\begin{remarks}
\label{rems:perimeter}
\begin{myenumi}[label=(\alph*)]
\item\label{rems:perimeter-a}
The random variable $\Pi^+$ describes the fluctuations contributed by non-zero extreme drifts, whereas $\Pi^0$ corresponds to directions in which the origin is the supporting point of $\cC_\bmu$. Both quantities have geometrical interpretations in terms of \emph{partial Cauchy formulae} that we describe in Appendix~\ref{sec:semi-cauchy}.
\item\label{rems:perimeter-b}
Let $N=1$. If $\bmu_1 = \0$, then $\perim \cC_\bmu =0$, $\Pi^+ = 0$, and $\Pi^0 = \int_0^{2\pi} \sup_{0 \leq t \leq 1} \be_\theta^\tra B^{(1)} (t) \ud \theta$ is the perimeter of planar Brownian motion with covariance matrix $\Sigma_1$, in which case Theorem~\ref{thm:perim_main_theorem} recovers the Donsker-type convergence result from~\cite{wx-scaling}. Otherwise,  $\bmu_1 \neq \0$. Then, $\perim \cC_\bmu = 2 \| \bmu_1\|$, $\Pi^0 = 0$, and $\Pi^+ = \int_{\Theta_{\bmu_1}} \be_\theta^\tra X^{(1)} \ud \theta = 2 \hbmu_1^\tra X^{(1)}$ so that $\Pi^+ \sim \cN ( 0, 4\sigma_1^2)$ with the notation of~\eqref{eq:spara-def-two}. Thus  Theorem~\ref{thm:perim_main_theorem} recovers the $N=1$ result from~\cite{wx-drift} quoted in~\eqref{eq:known-D-L-CLT} above.
\end{myenumi}
\end{remarks}

For $N=2$, Theorem~\ref{thm:perim_main_theorem} completes the picture 
that had been resolved in the Gaussian cases in~\cite{ikss,tomislav}, as we describe in \S\ref{subsec:perim-applic} below.
First we give a corollary that demonstrates that the generic picture is Gaussian.

\begin{corollary}
\label{cor:zero-in-interior-clt}
     Suppose that \eqref{ass:many-walks} holds, that  $\0$ is in the interior of $\cC_\bmu$,
      and that all $\bmu_k$, $k \in \cI$ are distinct. Without loss of generality, relabel the $\bmu_k$s so that $\cI = \{ 1, \ldots, m \}$ enumerated in anticlockwise order around $\partial \cC_\bmu$, and interpret index  $m+1$ as $1$. Then
      \[ \frac{L_n - n \sum_{k=1}^m \| \bmu_{k,k+1} \| }{\sqrt{n}} \tod{n} 
      \cN \left( 0 , \sum_{k=1}^m ( \hbmu_{k,k+1} - \hbmu_{k-1,k} )^\tra \Sigma_k ( \hbmu_{k,k+1} - \hbmu_{k-1,k} ) \right). \]
\end{corollary}
\begin{remark}
\label{rem:perimeter-generic}
    The hypotheses on $\bmu_1, \ldots, \bmu_N$ in Corollary~\ref{cor:zero-in-interior-clt} exclude many cases that are covered by the completely general Theorem~\ref{thm:perim_main_theorem}, but are in some sense generic for large~$N$. Indeed, if, for example, $\bmu_1, \ldots, \bmu_N$ are chosen randomly according to the uniform distribution on the unit circle, then $\0$ is in the interior of $\cC_\bmu$ with probability $1 - N 2^{1-N}$ (a special case solved by D.A.~Darling of a more general problem solved by Schl\"afli and Wendel~\cite{wendel}). 
\end{remark}

\subsection{Application for two random walks}\label{subsec:perim-applic}
In this subsection we apply Theorem \ref{thm:perim_main_theorem} to three specific examples for two independent planar random walks. The first one recovers the result \cite[Theorem 2.6]{ikss}, and the other two examples explain the border cases that were postulated in \cite{ikss} as potentially non-Gaussian, but were not treated there. In one of them we have one zero-drift, and one non-zero drift walk, and in the other one we have two non-zero drift random walks with the same drift vector. It turns out that the limits are indeed non-Gaussian.

\begin{example}[Two walks with linearly independent drifts]
	Assume $N=2$ and that the drift vectors $\bmu_1,\bmu_2\in\mathbb{R}^2$ are linearly independent.
	Then
	$\cC_\bmu=\hull\{\0,\bmu_1,\bmu_2\}$
	is a non-degenerate triangle whose set of extremal points is
	$\cE=\{\0,\bmu_1,\bmu_2\}$.
	In particular, there are no zero-drift walks, so $\cI^0=\emptyset$. For $\theta\in[0,2\pi)$ consider the projections $\be_\theta^\tra \bmu_1$ and $\be_\theta^\tra \bmu_2$.
	There are three open subsets of directions, defined as follows:
	\begin{align*}
		\Theta_{\bmu_1}
		&:=
		\bigl\{\theta\in[0,2\pi):\ \be_\theta^\tra \bmu_1 > \be_\theta^\tra \bmu_2 \ \text{and}\ \be_\theta^\tra \bmu_1>0\bigr\},\\
		\Theta_{\bmu_2}
		&:=
		\bigl\{\theta\in[0,2\pi):\ \be_\theta^\tra \bmu_2> \be_\theta^\tra \bmu_1 \ \text{and}\ \be_\theta^\tra\bmu_2>0\bigr\},\\
		\Theta_0
		&:=
		\bigl\{\theta\in[0,2\pi):\ \max( \be_\theta^\tra \bmu_1, \be_\theta^\tra \bmu_2)<0\bigr\}.
	\end{align*}
	The remaining directions, where $\be_\theta^\tra \bmu_1 = \be_\theta^\tra \bmu_2$ or where the maximum equals $0$,
	form a finite set and are therefore negligible for integration purposes. Geometrically, $\Theta_{\bmu_1}$ and $\Theta_{\bmu_2}$ are the interiors of the normal cones at the vertices
	$\bmu_1$ and $\bmu_2$ of the triangle $\cC_\bmu$, while $\Theta_0$ is the interior of the normal cone at the vertex $0$. Let $L_n$ be the perimeter of the convex hull generated by the full trajectories of both walks.
	Then Theorem \ref{thm:perim_main_theorem} applies with $\Pi_0 = 0$ and
	\[
	\Pi^+
	=
	\int_{\Theta_{\bmu_1}} ( \be_\theta^\tra X^{(1)})\,\ud\theta
	+
	\int_{\Theta_{\bmu_2}} ( \be_\theta^\tra X^{(2)})\,\ud\theta,
	\]
	for $X^{(k)}\sim\mathcal{N}(0,\Sigma_k)$ independent Gaussian vectors, and $\Pi^0 = 0$. Thus the limit reduces to the purely drift-dominated Gaussian term
	\[
	\Pi^+=\ba_1^\tra X^{(1)} + \ba_2^\tra X^{(2)}  ,
	\qquad
	\ba_i:=\int_{\Theta_{\bmu_i}} \be_\theta\,\ud\theta,
	\]
so
	$\Pi^+\sim\cN(0,\ \ba_1^\tra\Sigma_1 \ba_1 + \ba_2^\tra\Sigma_2 \ba_2 )$,
 the Gaussian limit from \cite[Theorem 2.6]{ikss}.
\end{example}

\begin{example}[Two walks, one zero-drift, one non-zero drift]\label{ex:perim_ice-cream}
		Without loss of generality we assume that the non-zero drift vector points in the positive $y$-direction. More precisely, take $\bmu_1 = \0$ and $\bmu_2=\bmu=(0,\|\bmu\|)$. 
	The drift polygon is
	$\cC_\bmu=\hull \{\0,\bmu\}$,
	which is simply a line segment of length $\|\bmu\|$.  In this degenerate case the
	perimeter equals twice the length of the segment, i.e.\ $\perim \cC_\bmu = 2\|\bmu\|$. Since
	$\be_\theta^\tra \bmu  = \|\bmu\|\sin\theta$,
	we have $\Theta_\bmu=(0,\pi)$ and $\Theta_0=(\pi,2\pi)$,
	up to the boundary angles $0$ and $\pi$ where $ \be_\theta^\tra \bmu=0$. Theorem~\ref{thm:perim_main_theorem}  gives
	\[
	\Pi^+ = \int_0^\pi (\be^\tra_\theta X^{(2)} )\,\ud\theta,
	\]
	where $X^{(2)}\sim\mathcal N(0,\Sigma_2)$. Also
	\[
 \int_0^\pi \be_\theta\,\ud\theta
	= \int_0^\pi (\cos\theta,\sin\theta)\,\ud\theta = 2 \hbmu.
	\]
Thus, with $\sigma_2^2$ as defined at~\eqref{eq:spara-def-two},
	$\Pi^+ = 2 \hbmu^\tra X^{(2)}  \sim \cN ( 0, 4 \sigma_2^2 )$.
	For the zero-drift contribution, 
	\[
	\Pi^0 = \int_\pi^{2\pi}
	\sup_{0\le t\le1}\big( \be_\theta^\tra B^{(1)}(t) \big)
	\,\ud\theta,
	\]
	where $B^{(1)}$ is planar Brownian motion with covariance matrix $\Sigma_1$. Denoting by $\mathcal G := \hull (B^{(1)}(t):0\le t\le 1)$,
	and using Lemma \ref{lem:partial_cauchy}, we have
$\Pi^0 = \mathfrak{m}(\mathcal G) - d_\ell - d_r$,
	where $\mathfrak{m}(\mathcal G)$ is the length of the boundary arc of the hull
	between the leftmost and rightmost points (the ``convex minorant'' arc), and $d_\ell$ and $d_r$ are the signed horizontal distances from those
	points to the $x$-axis. If the leftmost or rightmost point is not unique, we can choose arbitrarily. 
    Thus,
	\[
	\frac{L_n-2\|\bmu\|\,n}{\sqrt n}
	\tod{n} 2 \hbmu^\tra X^{(2)}  + \mathfrak{m}(\mathcal G) - d_\ell - d_r,
	\]
	where $2 \hbmu^\tra X^{(2)} \sim \cN (0, 4 \sigma_2^2 )$ is independent of $\mathfrak{m}(\mathcal G) - d_\ell - d_r$.
\end{example}

\begin{example}[Two random walks whose drift vectors are identical and non-zero]\label{ex:perim_two-the-same}
	Without loss of generality, we assume
$\bmu_1 = \bmu_2 = \bmu = (0, \|\bmu\|)$.
	As in the previous example, the drift polygon is a line segment $\cC_\bmu = \hull\{\0,\bmu\}$, of perimeter $2\|\bmu\|$, $\Theta_\bmu=(0,\pi)$ and $\Theta_0=(\pi,2\pi)$. Theorem \ref{thm:perim_main_theorem}  yields $\Pi^0 = 0$ and 
	\begin{align*}
		\Pi^+
		& = \int_0^\pi \max_{1 \leq k \leq 2} (\be_\theta^\tra X^{(k)} ) \,\ud\theta .
	\end{align*}
	Define
$U:= X^{(1)} + X^{(2)}$ and $V:= X^{(1)} - X^{(2)}$.
	Using $\max\{a, b\} = (a+b)/2 + |a-b|/2$ we get
	\begin{equation*}
		\Pi^+ = \frac{1}{2} \int_0^{\pi} \be_\theta^\tra U  \,\ud \theta + \frac{1}{2} \int_0^{\pi} | \be_\theta^\tra V| \,\ud\theta = \hbmu^\tra U + \|V\|,
	\end{equation*}
	where we used that $\int_0^{\pi} \be_\theta\, \ud \theta = (0, 2)$, $\int_0^{\pi}|\be_\theta^\tra v| \ud \theta = 2\|v\|$ for any $v \in \R^2$.
    Notice also that,
	\begin{equation*}
		\Exp [ U U^\tra ] = \Sigma_1 + \Sigma_2, \quad \Exp [ V V^\tra ] = \Sigma_1 + \Sigma_2, \quad \Exp [ U V^\tra ] = \Sigma_1 - \Sigma_2.
	\end{equation*}
	Hence, we have
	\[
	\frac{L_n-2n\|\bmu\|}{\sqrt n}
	\tod{n} \hbmu^\tra U + \|V\|,
	\]
	where
	$U$ and $V$ are independent when $\Sigma_1 = \Sigma_2$. Notice that in the case $\Sigma_1 = \Sigma_2 = I$ (identity), $ \hbmu^\tra U \sim \mathcal{N}(0, 2)$, and recall that $\|V\|$ has a so-called Rayleigh distribution with parameter $\sqrt{2}$.
In particular, $\Var(U_2)=2$ and $\Var(\|V\|)=4-\pi$,
	so that, in that case,  $\Var(\Pi^+)
	=
	6-\pi$. \qedhere
\end{example}

\subsection{Perimeter deviation metric}\label{sec:perimeter_deviation_metric}

Our second main result on perimeter for multiple random walks is a set approximation result that simplifies  further the 
approximating set from Theorem~\ref{thm:general-hausdorff}, where the approximation was in the Hausdorff sense, 
by using 
the $\ell_1$ distance between support functions, which is exactly adapted to perimeter and generates a weaker topology.

Recall that $\cK_0$ denotes the set of compact, convex subsets of $\R^2$ containing~$\0$. The \emph{support function} of $K \in \cK_0$ is $h_K : \Sp{}  \to \RP$ ($\Sp{} := \{ \be \in \R^2 : \| \be \| =1\}$) given by
\begin{equation}
    \label{eq:support-function}
h_K (\be) := \sup_{x \in K} (\be^\tra x), \quad \text{for all } \be \in \Sp{}.
\end{equation}  
For $\theta\in[0,2\pi]$, set $\be_\theta:=(\cos\theta,\sin\theta)$. Cauchy's formula (see \cite{schneider_book}) expresses the perimeter in terms of the support function, more precisely
\begin{equation*}
    \perim(K) = \int_0^{2\pi} h_K(\be_\theta) \ud \theta.
\end{equation*}
The $\ell_1$ metric for the support function, defined for $A,B \in \cK_0$ by
\[ \rhoone ( A, B) := \int_0^{2\pi} | h_A (\be_\theta) - h_B (\be_\theta)| \ud \theta \]
is a natural alternative to $\rhoH$, since, by Cauchy's formula, perimeter is Lipschitz in this metric. Since $| x -y | = 2 \max (x,y) - x -y$, the metric~$\rhoone$ can also be expressed as
\[
    \rhoone (A,B) = 
    2 \perim \hull (A \cup B) - \perim A - \perim B,\quad \text{ for } A,B \in \cK_0,
\]
which  is also known as the \emph{perimeter deviation metric}~\cite{florian}.

Recall the definition of $\cI^+$ and $\cI^0$ from below~\eqref{eq:cI-def}.
Define 
\[ \cF_n := \hull \left( \{ \0 \} \cup \left( \bigcup_{k \in \cI^+} \{ S_n^{(k)} \} \right) \cup  \left( \bigcup_{k \in \cI^0} \{ S_i^{(k)} : 0 \leq i \leq n \} \right)  \right),\]
the convex hull generated by the endpoints of walks whose non-zero drift vectors are extreme points of $\cC_\bmu$, together with the full trajectories of zero-drift walks when $\0$ is an extreme point of $\cC_\bmu$. The following result shows that, at the $\sqrt n$ scale, the full hull $\cH_n$ may be approximated  in the perimeter deviation metric by the reduced hull $\cB_n$.
Note that, compared to the Hausdorff approximant $\cG_n$ as defined at~\eqref{eq:g-def}, $\cB_n$ discards endpoints of walks corresponding to boundary but not extreme points, and trajectories of walks corresponding to faces of $\cC_\bmu$.

\begin{theorem}\label{thm:perimeter-deviation-metric}
Under assumption~\eqref{ass:many-walks},
\[
\frac{\rhoone(\cH_n,\cF_n)}{\sqrt n}\toL{n}{2}0.
\]
\end{theorem}
The proof of Theorem~\ref{thm:perimeter-deviation-metric} is given in \S\ref{sec:perim-proofs} along with the proof of
Theorem~\ref{thm:perim_main_theorem}.

\subsection{Proofs}
\label{sec:perim-proofs}
In the forthcoming subsections we present the proof of Theorem \ref{thm:perim_main_theorem} split into several parts.
Our approach is to iteratively approximate the perimeter $L_n$ of the joint convex hull $\cH_n$ (in $L^2$ sense, and with $\sqrt{n}$ scaling) by simpler quantities, until we reach an object with a transparent probabilistic interpretation from which the distributional limit follows naturally.

\subsubsection*{First approximation}\label{sec:first_aproximation}
For each $k \in [N]$\footnote{We use notation $[N] = \{1, 2, \ldots, N\}$, for any $N\in \mathbb{N}$.}, let
\begin{align*}
    P_n^{(k)} := \begin{cases}
        \{S_n^{(k)}\} & \text{if $\bmu_k \ne \0 $}, \\
        \{S_j^{(k)}\}_{0\leq j \leq n  } & \text{otherwise,}
    \end{cases}
\end{align*}
 the \emph{reduced path} of the $k$-th random walk: we keep the whole trajectory if the drift is zero, but otherwise only record its endpoint. By convention we also set $P_n^{(0)} := \{\0\}$. Let now
\begin{equation*}
    \cH_n^{(2)} := \hull \left( \bigcup_{k=0}^N P_n^{(k)} \right)
\end{equation*}
be the convex hull generated by the reduced paths. The perimeter of $\cH_n^{(2)}$ is denoted by $L_n^{(2)}$. Let $h_n^{(1)}(\theta) = h_{\cH_n}(\be_\theta)$
be the support function of the convex hull $\cH_n$, and let $h_n^{(2)}(\theta)$ be the support function of the convex hull $\cH_n^{(2)}$. In particular, $\0\in \cH_n^{(2)}$, so $h_n^{(2)}(\theta)\ge0$. For convenience, we write $L_n^{(1)} = L_n$. According to Cauchy's formula for the perimeter we have
\begin{align}\label{eq:Cauchy-integral-perimeter}
    L_n^{(j)} = \int_0^{2\pi} h_n^{(j)}(\theta) \ud \theta, \qquad j \in\{1,2\}.
\end{align}
Our goal now is to show that in $L^2$ sense, and with $\sqrt{n}$ scale, the perimeter of $\cH_n$ can be replaced with the perimeter of the convex hull of the ``reduced path'' processes, i.e.\ that
\begin{equation*}
    \frac{L_n^{(1)} - L_n^{(2)}}{\sqrt{n}} \toL{n}{2} 0.
\end{equation*}
Define
\[
    r_n(\theta):=\frac{h_n^{(1)}(\theta)-h_n^{(2)}(\theta)}{\sqrt n},
\]
and let
$\Theta_{\perp}:=\{\theta\in[0,2\pi):\exists k \text{ with } \bmu_k\neq 0 \text{ and } \be_{\theta}^\tra \bmu_k   = 0\}$.
Notice that $\Theta_{\perp}$ is finite (hence of Lebesgue measure zero).
We will first show that
\begin{equation}\label{eq:1st-appr-r_n-1}
    \sup_{\theta\in[0,2\pi]}\ \sup_{n\in\N}\ \Exp\big[r_n(\theta)^2\big] < \infty,
\end{equation}
and then that for every $\theta\in[0,2\pi)\setminus\Theta_{\perp}$ we have
\begin{equation}\label{eq:1st-appr-r_n-2}
    r_n(\theta)\toL{n}{2}0.
\end{equation}
Consequently, by the dominated convergence theorem we conclude that
\begin{equation}\label{eq:first-approx-conv}
\frac{L_n^{(1)}-L_n^{(2)}}{\sqrt n}
=
\int_0^{2\pi} r_n(\theta)\,\ud\theta
\ \toL{n}{2}\ 0.
\end{equation}
For $\theta\in[0,2\pi]$ and $j\in\ZP$, set 
\[
S^{(k)}_j(\theta):=\be_\theta^\tra S^{(k)}_j.
\]
Fix $n\in\N$ and $\theta\in[0,2\pi]$. For each $k$, write
\begin{equation}\label{eq:def_of_dk}
    d_k(\theta):=\be_\theta^\tra \bmu_k,
\qquad
M^{(k)}_n(\theta):=\max_{0\le j\le n} S^{(k)}_j(\theta).
\end{equation}
Then
$h_n^{(1)}(\theta)=\max_{1\le k\le N} M^{(k)}_n(\theta)$. 
Since $\cH_n^{(2)}$ contains the full trajectory of every zero-drift walk
and also contains the origin, we have deterministically
\[
0\le h_n^{(1)}(\theta)-h_n^{(2)}(\theta)
\le
\max_{k:\,\bmu_k\neq\0}
\Big(M^{(k)}_n(\theta)-\max\{0,S^{(k)}_n(\theta)\}\Big).
\]
Define, for $\bmu_k\neq\0$,
$E^{(k)}_n(\theta)
:=M^{(k)}_n(\theta)-\max\{0,S^{(k)}_n(\theta)\}  \ge  0$.
Then
\[
r_n(\theta)^2
\le
\frac{1}{n}\sum_{k:\,\bmu_k\neq\0}\big(E^{(k)}_n(\theta)\big)^2,
\]
so to prove~\eqref{eq:1st-appr-r_n-1} it suffices to show that for each fixed $k$,
\begin{equation}\label{eq:Ek-L2-goal}
\sup_{\theta\in[0,2\pi]}\ \sup_{n\in\N}\ \frac{1}{n}
\Exp\big[(E^{(k)}_n(\theta))^2\big]<\infty.
\end{equation}
Define the centred process
\[
Y^{(k)}_j(\theta):=S^{(k)}_j(\theta)-jd_k(\theta).
\]
Then $\{Y^{(k)}_j(\theta)\}_{j\in\ZP}$ is a martingale with $Y^{(k)}_0(\theta)=0$. We treat $E^{(k)}_n(\theta)$ by splitting on the sign of $S^{(k)}_n(\theta)$.
If $S^{(k)}_n(\theta)\ge0$, then
 $E^{(k)}_n(\theta)=M^{(k)}_n(\theta)-S^{(k)}_n(\theta)$. We have
\[
M^{(k)}_n(\theta)-S^{(k)}_n(\theta)
=\max_{0\le j\le n}\big( Y^{(k)}_j(\theta)-Y^{(k)}_n(\theta) + (j-n)d_k(\theta)\big).
\]
Now, if $d_k(\theta)\ge0$, then $(j-n)d_k(\theta)\le0$ for all $j\le n$. Hence,
\[
M^{(k)}_n(\theta)-S^{(k)}_n(\theta)
\le \max_{0\le j\le n}(Y^{(k)}_j(\theta)-Y^{(k)}_n(\theta))
\le \max_{0\le j\le n}|Y^{(k)}_j(\theta)|+|Y^{(k)}_n(\theta)|
\le 2\max_{0\le j\le n}|Y^{(k)}_j(\theta)|.
\]
On the other hand, if $d_k(\theta)<0$,  
since $S^{(k)}_n(\theta)=Y^{(k)}_n(\theta)+nd_k(\theta)\ge0$, we have $Y^{(k)}_n(\theta)\ge -nd_k(\theta)$. Thus for all $0\le j\le n$,
\[
(j-n)d_k(\theta)\le -nd_k(\theta) \le Y^{(k)}_n(\theta),
\]
and therefore $Y^{(k)}_j(\theta)-Y^{(k)}_n(\theta) +(j-n)d_k(\theta) \le Y^{(k)}_j(\theta)$. Hence
\[
M^{(k)}_n(\theta)-S^{(k)}_n(\theta) \le \max_{0\le j\le n}Y^{(k)}_j(\theta) \le \max_{0\le j\le n}|Y^{(k)}_j(\theta)|.
\]
Combining the two subcases, we obtain the uniform bound
\[
(M^{(k)}_n(\theta)-S^{(k)}_n(\theta))\mathbf 1_{\{S^{(k)}_n(\theta)\ge0\}}
\le 2\max_{0\le j\le n}|Y^{(k)}_j(\theta)|.
\]
Assume now that  $S^{(k)}_n(\theta)<0$.
Then $E^{(k)}_n(\theta)=M^{(k)}_n(\theta)$. If $d_k(\theta)\le0$, then $S^{(k)}_j(\theta)=Y^{(k)}_j(\theta)+jd_k(\theta)\le Y^{(k)}_j(\theta)$ for all $j\ge0$,
so $M^{(k)}_n(\theta)\le  \max_{0\le j\le n}|Y^{(k)}_j(\theta)|$.
If $d_k(\theta)>0$, then $S^{(k)}_n(\theta)<0$ implies $Y^{(k)}_n(\theta)<-nd_k(\theta)$. Hence, $nd_k(\theta)\le |Y^{(k)}_n(\theta)|\le\max_{0\le j\le n}|Y^{(k)}_j(\theta)|$,
and therefore
\[
M^{(k)}_n(\theta)=\max_{0\le j\le n}(Y^{(k)}_j(\theta)+jd_k(\theta))
\le \max_{0\le j\le n}|Y^{(k)}_j(\theta)|+nd_k(\theta)
\le 2\max_{0\le j\le n}|Y^{(k)}_j(\theta)|.
\]
Thus in all cases,
\[
E^{(k)}_n(\theta)\le 2\max_{0\le j\le n}|Y^{(k)}_j(\theta)|.
\]
By Doob's  maximal inequality we now conclude
\[
\Exp\Big[\max_{0\le j\le n}|Y^{(k)}_j(\theta)|^2\Big]
\le 4\,\Exp[Y^{(k)}_n(\theta)^2]
=4n\,\Var ( \be_\theta^\tra Z_{k,1} )
\le 4n\,\Exp[\|Z_{k,1}\|^2].
\]
Hence
\[
\frac{1}{n}\Exp\big[E^{(k)}_n(\theta)^2\big]
\le 16\,\Exp[\|Z_{k,1}\|^2],
\]
uniformly in $\theta$ and $n$. This proves \eqref{eq:Ek-L2-goal}, and 
\eqref{eq:1st-appr-r_n-1} follows.

Now we turn to showing that for every $\theta\in[0,2\pi)\setminus\Theta_{\perp}$ 
  \eqref{eq:1st-appr-r_n-2} holds true. 
Fix $\theta\notin\Theta_{\perp}$ and set 
\[
d_*(\theta):=\max_{1\le k\le N} d_k(\theta).
\]
Assume first that $d_*(\theta)<0$.
Then every projected walk $\{S_n^{(k)}(\theta)\}_{n\in\ZP}$ has negative drift $d_k(\theta)<0$.
By Theorem \ref{thm:WaldCLT}\ref{thm:WaldCLT-i},
$n^{-1/2} M_n^{(k)} (\theta) \to 0$ in $L^2$ for every $k \in [N]$.
It follows that
\[
\frac{h_n^{(1)}(\theta)}{\sqrt n}
=
\frac{1}{\sqrt n}\max_{k\in[N]} M_n^{(k)}(\theta)
\le
\sum_{k=1}^N \frac{1}{\sqrt n} M_n^{(k)}(\theta) 
\toL{n}{2}0.
\]
Using $0\le h_n^{(2)}(\theta)\le h_n^{(1)}(\theta)$, we get
$0\le r_n(\theta)\le h_n^{(1)}(\theta)/\sqrt n\to 0$ in $L^2$.

Next, if $d_*(\theta)=0$,
because we assumed $\theta\notin\Theta_{\perp}$, no non-zero drift can have zero projection.
Thus, if $d_*(\theta)=0$, the maximum drift projection can only come from walks with $\bmu_k=\0$ (i.e. truly zero drift in $\mathbb{R}^2$),
and all nonzero-drift walks satisfy $d_k(\theta)<0$.
In this case, $h_n^{(2)}(\theta)$ already includes the full trajectories of all zero-drift walks, and the negative-drift walks
do not contribute at $\sqrt n$ scale by Theorem \ref{thm:WaldCLT}\ref{thm:WaldCLT-i}. Indeed, we have
\[
\max_{k:\bmu_k=\0} M_n^{(k)} (\theta)
\le h_n^{(2)}(\theta)\le h_n^{(1)}(\theta)
\le \max\left\{\max_{k:\bmu_k=\0} M_n^{(k)} (\theta), \max_{k:\bmu_k\neq \0} M_n^{(k)} (\theta) \right\}.
\]
By Theorem \ref{thm:WaldCLT}\ref{thm:WaldCLT-i}, the second term on the right-hand side, when divided by $\sqrt{n}$, converges to zero in $L^2$. Hence $n^{-1/2} (h_n^{(1)}(\theta)-h_n^{(2)}(\theta)) \to 0$ in $L^2$ and consequently
 $r_n(\theta)\to 0$ in $L^2$.
 
Finally, assume $d_*(\theta)>0$.
We split the difference $h_n^{(1)}(\theta)-h_n^{(2)}(\theta)$ as
\begin{align*}
    h_n^{(1)}(\theta)-h_n^{(2)}(\theta)
    & = \max\left\{ \max_{k:d_k>0} M_n^{(k)} (\theta), \max_{k:d_k=0}M_n^{(k)} (\theta), \max_{k:d_k<0}M_n^{(k)} (\theta) \right\}\\
    & \qquad - \max\left\{ 0 , \max_{k:d_k>0}S_n^{(k)}(\theta), \max_{k:d_k=0}M_n^{(k)} (\theta), \max_{k:d_k<0}S_n^{(k)}(\theta) \right\} \\
    & \le  \Bigl( \max_{k:d_k>0}M_n^{(k)} (\theta) - \max_{k:d_k>0}S_n^{(k)}(\theta) \Bigr) +  
    \max_{k:d_k<0}M_n^{(k)} (\theta) .
\end{align*}
where we wrote $d_k = d_k(\theta)$ for short, and used that $h_2 (\theta) \geq 0$.
For the second term, in which $d_k(\theta)<0$, Theorem \ref{thm:WaldCLT}\ref{thm:WaldCLT-i} applies. For the first term, 
 observe the inequality
\begin{align*}
    \max_{k:d_k>0} M_n^{(k)} (\theta) -\max_{k:d_k>0}S_n^{(k)}(\theta)
    & \le
\max_{k:d_k>0}\Bigl(M_n^{(k)} (\theta) -S_n^{(k)}(\theta)\Bigr) .
\end{align*}
Theorem \ref{thm:WaldCLT}\ref{thm:WaldCLT-ii} now applies 
for each $k$ for which $d_k>0$, 
which finishes the proof of \eqref{eq:1st-appr-r_n-2}.
This verifies~\eqref{eq:first-approx-conv}.


\subsubsection*{Second approximation}\label{sec:second_aproximation}
We next show that we can ignore all the non-zero drift walks whose drift vectors are not extremal. Recall that $\cE$ is the set of extreme points of $\cC_{\bmu}$, and $\cI$ is the set of indices of random walks with drift vectors that are extreme points of $\cC_{\bmu}$. Let us define
\begin{equation}
\label{eq:H3-def}
    \cH_n^{(3)} := \hull \left( \bigcup_{k \in \cI \cup \{0\}} P_n^{(k)} \right).
\end{equation}
The difference between  $\cH_n^{(2)}$ and $\cH_n^{(3)}$ comes from  endpoints $S_n^{(k)}$ with $\bmu_k\notin \cE$.
Analogously as before, the corresponding support function is denoted by $h_n^{(3)}(\theta)$, and the perimeter of the set $\cH_n^{(3)}$ is denoted by $L_n^{(3)}$. We aim to show that
\begin{equation}\label{eq:2nd-approx-conv}
    \frac{L_n^{(2)} - L_n^{(3)}}{\sqrt{n}} \toL{n}{2} 0.
\end{equation}
Since $\cC_{\bmu}$ is the convex hull of finitely many points in $\mathbb{R}^2$, it is a (possibly degenerate) polygon. For $\theta\in[0,2\pi]$ consider the linear functional $v\mapsto \be_\theta^\tra v$ on $\cC_{\bmu}$. Either it is maximized at a unique point of $C$ (necessarily an extreme point), or it is maximized along an entire edge. The set of directions $\theta$ for which the maximizer is not unique corresponds to outward normals to edges of $\cC_{\bmu}$. Since $\cC_{\bmu}$ has finitely many edges, there are only finitely many such directions. Let $\Theta_*\subset[0,2\pi]$ denote this finite set.
Fix $\varepsilon\in(0,1)$ and set
\[
B_\varepsilon:=\bigcup_{\theta_0\in\Theta_*}\big((\theta_0-\varepsilon,\theta_0+\varepsilon)\cap[0,2\pi]\big),
\qquad
G_\varepsilon:=[0,2\pi]\setminus B_\varepsilon.
\]
Then ${\rm Leb}(B_\varepsilon)\le 2\varepsilon\,{\rm card}(\Theta_*)$, where ${\rm Leb}$  and ${\rm card}$ stand for the Lebesgue measure and cardinality, respectively.
For each $\theta\in G_\varepsilon$ the maximizer of $v\mapsto \be_\theta^\tra v $ over $v\in \cC_{\bmu}$ is unique; denote it by $v_*(\theta)$. Necessarily $v_*(\theta)\in \cE$. Moreover, for every drift vector $\bmu_k\notin \cE$ we have the strict inequality
$\be_\theta^\tra
v_*(\theta)  > \be_\theta^\tra \bmu_k
$. 
Since $\theta\mapsto \be_\theta^\tra v_*(\theta)$ and $\theta\mapsto \be_\theta^\tra \bmu_k$ are continuous
and $G_\varepsilon$ is compact, the strict inequality is uniform:
\begin{equation}\label{eq:Delta-eps-full}
\Delta_\varepsilon
:=\inf_{\theta\in G_\varepsilon}\Big( \be_\theta^\tra v_*(\theta)  - \max_{k:\,\bmu_k\notin \cE} \be_\theta^\tra \bmu_k \Big)
>0.
\end{equation}
By Cauchy's perimeter formula we have
\begin{equation}\label{eq:cauchy-perim-full}
L_n^{(m)}=\int_0^{2\pi} h_n^{(m)}(\theta)\,\ud\theta,\qquad m\in\{2,3\}.
\end{equation}
Define for each $\theta$ and $n$,
\[
g_n(\theta):=\frac{h_n^{(2)}(\theta)-h_n^{(3)}(\theta)}{\sqrt n}\ge 0.
\]
By definitions of $\cH_n^{(2)}$ and $\cH_n^{(3)}$, for every $\theta$, we have
\begin{equation}\label{eq:h2-h3-as-max}
h_n^{(2)}(\theta)=\max\Big\{h_n^{(3)}(\theta),\ \max_{k:\,\bmu_k\notin \cE}S_n^{(k)}(\theta)\Big\},
\end{equation}
so
\begin{equation}\label{eq:h2-h3-basic-full}
0\le h_n^{(2)}(\theta)-h_n^{(3)}(\theta)
\le \max_{k:\,\bmu_k\notin \cE}\big(S_n^{(k)}(\theta)-h_n^{(3)}(\theta)\big)^+.
\end{equation}
Now fix $\theta$ and choose an index $\ell(\theta)$ among extremal drifts that maximizes the projected drift:
\[
\ell(\theta)\in\argmax_{\ell \in \cI}\{\bmu_\ell\cdot \be_\theta\}.
\]
Notice first that if there are no non-zero elements of $\cE$, then $\cH_n^{(2)}$ and $\cH_n^{(3)}$ coincide and there is nothing to prove. By definition of $\cH_n^{(3)}$, we always have $h_n^{(3)}(\theta)\ge S_n^{(\ell(\theta))}(\theta)$, hence from
\eqref{eq:h2-h3-basic-full},
\begin{equation}\label{eq:h2-h3-endpoint-compare}
0\le h_n^{(2)}(\theta)-h_n^{(3)}(\theta)
\le \max_{k:\,\bmu_k\notin \cE}\big(S_n^{(k)}(\theta)-S_n^{(\ell(\theta))}(\theta)\big)^+.
\end{equation}
Next, 
fix a particular $k$ with $\bmu_k\notin \cE$ (i.e.~$k\notin \cI$) and define 
\[
W_j^{(k)}(\theta):=S_j^{(k)}(\theta)-S_j^{(\ell(\theta))}(\theta),\qquad j\ge 0.
\]
Its increment is $(Z_{k,i}-Z_{\ell(\theta),i})\cdot \be_\theta$ and its drift is
\[
\mathbb{E}[W_1^{(k)}(\theta)]=\big(\bmu_k-\bmu_{\ell(\theta)}\big)\cdot \be_\theta \le 0,
\]
because $\ell(\theta)$ attains the maximum of $v\mapsto v\cdot \be_\theta$ over $\cC_{\bmu}$ restricted to $\cE$ and a non-extremal drift cannot exceed that maximum. Since $\mathbb{E}[W_1^{(k)}(\theta)]\le 0$, then for every $j$,
\[
W_j^{(k)}(\theta) = \bar W_j^{(k)}(\theta) + j\,\mathbb{E}[W_1^{(k)}(\theta)]
\le \bar W_j^{(k)}(\theta),
\]
where
$
\bar W_j^{(k)}(\theta):=W_j^{(k)}(\theta)-j\,\mathbb{E}[W_1^{(k)}(\theta)]
$
is a mean-zero martingale.
Therefore
\[
\big(W_n^{(k)}(\theta)\big)^+ \le \max_{0\le j\le n} W_j^{(k)}(\theta)
\le \max_{0\le j\le n} \bar W_j^{(k)}(\theta)
\le \max_{0\le j\le n} \big|\bar W_j^{(k)}(\theta)\big|.
\]
Combining with \eqref{eq:h2-h3-endpoint-compare} yields the   bound
\begin{equation}\label{eq:h2-h3-maxWtilde}
h_n^{(2)}(\theta)-h_n^{(3)}(\theta)
\le \max_{k:\,\bmu_k\notin \cE}\ \max_{0\le j\le n}\big|\bar W_j^{(k)}(\theta)\big|.
\end{equation}
Hence, we have
\[
\big(h_n^{(2)}(\theta)-h_n^{(3)}(\theta)\big)^2
\le
\sum_{k:\,\bmu_k\notin \cE}\Big(\max_{0\le j\le n}\big|\bar W_j^{(k)}(\theta)\big|\Big)^2.
\]
Taking expectations and applying Doob's maximal inequality it follows that
\begin{align*}
    \mathbb{E}\Big[\max_{0\le j\le n}\big|\bar W_j^{(k)}(\theta)\big|^2\Big] 
    & \le 4\,\mathbb{E}\big[|\bar W_n^{(k)}(\theta)|^2\big] =4n\,\mathrm{Var}\big((Z_{k,1}-Z_{\ell(\theta),1})\cdot \be_\theta\big) \\
    & \le 8n\left(\mathbb{E}[\|Z_{k,1}\|^2]+  \mathbb{E}[\|Z_{\ell(\theta),1}\|^2]\right).
\end{align*}
Putting these bounds together,
\[
\mathbb{E}\big[(h_n^{(2)}(\theta)-h_n^{(3)}(\theta))^2\big]
\le
8n\sum_{k:\,\bmu_k\notin \cE}\Big(\mathbb{E}[\|Z_{k,1}\|^2]+\sup_{\ell \in \cI} \mathbb{E}[\|Z_{\ell,1}\|^2]\Big)
\le C\,n,
\]
where the constant $C$ depends only on $N$ and  $(\mathbb{E}[\|Z_{k,1}\|^2])_{k\in[N]}$.
Hence 
\begin{equation}\label{eq:unif-L2-gn-full}
\sup_{\theta\in[0,2\pi]}\
\sup_{n\in\N} \mathbb{E}\big[g_n(\theta)^2\big]
= \sup_{\theta\in[0,2\pi]}\ \sup_{n\in\N}\frac{1}{n}\mathbb{E}\big[(h_n^{(2)}(\theta)-h_n^{(3)}(\theta))^2\big]
\le C <\infty.
\end{equation}

To obtain an exponential concentration inequality for the competition between extremal and non-extremal walks, we shall reduce to a bounded-increment model via a drift-preserving truncation.
We fix  $M>0$, and for each $k$ we define the drift-preserving truncated increment as
\begin{equation}\label{eq:trunc-Zbar-full}
\bar Z_{k,i}(M):= Z_{k,i}\mathbf{1}_{\{\|Z_{k,i}\|\le M\}}
-\mathbb{E}\!\left[Z_{k,1}\mathbf{1}_{\{\|Z_{k,1}\|\le M\}}\right]+\bmu_k.
\end{equation}
The remainder is defined as $R_{k,i}(M):=Z_{k,i}-\bar Z_{k,i}$.
Notice that
\[
R_{k,i}(M)
=
Z_{k,i}\mathbf 1_{\{\|Z_{k,i}\|>M\}}
-\mathbb{E}\!\left[Z_{k,1}\mathbf 1_{\{\|Z_{k,1}\|>M\}}\right].
\]
In particular, $\mathbb{E}[R_{k,i}(M)]=0$ and
\begin{equation}\label{eq:R-second-moment-full}
\mathbb{E}[\|R_{k,1}(M)\|^2] \le  \mathbb{E}\big[\|Z_{k,1}\|^2\,\mathbf 1_{\{\|Z_{k,1}\|>M\}}\big] \to 0,
\end{equation}
as $M\to\infty$.
Further,
define truncated walks
\[
\bar S^{(k)}_n(M):=\sum_{i=1}^n \bar Z_{k,i}(M),\quad
T^{(k)}_n(M):=\sum_{i=1}^n R_{k,i}(M),\qquad n\in\N,
\]
so $S^{(k)}_n=\bar S^{(k)}_n(M)+T^{(k)}_n(M)$.
Define $\bar \cH_n^{(2)}$ and $\bar \cH_n^{(3)}$ by replacing $S$ with $\bar S$ in the definitions of $\cH_n^{(2)},\cH_n^{(3)}$.
Let $\bar h_n^{(m)}(\theta)$ and $\bar L_n^{(m)}$ denote support functions and perimeters of $\bar \cH_n^{(m)}$ for $m=2,3$.
Fix now $m\in\{2,3\}$, $n\in\N$, and $\theta \in [0, 2\pi]$.
Both $h_n^{(m)}(\theta)$ and $\bar h_n^{(m)}(\theta)$ are maxima of projections of certain points generated from the walks.
Every point in the definition of $\cH_n^{(m)}$ is of the form $S_j^{(k)}$ for some $(k,j)$, and the corresponding point in $\bar \cH_n^{(m)}$
is $\bar S_j^{(k)}(M)$. Since $S_j^{(k)}-\bar S_j^{(k)}(M)=T_j^{(k)}(M)$, we have
$S_j^{(k)}(\theta)-\bar S_j^{(k)}(M)(\theta)=\be_\theta^\tra T_j^{(k)}(M)$,
so
$$|S_j^{(k)}(\theta)-\bar S_j^{(k)}(M)(\theta)|\le \|T_j^{(k)}(M)\|.
$$
Taking maxima over all points involved in $h_n^{(m)}(\theta)$ and $\bar h_n^{(m)}(\theta)$ gives the  bound
\begin{equation}\label{eq:hsupp-lipschitz}
|h_n^{(m)}(\theta)-\bar h_n^{(m)}(\theta)|
\le \max_{1\le k\le N}\max_{0\le j\le n}\|T^{(k)}_j(M)\|.
\end{equation}
Fix now $k\in[N]$. The process $( {T^{(k)}_n(M)} )_{n\in\ZP}$ is  mean-zero $\mathbb{R}^2$-valued random walk with i.i.d.\ increments $R_{k,i}(M)$.
Write $T^{(k)}_j(M)=(T^{(k,1)}_j(M),T^{(k,2)}_j(M))$.
Then
\[
\max_{0\le j\le n}\|T^{(k)}_j(M)\|^2
\le \max_{0\le j\le n}\big(T^{(k,1)}_j(M)\big)^2 + \max_{0\le j\le n}\big(T^{(k,2)}_j(M)\big)^2.
\]
Since each coordinate $( T^{(k,r)}_n(M) )_{n\in\ZP}$ is a mean-zero martingale, by Doob's maximal inequality  
\[
\mathbb{E}\Big[\max_{0\le j\le n}\big(T^{(k,r)}_j(M)\big)^2\Big]
\le 4\,\mathbb{E}\big[(T^{(k,r)}_n(M))^2\big]
=4n\,\mathrm{Var}(R_{k,1}^{(r)}(M))\le 4n\,\mathbb{E}\big[(R_{k,1}^{(r)}(M))^2\big],
\] where $( R_{k,i}^{(r)}(M) )_{i\in\N}$ denote the increments of $( T^{(k,r)}_n(M) )_{n\in\ZP}$.
Summing over $r=1,2$ gives
\begin{equation}\label{eq:maxTk}
\mathbb{E}\Big[\max_{0\le j\le n}\|T^{(k)}_j(M)\|^2\Big]
\le 4n\,\mathbb{E}[\|R_{k,1}(M)\|^2].
\end{equation}
This implies
\begin{equation}\label{eq:maxT-allk}
\mathbb{E}\Big[\Big(\max_{1\le k\le N}\max_{0\le j\le n}\|T^{(k)}_j(M)\|\Big)^2\Big]
\le \sum_{k=1}^N \mathbb{E}\Big[\max_{0\le j\le n}\|T^{(k)}_j(M)\|^2\Big]
\le 4n\sum_{k=1}^N \mathbb{E}[\|R_{k,1}(M)\|^2].
\end{equation}
From \eqref{eq:hsupp-lipschitz} and Cauchy's formula \eqref{eq:cauchy-perim-full}, it now follows that
\begin{align*}
|L_n^{(m)}-\bar L_n^{(m)}|
&\le
\int_0^{2\pi}|h_n^{(m)}(\theta)-\bar h_n^{(m)}(\theta)|\,\ud\theta
 \le 2\pi \max_{1\le k\le N}\max_{0\le j\le n}\|T^{(k)}_j(M)\|.
\end{align*}
Therefore,
\[
\mathbb{E}\Big[\Big(\frac{L_n^{(m)}-\bar L_n^{(m)}}{\sqrt n}\Big)^2\Big]
\le
(2\pi)^2\,\frac{1}{n}\,\mathbb{E}\Big[\Big(\max_{1\le k\le N}\max_{0\le j\le n}\|T^{(k)}_j(M)\|\Big)^2\Big]
\le
16\pi^2\sum_{k=1}^N \mathbb{E}[\|R_{k,1}(M)\|^2].
\]
Hence, for $m=2,3$,
\begin{equation}\label{eq:trunc-error-perim}
\sup_{n\in\N}\mathbb{E}\Big[\Big(\frac{L_n^{(m)}-\bar L_n^{(m)}}{\sqrt n}\Big)^2\Big]
\le
16\pi^2\sum_{k=1}^N \mathbb{E}[\|R_{k,1}(M)\|^2].
\end{equation}
By \eqref{eq:R-second-moment-full}, the right-hand side of~\eqref{eq:trunc-error-perim} tends to $0$ as $M\to\infty$.
We now switch to the truncated model. Note that the truncation was chosen to preserve drifts, i.e., 
$\mathbb{E}[\bar Z_{k,1}(M)]=\bmu_k$, so the drift gap $\Delta_\varepsilon$ from \eqref{eq:Delta-eps-full} is unchanged.
From \eqref{eq:trunc-Zbar-full},
$\|\bar Z_{k,1}(M)\| \leq 2M + \|\bmu_k\|$.
In particular, for each fixed $M$, there is a finite constant $B(M)$ such that for all $k\in[N]$,
\begin{equation}\label{eq:bounded-Zbar}
\|\bar Z_{k,1}(M)\|\le B(M)\qquad\as
\end{equation}
Fix $\theta\in G_\varepsilon$.
Let $\ell_*(\theta)$ be any index with $\bmu_{\ell_*(\theta)}=v_*(\theta)$. For each non-extremal index $k$ define
\[
\bar W_n^{(k)}(M)(\theta):=\bar S_n^{(k)}(M)(\theta)-\bar S_n^{(\ell_*(\theta))}(M)(\theta).
\]
Then
\[
\mathbb{E}[\bar W_n^{(k)}(M)(\theta)]=n\bigl(d_k(\theta)-d_{\ell_*(\theta)}(\theta)\bigr)\le -n\Delta_\varepsilon.
\]
Moreover, by \eqref{eq:bounded-Zbar}, it follows that for all $i\in\N$,
\[
\big|\bar W_i^{(k)}(M)(\theta)-\bar W_{i-1}^{(k)}(M)(\theta)\big|
=
\big|\be_\theta^\tra (\bar Z_{k,i}(M)-\bar Z_{\ell_*(\theta),i}(M))  \big|
\le \|\bar Z_{k,i}(M)-\bar Z_{\ell_*(\theta),i}(M)\|
\le 2B(M).
\]
By Hoeffding's inequality (see, e.g., \cite[Theorem 2.8]{BLM-concentration}),
\[
\mathbb{P}\big(\bar W_n^{(k)}(M)(\theta)\ge 0\big)
=
\mathbb{P}\Big(\bar W_n^{(k)}(M)(\theta)-\mathbb{E}[\bar W_n^{(k)}(M)(\theta)]\ge n\Delta_\varepsilon\Big)
\le
\exp\Big(-\frac{n\Delta_\varepsilon^2}{8B(M)^2}\Big).
\]
Taking the union bound over the finitely many indices $k$ with $\bmu_k\notin \cE$ yields
\begin{equation}\label{eq:exp-rare-full}
\sup_{\theta\in G_\varepsilon}
\mathbb{P}\Big(\max_{k:\,\bmu_k\notin \cE}\bar S_n^{(k)}(M)(\theta)\ \ge\ \bar S_n^{(\ell_*(\theta))}(M)(\theta)\Big)
\le
C_1\,\exp\big(-C_2\,n\big),
\end{equation}
for  constants $C_1=C_1(N)$ and $C_2=C_2(\varepsilon,M)>0$.
We write
$$
\frac{\bar L_n^{(2)}-\bar L_n^{(3)}}{\sqrt n}=\int_0^{2\pi}\bar g_n(\theta)\,\ud\theta, \quad \text{where} \quad
\bar g_n(\theta):=\frac{\bar h_n^{(2)}(\theta)-\bar h_n^{(3)}(\theta)}{\sqrt n}\ge 0.$$
We first show that for each fixed $M$,
\begin{equation}\label{eq:trunc-L2-goal}
\frac{\bar L_n^{(2)}-\bar L_n^{(3)}}{\sqrt n}\toL{n}{2}0,
\end{equation}
and then remove truncation using \eqref{eq:trunc-error-perim}.
Write
\[
\int_0^{2\pi}\bar g_n(\theta)\,\ud\theta
=
\int_{G_\varepsilon}\bar g_n(\theta)\,\ud\theta
+
\int_{B_\varepsilon}\bar g_n(\theta)\,\ud\theta.
\]
Because truncated increments are bounded, they have finite second moments; thus, as before, 
\begin{equation}\label{eq:unif-L2-barg}
\sup_{\theta\in[0,2\pi]}\
\sup_{n\in\N}\
\mathbb{E}\big[\bar g_n(\theta)^2\big] \le C(M)<\infty.
\end{equation}
Further, the Cauchy--Schwarz inequality yields
\[
\Big(\int_{B_\varepsilon}\bar g_n(\theta)\,d\theta\Big)^2
\le {\rm Leb}(B_\varepsilon)\int_{B_\varepsilon}\bar g_n(\theta)^2\,\ud\theta.
\]
Taking expectations and using \eqref{eq:unif-L2-barg} gives
\[
\sup_{n\in\N}\mathbb{E}\Big[\Big(\int_{B_\varepsilon}\bar g_n(\theta)\,\ud\theta\Big)^2\Big]
\le {\rm Leb}(B_\varepsilon)^2\,C(M),
\] which goes to zero as $\varepsilon\to0.$
Fix next $\theta\in G_\varepsilon$.
On $G_\varepsilon$, $\bar h_n^{(2)}(\theta)>\bar h_n^{(3)}(\theta)$ can only happen if some non-extremal endpoint beats the extremal endpoint
$\bar S_n^{(\ell_*(\theta))}(M)(\theta)$; by \eqref{eq:exp-rare-full} this event has probability at most $C_1 \re^{-C_2 n}$ uniformly in $\theta\in G_\varepsilon$. Next, in the truncated model, 
\[
|\bar S_n^{(k)}(M)(\theta)| \le \|\bar S_n^{(k)}(M)\| \le \sum_{i=1}^n \|\bar Z_{k,i}(M)\| \le nB(M),
\]
so it holds that $0\le \bar h_n^{(2)}(\theta)\le nB(M)$ and $0\le \bar h_n^{(3)}(\theta)\le nB(M)$.
Hence
\begin{equation}\label{eq:barg-crude}
0\le \bar g_n(\theta)=\frac{\bar h_n^{(2)}(\theta)-\bar h_n^{(3)}(\theta)}{\sqrt n}\le \frac{nB(M)}{\sqrt n}=B(M)\sqrt n.
\end{equation}
Therefore, if $\bar A_n(\theta)$ denotes the event that a non-extremal endpoint beats the extremal endpoint in direction $\theta$,
then $\bar g_n(\theta)=0$ on $\bar A_n(\theta)^c$ and \eqref{eq:barg-crude} gives
\[
\mathbb{E}\big[\bar g_n(\theta)^2\big]
=\mathbb{E}\big[\bar g_n(\theta)^2\mathbf 1_{\bar A_n(\theta)}\big]
\le B(M)^2\,n\,\mathbb{P}(\bar A_n(\theta))
\le B(M)^2\,n\,C_1 \re^{-C_2 n},
\]
uniformly over $\theta\in G_\varepsilon$ by \eqref{eq:exp-rare-full}.
Using the Cauchy--Schwarz inequality once again, it holds that
\[
\mathbb{E}\Big[\Big(\int_{G_\varepsilon}\bar g_n(\theta)\,d\theta\Big)^2\Big]
\le {\rm Leb}(G_\varepsilon)\int_{G_\varepsilon}\mathbb{E}\big[\bar g_n(\theta)^2\big]\,\ud\theta
\le (2\pi)^2\,B(M)^2\,n\,C_1 \re^{-C_2 n}.
\]
Combining with the $B_\varepsilon$ estimate proves \eqref{eq:trunc-L2-goal}.
Finally, write the telescoping decomposition
\[
\frac{L_n^{(2)}-L_n^{(3)}}{\sqrt n}
=
\frac{L_n^{(2)}-\bar L_n^{(2)}}{\sqrt n}
+\frac{\bar L_n^{(2)}-\bar L_n^{(3)}}{\sqrt n}
+\frac{\bar L_n^{(3)}-L_n^{(3)}}{\sqrt n}.
\]
Then
\[
\mathbb{E}\Big[\Big(\frac{L_n^{(2)}-L_n^{(3)}}{\sqrt n}\Big)^2\Big]
\le
3\sum_{m=2}^3 \mathbb{E}\Big[\Big(\frac{L_n^{(m)}-\bar L_n^{(m)}}{\sqrt n}\Big)^2\Big]
+3\,\mathbb{E}\Big[\Big(\frac{\bar L_n^{(2)}-\bar L_n^{(3)}}{\sqrt n}\Big)^2\Big].
\]
Fix $\eta>0$. Choose $M$ so large that the truncation error terms  are each $\le \eta/4$ uniformly in $n$.
This is possible because the right-hand side of \eqref{eq:trunc-error-perim} tends to $0$ as $M\to\infty$ by \eqref{eq:R-second-moment-full}.
Then, for this fixed $M$, choose $n$ large enough that the truncated term is $\le \eta/2$ by \eqref{eq:trunc-L2-goal}.
This finally proves \eqref{eq:2nd-approx-conv}.
Before moving on with the proof of Theorem~\ref{thm:perim_main_theorem}, we are now in a position to give the proof of Theorem~\ref{thm:perimeter-deviation-metric}.

\begin{proof}[Proof of Theorem~\ref{thm:perimeter-deviation-metric}]
The set $\cF_n$ equals the set $\cH_n^{(3)}$ defined at~\eqref{eq:H3-def} above. 
By construction,
$\cH_n^{(3)}\subseteq \cH_n^{(2)}\subseteq \cH_n$.
If $A,B\in\cK_0$ and $A\subseteq B$, then $h_A(\be_\theta)\le h_B(\be_\theta)$ for every $\theta$, and hence Cauchy's formula gives
\[
\rhoone(B,A)=\int_0^{2\pi}\bigl(h_B(\be_\theta)-h_A(\be_\theta)\bigr)\,\ud\theta
=\perim(B)-\perim(A).
\]
It follows that
\[
\rhoone(\cH_n,\cH_n^{(2)})=L_n^{(1)}-L_n^{(2)},
\qquad
\rhoone(\cH_n^{(2)},\cH_n^{(3)})=L_n^{(2)}-L_n^{(3)}.
\]
Approximations \eqref{eq:first-approx-conv} and \eqref{eq:2nd-approx-conv} show that both quantities in the last display, divided by $\sqrt n$, converge to zero in $L^2$. The conclusion follows from the triangle inequality for $\rhoone$.
\end{proof}


\subsubsection*{Third approximation}

In this step we split the interval $[0, 2\pi]$ in two parts, depending on whether non-zero drift walks dominate, or they all have negative projections. We integrate the support function of the convex hull spanned by the end-points of non-zero drift walks only over angles where such walks dominate, and on the complement we integrate only the support function of zero drift random walks, if there are any.
Recall that $\cI = \{ k \in [N] : \bmu_k \in \cE\}$, where $\cE$ is the set of extremal points of $\cC_{\bmu}$, and that we split $\cI$ into $\cI^+$ and $\cI^0$. Recall that $\Theta_0=\{\theta\in[0,2\pi]: \max_{k \in \cI^+}\bmu_k\cdot \be_\theta<0\}$, $P_n^{(k)}=\{S^{(k)}_n\}$ if $\bmu_k\neq \0$ and
$P_n^{(k)}=\{S^{(k)}_j:0\le j\le n\}$ if $\bmu_k=\0$. Define the convex hulls
\[
\cH_n^{(4),+}:=\mathrm{hull}\Big(\{\0\}\cup\bigcup_{k\in \cI^+}P_n^{(k)}\Big),
\qquad
\cH_n^{(4),0}:=\mathrm{hull}\Big(\bigcup_{k\in \cI^0}P_n^{(k)}\Big).
\]
If $\cI^0=\emptyset$, we set $\cH_n^{(4),0}=\{\0\}$.
Let $h_n^{(4),+}(\theta)$ and $h_n^{(4),0}(\theta)$ be their support functions, respectively. 
Recall that $\cH_n^{(3)}$ is the convex hull spanned by $\{\0\}$ and by 
the reduced paths $P_n^{(k)}$ with $k\in \cI$; equivalently, it contains the endpoints of walks in $\cI^+$ and the full trajectories of walks in $\cI^0$.
As before, we denote its support function by $h_n^{(3)}(\theta)$ and its perimeter by $L_n^{(3)}$.
Define
\[
L_n^{(4)} = L_n^{(4),+} + L_n^{(4),0} :=\int_{[0,2\pi]\setminus\Theta_0} h_n^{(4),+}(\theta)\,\ud\theta
\;+\;
\int_{\Theta_0} h_n^{(4),0}(\theta)\,\ud\theta.
\]
We aim to show that
\begin{equation}
    \label{eq:third-approx}
\frac{L_n^{(3)}-L_n^{(4)}}{\sqrt n}\ \toL{n}{2}  0.
\end{equation}
Since $\cH_n^{(3)}=\mathrm{hull}(\cH_n^{(4),+}\cup \cH_n^{(4),0})$, the corresponding support functions satisfy
\begin{equation}\label{eq:h3_is_max_final}
h_n^{(3)}(\theta)=\max\big\{h_n^{(4),+}(\theta),\,h_n^{(4),0}(\theta)\big\},\qquad \theta\in[0,2\pi].
\end{equation}
Define
\[
\Delta_n(\theta):=\frac{h_n^{(3)}(\theta)-h_n^{(4),+}(\theta)}{\sqrt n},
\qquad
\Gamma_n(\theta):=\frac{h_n^{(3)}(\theta)-h_n^{(4),0}(\theta)}{\sqrt n}.
\]
Then \eqref{eq:h3_is_max_final} implies the pointwise bounds
\begin{equation}\label{eq:DeltaGamma_bounds_final}
0\le \Delta_n(\theta)\le \frac{h_n^{(4),0}(\theta)}{\sqrt n},
\qquad
0\le \Gamma_n(\theta)\le \frac{h_n^{(4),+}(\theta)}{\sqrt n}.
\end{equation}
Moreover,
\begin{equation}\label{eq:Ln3Ln4_split_final}
\frac{L_n^{(3)}-L_n^{(4)}}{\sqrt n}
=
\int_{[0,2\pi]\setminus\Theta_0} \Delta_n(\theta)\,\ud\theta
+
\int_{\Theta_0} \Gamma_n(\theta)\,\ud\theta.
\end{equation}
So it suffices to show that both integrals in \eqref{eq:Ln3Ln4_split_final} converge to $0$ in $L^2$. Let
\[
m(\theta):=\max_{k \in \cI^+} \be_\theta^\tra \bmu_k,
\]
which is continuous on $[0,2\pi]$.
Then, $\Theta_0=\{\theta:\ m(\theta)<0\}$. Fix $\eta\in(0,1)$ and define
\[
\Theta_0^{(\eta)}:=\{\theta:\ m(\theta)\le -\eta\}\subset \Theta_0,
\qquad
\Theta_+^{(\eta)}:=\{\theta:\ m(\theta)\ge \eta\}\subset [0,2\pi]\setminus\Theta_0,
\]
and
$B_\eta:=\{\theta:\ |m(\theta)|<\eta\}$.
Then, $\Theta_0=\Theta_0^{(\eta)}\cup(\Theta_0\cap B_\eta)$ and
$[0,2\pi]\setminus\Theta_0=\Theta_+^{(\eta)}\cup(([0,2\pi]\setminus\Theta_0)\cap B_\eta)$.
Since  $m(\theta)$ is the maximum of finitely many non-identically-zero sinusoidal functions,
the zero set \(\{\theta:m(\theta)=0\}\) is contained in the finite union
\[
\bigcup_{k\in\mathcal I^+}\{\theta:  \be_\theta^\tra \bmu_k =0\}.
\]
Hence \(\{\theta:m(\theta)=0\}\) is finite and has Lebesgue measure zero.
Since \(B_\eta=\{\theta:|m(\theta)|<\eta\}\) decreases to
\(\{\theta:m(\theta)=0\}\) as \(\eta\downarrow0\), 
we conclude that
\begin{equation}\label{eq:B_eta_small_final}
\lim_{\eta\to0}{\rm Leb}(B_\eta)=0.
\end{equation}
From \eqref{eq:DeltaGamma_bounds_final}, for every $\theta$ we have
$0\le \Delta_n(\theta)\le h_n^{(4),0}(\theta)/\sqrt n$.
We now show that $h_n^{(4),0}(\theta)/\sqrt n$ has uniformly bounded second moment (uniformly in $n$ and $\theta$). If $\cI^0=\emptyset$, then $h_n^{(4),0}(\theta)\equiv 0$ and $\Delta_n(\theta)\equiv 0$.
Otherwise,
$h_n^{(4),0}(\theta)= \max_{k\in \cI^0} M_n^{(k)}(\theta)$.
Hence,
\[
\big(h_n^{(4),0}(\theta)\big)^2
\le
\sum_{k\in \cI^0}\Big( M_n^{(k)}(\theta)\Big)^2.
\]
For each $k\in \cI^0$, the random walk $(S_n^{(k)}(\theta))_{n\in\ZP}$ is a mean-zero martingale,
so Doob's  maximal inequality gives $\Exp[ (M_n^{(k)}(\theta) )^2 ]
\le 4 n\,\Exp[\|Z_{k,1}\|^2]$.
Hence,
\begin{equation}\label{eq:unif_L2_h40_final}
\sup_{\theta\in[0,2\pi]}\
\sup_{n\in\N}
\mathbb{E}\Big[\Big(\frac{h_n^{(4),0}(\theta)}{\sqrt n}\Big)^2\Big]
\le
4\sum_{k\in \cI^0}\mathbb{E}[\|Z_{k,1}\|^2]
<\infty.
\end{equation}
Combining \eqref{eq:DeltaGamma_bounds_final} and \eqref{eq:unif_L2_h40_final} yields
\begin{equation}\label{eq:unif_L2_Delta_final}
\sup_{\theta\in[0,2\pi]\setminus\Theta_0}\
\sup_{n\in\N}
\mathbb{E}[\Delta_n(\theta)^2]<\infty.
\end{equation}
From \eqref{eq:DeltaGamma_bounds_final}, for $\theta\in\Theta_0$ we have
$0\le \Gamma_n(\theta)\le h_n^{(4),+}(\theta)/\sqrt n$.
We now show that $h_n^{(4),+}(\theta)/\sqrt n$ has uniformly bounded second moment on $\Theta_0$.
This uses only that all drift projections are negative on $\Theta_0$. Fix $\theta\in\Theta_0$. Then $m(\theta)<0$, hence for every $k\in \cI^+$,
\[
d_k(\theta) = \be_\theta^\tra \bmu_k\le m(\theta)<0.
\]
Write $S_n^{(k)}(\theta)=n d_k(\theta)+W_n^{(k)}(\theta)$ where
\[
W_n^{(k)}(\theta):=\sum_{i=1}^n\bigl(\be_\theta^\tra Z_{k,i}-d_k(\theta)\bigr)
\]
is mean zero. Since $d_k(\theta)\le 0$,
\[
(S_n^{(k)}(\theta))^+=(n d_k(\theta)+W_n^{(k)}(\theta))^+ \le (W_n^{(k)}(\theta))^+\le |W_n^{(k)}(\theta)|.
\]
Therefore,
\[
\mathbb{E}\Big[\Big(\frac{(S_n^{(k)}(\theta))^+}{\sqrt n}\Big)^2\Big]
\le \frac{1}{n}\mathbb{E}[W_n^{(k)}(\theta)^2]
= \Var ( \be_\theta^\tra Z_{k,1} )\le \mathbb{E}[\|Z_{k,1}\|^2],
\]
uniformly in $n$ and $\theta\in\Theta_0$.
Now
\[
h_n^{(4),+}(\theta)=\max\Big\{0,\ \max_{k\in \cI^+}S_n^{(k)}(\theta)\Big\}
=\max_{k\in \cI^+}(S_n^{(k)}(\theta))^+,
\]
so
\[
\Big(\frac{h_n^{(4),+}(\theta)}{\sqrt n}\Big)^2
\le \sum_{k\in \cI^+}\Big(\frac{(S_n^{(k)}(\theta))^+}{\sqrt n}\Big)^2.
\]
Taking expectations yields
\begin{equation}\label{eq:unif_L2_h4plus_on_theta0_final}
\sup_{\theta\in\Theta_0}\ \sup_{n\in\N}
\mathbb{E}\Big[\Big(\frac{h_n^{(4),+}(\theta)}{\sqrt n}\Big)^2\Big]
\le \sum_{k\in \cI^+}\mathbb{E}[\|Z_{k,1}\|^2]<\infty.
\end{equation}
Combining~\eqref{eq:unif_L2_h4plus_on_theta0_final} with \eqref{eq:DeltaGamma_bounds_final} yields
\begin{equation}\label{eq:unif_L2_Gamma_final}
\sup_{\theta\in\Theta_0}\ \sup_{n\in\N}\mathbb{E}[\Gamma_n(\theta)^2]<\infty.
\end{equation}
Fix now $\eta\in(0,1)$ and $\theta\in\Theta_0^{(\eta)}$, so $m(\theta)\le -\eta$.
Then for every $k\in \cI^+$,
\[  \be_\theta^\tra
\bmu_k  \le m(\theta)\le -\eta.
\]
Hence, $(S_n^{(k)}(\theta))_{n\in\ZP}$ is a one-dimensional random walk with strictly negative drift.
Thus 
$n^{-1/2} ( (S_n^{(k)}(\theta))^+ )^2 \leq n^{-1/2} ( S_n^{(k)}(\theta) - n \be_\theta^\tra \bmu_k )^2$ is uniformly integrable, by~\eqref{eq:sum-mean-ui}, and $( S_n^{(k)}(\theta) )^+ \to 0$, a.s., as $n \to \infty$, by the SLLN. It follows that $n^{-1/2} (S_n^{(k)}(\theta))^+ \to0$ in $L^2$.
Since $\cI^+$ is finite, 
\[
\frac{h_n^{(4),+}(\theta)}{\sqrt n}=\max_{k\in \cI^+}\frac{(S_n^{(k)}(\theta))^+}{\sqrt n}
\toL{n}{2}0,\qquad \theta\in\Theta_0^{(\eta)}.
\]
Using $0\le \Gamma_n(\theta)\le h_n^{(4),+}(\theta)/\sqrt n$, from \eqref{eq:DeltaGamma_bounds_final} we conclude
\begin{equation}\label{eq:Gamma_to_0_theta0eta_final}
\Gamma_n(\theta)\toL{n}{2}0,\qquad \theta\in\Theta_0^{(\eta)}.
\end{equation}
Fix next $\eta\in(0,1)$ and $\theta\in\Theta_+^{(\eta)}$, so $m(\theta)\ge \eta$.
Choose $\ell=\ell(\theta)\in \cI^+$ such that $\bmu_\ell\cdot \be_\theta=m(\theta)\ge\eta$.
Define the ``wrong winner'' event $A_n(\theta):=\{h_n^{(4),0}(\theta)>h_n^{(4),+}(\theta)\}$.
Then $\Delta_n(\theta)=0$ on $A_n(\theta)^c$, and \eqref{eq:DeltaGamma_bounds_final} gives
\begin{equation}\label{eq:Delta_indicator_final}
0\le \Delta_n(\theta)\le \frac{h_n^{(4),0}(\theta)}{\sqrt n}\,\mathbf 1_{A_n(\theta)}.
\end{equation}
We will show that $\mathbb{E}[\Delta_n(\theta)^2]\to 0$ (uniformly in $\theta\in\Theta_+^{(\eta)}$) via truncation.
Let $\bar Z_{k,i}(M)$ and $R_{k, i}(M)$ be as in the second approximation. Let $\bar S_n^{(k)}(M):=\sum_{i=1}^n\bar Z_{k,i}(M)$ and define $\bar h_n^{(3)}(\theta)$, $\bar h_n^{(4),+}(\theta)$, $\bar h_n^{(4),0}(\theta)$
from $\bar S$ exactly as $h_n^{(3)}(\theta)$, $h_n^{(4),+}(\theta)$, $h_n^{(4),0}(\theta)$ were defined from $S$.
Define
\[
\bar\Delta_n(\theta):=\frac{\bar h_n^{(3)}(\theta)-\bar h_n^{(4),+}(\theta)}{\sqrt n}.
\]
Recall that for fixed $M$, the increments $\bar Z_{k,i}(M)$ are bounded in norm, i.e.,\ there exists a constant $B(M)$ such that $\|\bar Z_{k,i}(M)\|\le B(M)$ a.s.
Hence $| \be_\theta^\tra \bar Z_{k,i}(M)|\le B(M)$ for all $\theta$. Define the truncated ``wrong winner'' event
\[
\bar A_n(\theta):=\{\bar h_n^{(4),0}(\theta)>\bar h_n^{(4),+}(\theta)\}.
\]
We claim there exist constants $C_1=C_1(\eta,M)$ and $C_2=C_2(\eta,M)$, such that
\begin{equation}\label{eq:barA_exp_final}
\sup_{\theta\in\Theta_+^{(\eta)}}\mathbb{P}(\bar A_n(\theta))\le C_1e^{-C_2n}.
\end{equation}
To prove \eqref{eq:barA_exp_final}, note that
\[
\bar h_n^{(4),+}(\theta)\ge \max\{0,\bar S_n^{(\ell)}(M)(\theta)\},
\qquad
\bar h_n^{(4),0}(\theta)=\max_{k\in \cI^0}\max_{0\le j\le n}\bar S_j^{(k)}(M)(\theta).
\]
Hence, on $\bar A_n(\theta)$ we must have
\[
\max_{k\in \cI^0}\max_{0\le j\le n}\bar S_j^{(k)}(M)(\theta)>\bar S_n^{(\ell)}(M)(\theta).
\]
Therefore,
\begin{equation}\label{eq:barAn_as_union}
    \bar A_n(\theta)\subseteq
\Big\{\max_{k\in \cI^0}\max_{0\le j\le n}\bar S_j^{(k)}(M)(\theta)\ge \tfrac12\eta n\Big\}
\ \cup\
\Big\{\bar S_n^{(\ell)}(M)(\theta)\le \tfrac12\eta n\Big\}.    
\end{equation}
We bound the  probabilities of these two events. We have $\mathbb{E}[\bar S_n^{(\ell)}(M)(\theta)]=n  \be_\theta^\tra \bmu_\ell  =n m(\theta)\ge \eta n$.
Also, $\bar S_n^{(\ell)}(M)(\theta)$ is a bounded-increment sum with increments bounded by $B(M)$ in absolute value.
Thus, Hoeffding's inequality  yields
\begin{equation}\label{eq:barAn_first_term}
    \mathbb{P}\Big(\bar S_n^{(\ell)}(M)(\theta)\le \tfrac12\eta n\Big)
\le
\mathbb{P}\Big(\bar S_n^{(\ell)}(M)(\theta)-\mathbb{E}[\bar S_n^{(\ell)}(M)(\theta)]\le -\tfrac12\eta n\Big)
\le \exp\Big(-\frac{\eta^2 n}{8B(M)^2}\Big).
\end{equation}
Fix $k\in \cI^0$. Then, $\bmu_k=0$, so $(\bar S_n^{(k)} (M) (\theta))_{n\in\ZP}$ is a  martingale with increments bounded by $B(M)$.
A standard maximal version of Hoeffding's inequality (see \cite[Theorem 3.2.1]{roch2024}) gives
\[
\mathbb{P}\Big(\max_{0\le j\le n}\bar S_j^{(k)}(M)(\theta)\ge \tfrac12\eta n\Big)
\le \exp\Big(-\frac{\eta^2 n}{8B(M)^2}\Big).
\]
Finally, we have that
\begin{equation}\label{eq:barAn_second_term}
    \mathbb{P}\Big(\max_{k\in \cI^0}\max_{0\le j\le n}\bar S_j^{(k)}(M)(\theta)\ge \tfrac12\eta n\Big)
\le |\cI^0|\exp\Big(-\frac{\eta^2 n}{8B(M)^2}\Big).
\end{equation}
Combining \eqref{eq:barAn_as_union}, \eqref{eq:barAn_first_term} and \eqref{eq:barAn_second_term} proves \eqref{eq:barA_exp_final}.
In the truncated model we have the crude  bounds $0\le \bar h_n^{(3)}(\theta)\le nB(M)$ and $0\le \bar h_n^{(4),+}(\theta)\le nB(M)$.
Hence, $0\le \bar\Delta_n(\theta)\le B(M)\sqrt n$.
Moreover, $\bar\Delta_n(\theta)=0$ on $\bar A_n(\theta)^c$, so
\[
\mathbb{E}[\bar\Delta_n(\theta)^2]
\le B(M)^2\,n\,\mathbb{P}(\bar A_n(\theta))
\le B(M)^2\,n\,C_1e^{-C_2n},
\] which converges to zero, as $n\to\infty$,
uniformly in $\theta\in\Theta_+^{(\eta)}$.
As in the previous truncation, the truncation error in support functions is controlled by the maximal remainder walk.
Let $T_j^{(k)}(M):=\sum_{i=1}^j R_{k,i}(M)$, so $S_j^{(k)}=\bar S_j^{(k)}(M)+T_j^{(k)}(M)$.
For each of the three support functions $h\in\{h_n^{(3)},h_n^{(4),+},h_n^{(4),0}\}$ and its truncated version $\bar h$,
a deterministic Lipschitz argument gives
\[
|h(\theta)-\bar h(\theta)|\le \max_{1\le k\le N}\max_{0\le j\le n}\|T_j^{(k)}(M)\|.
\]
Doob's maximal inequality applied to the coordinate martingales of $T^{(k)}(M)$ yields
\[
\mathbb{E}\Big[\Big(\max_{1\le k\le N}\max_{0\le j\le n}\|T_j^{(k)}(M)\|\Big)^2\Big]
\le C\,n\sum_{k=1}^N\mathbb{E}[\|R_{k,1}(M)\|^2]
\]
for a universal constant $C$ (depending only on $N$).
Thus
\[
\sup_{\theta\in[0,2\pi]}\
\sup_{n\in\N}\
\mathbb{E}\Big[\Big(\frac{h(\theta)-\bar h(\theta)}{\sqrt n}\Big)^2\Big]
\le C\sum_{k=1}^N\mathbb{E}[\|R_{k,1}(M)\|^2],
\] which converges to zero as $M\to\infty$.
Consequently, for fixed $\eta$, the uniform $L^2$ convergence of $\bar\Delta_n(\theta)$ to $0$ on $\Theta_+^{(\eta)}$
implies the same for $\Delta_n(\theta)$, i.e.,
\begin{equation}\label{eq:Delta_to_0_theta+eta_final}
\lim_{n\to\infty}\sup_{\theta\in\Theta_+^{(\eta)}}\mathbb{E}[\Delta_n(\theta)^2]=0.
\end{equation}
We now return to \eqref{eq:Ln3Ln4_split_final} and split both integrals into interior and boundary-band parts:
\begin{equation}\label{eq:interior+boundary_band}
    \begin{aligned}
\int_{[0,2\pi]\setminus\Theta_0}\Delta_n(\theta)\,\ud\theta
&=
\int_{\Theta_+^{(\eta)}}\Delta_n(\theta)\,\ud\theta
+\int_{([0,2\pi]\setminus\Theta_0)\cap B_\eta}\Delta_n(\theta)\,\ud\theta,\\
\int_{\Theta_0}\Gamma_n(\theta)\,\ud\theta
&=
\int_{\Theta_0^{(\eta)}}\Gamma_n(\theta)\,\ud\theta
+\int_{\Theta_0\cap B_\eta}\Gamma_n(\theta)\,\ud\theta.
    \end{aligned}    
\end{equation}
Using the Cauchy--Schwarz inequality, we have that
\[
\mathbb{E}\Big[\Big(\int_{\Theta_+^{(\eta)}}\Delta_n(\theta)\,\ud\theta\Big)^2\Big]
\le {\rm Leb}(\Theta_+^{(\eta)})\int_{\Theta_+^{(\eta)}}\mathbb{E}[\Delta_n(\theta)^2]\,\ud\theta
\le 4\pi^2 \sup_{\theta\in\Theta_+^{(\eta)}}\mathbb{E}[\Delta_n(\theta)^2],
\]
and the right-hand side tends to $0$ by \eqref{eq:Delta_to_0_theta+eta_final}.
Similarly,
\[
\mathbb{E}\Big[\Big(\int_{\Theta_0^{(\eta)}}\Gamma_n(\theta)\,\ud\theta\Big)^2\Big]
\le {\rm Leb}(\Theta_0^{(\eta)})\int_{\Theta_0^{(\eta)}}\mathbb{E}[\Gamma_n(\theta)^2]\,\ud\theta,
\]
and since $\Gamma_n(\theta)$ converges to zero in $L^2$ pointwise on $\Theta_0^{(\eta)}$ by \eqref{eq:Gamma_to_0_theta0eta_final}
and is dominated by the uniform bound \eqref{eq:unif_L2_Gamma_final}, dominated convergence gives
$$\lim_{n\to\infty}\int_{\Theta_0^{(\eta)}}\mathbb{E}[\Gamma_n(\theta)^2]\,\ud\theta= 0.$$
Hence the second interior term also converges to $0$ in $L^2$.
Again by the  Cauchy--Schwarz inequality and the pointwise bounds \eqref{eq:DeltaGamma_bounds_final},
\[
\Big(\int_{([0,2\pi]\setminus\Theta_0)\cap B_\eta}\Delta_n(\theta)\,d\theta\Big)^2
\le {\rm Leb}(B_\eta)\int_{B_\eta}\Delta_n(\theta)^2\,\ud\theta
\le {\rm Leb}(B_\eta)\int_{B_\eta}\Big(\frac{h_n^{(4),0}(\theta)}{\sqrt n}\Big)^2\,\ud\theta,
\]
so taking expectations and using \eqref{eq:unif_L2_h40_final},
\begin{equation}\label{eq:boundary_band_1}
    \begin{aligned}\sup_{n\in\N}\mathbb{E}\Big[\Big(\int_{([0,2\pi]\setminus\Theta_0)\cap B_\eta}\Delta_n(\theta)\,\ud\theta\Big)^2\Big]
&\le {\rm Leb}(B_\eta)^2 \sup_{n,\theta}\mathbb{E}\Big[\Big(\frac{h_n^{(4),0}(\theta)}{\sqrt n}\Big)^2\Big]
\\ &\le 4{\rm Leb}(B_\eta)^2 \sum_{k\in \cI^0}\mathbb{E}[\|Z_{k,1}\|^2].
\end{aligned}
\end{equation}
Similarly, using \eqref{eq:unif_L2_h4plus_on_theta0_final},
\begin{equation}\label{eq:boundary_band_2}
\begin{aligned}
\sup_{n\in\N}\mathbb{E}\Big[\Big(\int_{\Theta_0\cap B_\eta}\Gamma_n(\theta)\,\ud\theta\Big)^2\Big]
&\le {\rm Leb}(B_\eta)^2\cdot \sup_{n,\theta\in\Theta_0}\mathbb{E}\Big[\Big(\frac{h_n^{(4),+}(\theta)}{\sqrt n}\Big)^2\Big]
\\ &\le {\rm Leb}(B_\eta)^2\cdot \sum_{k\in \cI^+}\mathbb{E}[\|Z_{k,1}\|^2].
\end{aligned}
\end{equation}
By \eqref{eq:B_eta_small_final}, both bounds can be made arbitrarily small by choosing $\eta$ small, uniformly in $n$.
Finally, fix $\varepsilon>0$, and choose $\eta$ so small that the boundary-band terms in \eqref{eq:boundary_band_1} and \eqref{eq:boundary_band_2} are each smaller than or equal to $\varepsilon/4$ for all $n$. Then choose $n$ so large that the two interior terms in \eqref{eq:interior+boundary_band} are each smaller than or equal to $\varepsilon/4$. Using \eqref{eq:Ln3Ln4_split_final} we thus verify~\eqref{eq:third-approx}.

\subsubsection*{Fourth approximation}
Recall that $L_n^{(4)} = L_n^{(4),+} + L_n^{(4),0}$. In this step we split the set $[0, 2\pi] \setminus \Theta_0$ into smaller subsets, depending on the dominating drift. We approximate $L_n^{(4),+}$ with a sum of integrals, over those smaller subsets, of maximal projections of only those walks whose drifts correspond to the dominating drift in each particular subset.
For $\theta\in[0,2\pi)$ write
\[
U_n^{(k)}:=S_n^{(k)}-n\bmu_k,\qquad
U_n^{(k)}(\theta):= \be_\theta^\tra U_n^{(k)} .
\]
Recall that
$m(\theta):=\max_{k \in \cI^+} \be_\theta^\tra \bmu_k$, $\Theta_0=\{\theta:\ m(\theta)<0\}$,
and that for  $\bmu\in \cE \setminus \{ \0\}$,
\[
\Theta_\bmu=\mathrm{int}\Big\{\theta\in[0,2\pi):\ \be_\theta^\tra \bmu  =\max_{\bmu'\in \cE} \be_\theta^\tra \bmu' \Big\},
\qquad
\cI_\bmu=\{k\in \cI:\ \bmu_k=\bmu\}.
\]
The boundary angles where ties occur form a finite set and may be ignored in integrals. Then up to Lebesgue-null sets,
\[
[0,2\pi)\setminus\Theta_0=\bigcup_{\bmu\in \cE\setminus\{\0\}} \Theta_\bmu,
\qquad \Theta_\bmu\cap\Theta_{\bmu'}=\emptyset,\ \ \text{for }\bmu\neq\bmu'.
\]
Define
\[
L_n^{(5),+}
:=\sum_{\bmu\in \cE\setminus\{0\}}
\int_{\Theta_\bmu}
\Big[
n\,  \be_\theta ^\tra \bmu + \max_{k\in \cI_\bmu}U_n^{(k)}(\theta)
\Big]\,\ud\theta.
\]
Note that
\[
\sum_{\bmu\in \cE\setminus\{0\}} \int_{\Theta_\bmu}  \be_\theta^\tra \bmu\,\ud\theta
=\int_{[0,2\pi)\setminus\Theta_0} m(\theta)\,\ud\theta
=\perim \cC_\bmu ,
\]
so $L_n^{(5),+}=n\,\perim \cC_\bmu +\sum_{\bmu \in \cE \setminus \{0\}}\int_{\Theta_\bmu}\max_{k\in \cI_\bmu}U_n^{(k)}(\theta)\,\ud\theta$.
We aim to show that
\begin{equation}
    \label{eq:fourth-approx}
\frac{L_n^{(5),+}-L_n^{(4),+}}{\sqrt{n}}\toL{n}{2}0.
\end{equation}
Fix $M>0$, and let $\bar{Z}_{k, i}(M)$, $R_{k, i}(M)$, $\bar{S}_n^{(k)}(M)$ and $T_n^{(k)}(M)$ be as before. Since truncation preserves drifts, the centred versions satisfy
\[
\bar{U}_n^{(k)}(M):=\bar S_n^{(k)}(M)-n\bmu_k,
\qquad
U_n^{(k)}-\bar{U}_n^{(k)}(M)=T_n^{(k)}(M).
\]
Let $\bar h_n^{(4),+}(\theta)$ be defined from $(\bar S_n^{(k)}(M)(\theta))_{n\in\ZP}$ in the same way as $h_n^{(4),+}(\theta)$ from $(S_n^{(k)}(\theta))_{n\in\ZP}$, and let
\[
\bar L_n^{(4),+}:=\int_{[0,2\pi)\setminus\Theta_0}\bar h_n^{(4),+}(\theta)\,\ud\theta.
\]
Similarly define
\[
\bar L_n^{(5),+}
:=\sum_{\bmu\in \cE\setminus\{\0\}}
\int_{\Theta_\bmu}
\Big[
n\,\be_\theta^\tra \bmu  + \max_{k\in \cI_\bmu}\bar{U}_n^{(k)}(M)(\theta)
\Big]\,\ud\theta.
\]
We will now show that  for each fixed $M$ it holds that
\begin{equation}\label{eq:trunc_error}
    \sup_{n\in\N}\mathbb{E}\Big[\Big(\frac{L_n^{(4),+}-\bar L_n^{(4),+}}{\sqrt{n}}\Big)^2\Big]
+
\sup_{n\in\N}\mathbb{E}\Big[\Big(\frac{L_n^{(5),+}-\bar L_n^{(5),+}}{\sqrt{n}}\Big)^2\Big]
\le 8\pi^2\sum_{k=1}^N\mathbb{E}[\|R_{k,1}(M)\|^2].
\end{equation}
Consequently, by \eqref{eq:R-second-moment-full}, the left-hand side converges to $0$ as $M\to\infty$. We use only simple Lipschitz facts. In particular, for each $\theta$,
\[
|h_n^{(4),+}(\theta)-\bar h_n^{(4),+}(\theta)|
=
\Big|\max\{0,\max_{k\in \cI^+}S_n^{(k)}(\theta)\}-\max\{0,\max_{k\in \cI^+}\bar S_n^{(k)}(M)(\theta)\}\Big|
\le \max_{k\in \cI^+}|T_n^{(k)}(M)(\theta)|.
\]
Since $|T_n^{(k)}(M)(\theta)|\le \|T_n^{(k)}(M)\|$, we get
\begin{equation}\label{eq:L4_pointwise_trunc}
|h_n^{(4),+}(\theta)-\bar h_n^{(4),+}(\theta)|
\le \max_{1\le k\le N}\|T_n^{(k)}(M)\|.
\end{equation}
Integrating \eqref{eq:L4_pointwise_trunc} over $[0,2\pi)\setminus\Theta_0\subset[0,2\pi)$ gives
\[
|L_n^{(4),+}-\bar L_n^{(4),+}|
\le 2\pi\,\max_{1\le k\le N}\|T_n^{(k)}(M)\|.
\]
Thus
\begin{equation}\label{eq:L4_trunc_L2}
\mathbb{E}\Big[\Big(\frac{L_n^{(4),+}-\bar L_n^{(4),+}}{\sqrt{n}}\Big)^2\Big]
\le \frac{(2\pi)^2}{n}\mathbb{E}[\max_{1\le k\le N}\|T_n^{(k)}(M)\|^2]
\le (2\pi)^2\sum_{k=1}^N \frac{\mathbb{E}[\|T_n^{(k)}(M)\|^2]}{n}.
\end{equation}
Since $(R_{k,i}(M))_{i\in\N}$ are i.i.d.\ with mean $\0$ and finite second moment,
\[
\mathbb{E}[\|T_n^{(k)}(M)\|^2]
=\mathbb{E}\left[\Big\|\sum_{i=1}^n R_{k,i}(M)\Big\|^2\right]
=\sum_{i=1}^n \mathbb{E}[\|R_{k,i}(M)\|^2]
=n\,\mathbb{E}[\|R_{k,1}(M)\|^2].
\]
Substituting into \eqref{eq:L4_trunc_L2} yields
\[
\sup_{n\in\N}\mathbb{E}\Big[\Big(\frac{L_n^{(4),+}-\bar L_n^{(4),+}}{\sqrt{n}}\Big)^2\Big]
\le (2\pi)^2\sum_{k=1}^N\mathbb{E}[\|R_{k,1}(M)\|^2].
\]
For each $\bmu$ and each $\theta\in\Theta_\bmu$,
\begin{align*}
\Big|\max_{k\in \cI_\bmu}U_n^{(k)}(\theta) - \max_{k\in \cI_\bmu}\bar{U}_n^{(k)}(M)(\theta)\Big|
&\le \max_{k\in \cI_\bmu}|U_n^{(k)}(\theta)-\bar{U}_n^{(k)}(M)(\theta)|
\\&\le \max_{1\le k\le N}\|T_n^{(k)}(M)\|.
\end{align*}
Hence
\[
|L_n^{(5),+}-\bar L_n^{(5),+}|
\le \sum_{\bmu\neq 0}\int_{\Theta_\bmu}\max_{1\le k\le N}\|T_n^{(k)}(M)\|\,\ud\theta
\le 2\pi\,\max_{1\le k\le N}\|T_n^{(k)}(M)\|.
\]
The same computation as above gives
\[
\sup_{n\in\N}\mathbb{E}\Big[\Big(\frac{L_n^{(5),+}-\bar L_n^{(5),+}}{\sqrt{n}}\Big)^2\Big]
\le (2\pi)^2\sum_{k=1}^N\mathbb{E}[\|R_{k,1}(M)\|^2],
\]
which proves the claim~\eqref{eq:trunc_error}. By equation~\eqref{eq:trunc_error}, it suffices to prove the desired $L^2$ convergence for the truncated quantities
$\bar L_n^{(4),+}$ and $\bar L_n^{(5),+}$ for each fixed $M$, and then let $M\to\infty$. Fix now $M>0$ and work with the truncated walks $\bar S^{(k)}(M)$.
Let $\bar h_n^{(4),+}$ and $\bar L_n^{(4),+}$ be as above. For each $\bmu\neq0$ and $\theta\in\Theta_\bmu$, define
\[
D_{n,\bmu}(\theta):=\bar h_n^{(4),+}(\theta) - \left(n\,\be_\theta^\tra \bmu +\max_{k\in \cI_\bmu}\bar U_n^{(k)}(M)(\theta)\right)\ge 0,\qquad \theta\in\Theta_\bmu.
\]
Since $[0,2\pi)\setminus\Theta_0$ is (up to null sets) the disjoint union of the $\Theta_\bmu$'s, we have
\begin{equation}\label{eq:trunc_diff_as_sum_mu}
\bar L_n^{(4),+}-\bar L_n^{(5),+}
=\sum_{\bmu\in \cE\setminus\{0\}}\int_{\Theta_\bmu} D_{n,\bmu}(\theta)\,\ud\theta.
\end{equation}
Fix $\bmu\neq0$ and $\theta\in\Theta_\bmu$. Recall that for each $k\in \cI^+$ we have
\[
\bar S_n^{(k)}(M)(\theta)=n \be_\theta^\tra \bmu_k +\bar U_n^{(k)}(M)(\theta).
\]
Since $\theta\in\Theta_\bmu$, $\bmu$ is the unique maximizer among $\cE$ in direction $\theta$, hence for each $\bmu'\in \cE\setminus\{\bmu\}$,
$\be_\theta^\tra 
(\bmu-\bmu') >0$.
Since
\[
\bar h_n^{(4),+}(\theta)
=
\max\Big\{
0,\
\max_{\bmu'\in\cE\setminus\{0\}}
\max_{k\in\cI_{\bmu'}}
\big(n\be_\theta^\tra \bmu' +\bar U_n^{(k)}(M)(\theta)\big)
\Big\},
\]
the quantity \(D_{n,\bmu}(\theta)\) can be positive only if either some walk corresponding to another non-zero extreme drift
\(\bmu'\neq\bmu\) exceeds the best walk with drift \(\bmu\), or if the deterministic point \(0\) exceeds the latter.
Set
$d_{\bmu,\bmu'}(\theta):=\be_\theta^\tra (\bmu-\bmu') $.
We thus have
\begin{align}
	D_{n,\bmu}(\theta)
	\le&
	\max_{\bmu'\in \cE\setminus\{0,\bmu\}}
	\max_{k\in \cI_{\bmu'}}
	\max_{\ell\in \cI_\bmu}
	\bigl(
	\bar U_n^{(k)}(M)(\theta)
	-\bar U_n^{(\ell)}(M)(\theta)
	-n d_{\bmu,\bmu'}(\theta)
	\bigr)^+
	\nonumber\\
	&+
	\max_{\ell\in \cI_\bmu}
	\bigl(
	-\bar U_n^{(\ell)}(M)(\theta)
	-n\be_\theta^\tra\bmu
	\bigr)^+ .
	\label{eq:D_bound_by_pairwise}
\end{align}
Consequently,
\begin{align}
	D_{n,\bmu}(\theta)^2
	\le&
	2\sum_{\bmu'\in \cE\setminus\{\0,\bmu\}}
	\sum_{k\in \cI_{\bmu'}}
	\sum_{\ell\in \cI_\bmu}
	\Bigl( \bigl(
	\bar U_n^{(k)}(M)(\theta)
	-\bar U_n^{(\ell)}(M)(\theta)
	-n d_{\bmu,\bmu'}(\theta)
	\bigr)^+ \Bigr)^2
	\nonumber\\
	&+
	2\sum_{\ell\in \cI_\bmu}
	\Bigl( \bigl(
	-\bar U_n^{(\ell)}(M)(\theta)
	-n\be_\theta^\tra \bmu
	\bigr)^+ \Bigr)^2 .
	\label{eq:D2_sum_bound}
\end{align}

We now use that in the truncated model, the increments are uniformly bounded.
Recall that for fixed $M>0$ there exists a constant $B(M)$ such that for all $k,i,\theta$,
$\|\bar Z_{k,i}(M)\|\le B(M)$ and $|\bar \be_\theta^\tra Z_{k,i}(M)  |\le B(M)$, a.s.
Consequently, the centred projected increments
\[
\xi_{k,i}(M)(\theta):=\be_\theta^\tra \bar Z_{k,i}(M) t -\be_\theta^\tra \bmu_k 
\]
satisfy $|\xi_{k,i}(\theta)|\le 2B(M)$ a.s. Fix $\bmu'\neq\bmu$, $k\in \cI_{\bmu'}$ and $\ell\in \cI_\bmu$. For fixed $\theta$, define
\[
X_n(M)(\theta):=\bar U_n^{(k)}(M)(\theta)-\bar U_n^{(\ell)}(M)(\theta).
\]
Then $X_n(M)(\theta)=\sum_{i=1}^n (\xi_{k,i}(M)(\theta)-\xi_{\ell,i}(M)(\theta))$ is a sum of independent mean-zero increments.
Notice that each increment satisfies
\[
|\xi_{k,i}(M)(\theta)-\xi_{\ell,i}(M)(\theta)|\le 4B(M) \quad\text{a.s.}
\]
By Hoeffding's lemma \cite[Lemma 2.2]{BLM-concentration} applied for bounded mean-zero increments, $X_n(\theta)$ is sub-Gaussian, i.e.,\ for all $\lambda\in\mathbb{R}$,
\begin{equation}\label{eq:subg_Xn}
\mathbb{E}\exp(\lambda X_n(\theta))\le \exp\big(8nB(M)^2\lambda^2\big).
\end{equation}
Applying Lemma~\ref{lem:overshoot} with $a=n\,d_{\bmu,\bmu'}(\theta)>0$ and $\sigma^2=16nB(M)^2$ yields
\begin{equation}\label{eq:overshoot_bound_applied}
\begin{aligned}
\mathbb{E}\Big[ \Bigl( \bigl(X_n(\theta)-n\,d_{\bmu,\bmu'}(\theta)\bigr)^+ \Bigr)^2\Big]
&\le 
32B(M)^2\,n\,\exp\Big(-\frac{n\,d_{\bmu,\bmu'}(\theta)^2}{32B(M)^2}\Big).
\end{aligned}
\end{equation}
The second term in \eqref{eq:D2_sum_bound} is handled in the same way. Indeed, for fixed
\(\ell\in\cI_\bmu\), the random variable
\(-\bar U_n^{(\ell)}(M)(\theta)\) is a centred sum of independent bounded increments. Applying Lemma~\ref{lem:overshoot} with
\(a=n\be_\theta^\tra \bmu>0\) gives
\begin{equation}\label{eq:zero_competitor_bound}
	\mathbb{E}\Big[
	\Bigl( \bigl(
	-\bar U_n^{(\ell)}(M)(\theta)
	-n\be_\theta^\tra \bmu
	\bigr)^+ \Bigr)^2
	\Big]
	\le
	32B(M)^2\,n
	\exp\Big(
	-\frac{n(\be_\theta^\tra \bmu )^2}{32B(M)^2}
	\Big).
\end{equation}
Taking expectations in \eqref{eq:D2_sum_bound} and using
\eqref{eq:overshoot_bound_applied} and \eqref{eq:zero_competitor_bound}, we obtain
\begin{align}
	\mathbb{E}[D_{n,\bmu}(\theta)^2]
	\le&
	64B(M)^2\,n
	\sum_{\bmu'\in \cE\setminus\{\0,\bmu\}}
	{\rm card}(\cI_{\bmu'})\,{\rm card}(\cI_\bmu)
	\exp\Big(
	-\frac{n\,d_{\bmu,\bmu'}(\theta)^2}{32B(M)^2}
	\Big)
	\nonumber\\
	&+
	64B(M)^2\,n\,{\rm card}(\cI_\bmu)
	\exp\Big(
	-\frac{n(\be_\theta^\tra \bmu)^2}{32B(M)^2}
	\Big),
	\qquad \theta\in\Theta_\bmu .
	\label{eq:ED2_pointwise}
\end{align}

Now define the $\bmu$-piece of the scaled perimeter difference
\[
\Delta_{n,\bmu}:=\int_{\Theta_\bmu}\frac{D_{n,\bmu}(\theta)}{\sqrt{n}}\,\ud\theta.
\]
By the Cauchy--Schwarz inequality,
\[
\mathbb{E}[\Delta_{n,\bmu}^2]
\le {\rm Leb}(\Theta_\bmu) \int_{\Theta_\bmu}\mathbb{E}\Big[\frac{D_{n,\bmu}(\theta)^2}{n}\Big]\,\ud\theta.
\]
Using \eqref{eq:ED2_pointwise},
\begin{align}
	\mathbb{E}[\Delta_{n,\bmu}^2]
	\le&
	64{\rm Leb}(\Theta_\bmu) B(M)^2
	\sum_{\bmu'\in \cE\setminus\{\0,\bmu\}}
	{\rm card}(\cI_{\bmu'})\,{\rm card}(\cI_\bmu)
	\int_{\Theta_\bmu}
	\exp\Big(
	-\frac{n\,d_{\bmu,\bmu'}(\theta)^2}{32B(M)^2}
	\Big)\,\ud\theta
	\nonumber\\
	&+
	64{\rm Leb}(\Theta_\bmu) B(M)^2\,{\rm card}(\cI_\bmu)
	\int_{\Theta_\bmu}
	\exp\Big(
	-\frac{n(\be_\theta^\tra \bmu)^2}{32B(M)^2}
	\Big)\,\ud\theta .
	\label{eq:Delta_mu_bound_before_geom}
\end{align}
Apply Lemma~\ref{lem:geom_int} first with \(v=\bmu-\bmu'\neq\0\), and then with \(v=\bmu\neq\0\), in both cases with \(c=\frac{1}{32B(M)^2}\). Hence all integrals appearing in \eqref{eq:Delta_mu_bound_before_geom} are bounded by a constant times \(n^{-1/2}\). Therefore
$\mathbb{E}[\Delta_{n,\bmu}^2]
\le n^{-1/2} {C(M)}$,
where the constant \(C(M)\) does not depend on \(n\).
In particular, $\mathbb{E}[\Delta_{n,\bmu}^2]$ converges to zero as $n\to\infty$.
Finally, by \eqref{eq:trunc_diff_as_sum_mu},
\[
\frac{\bar L_n^{(4),+}-\bar L_n^{(5),+}}{\sqrt{n}}
=\sum_{\bmu \in \cE \setminus \{0\}}\Delta_{n,\bmu}.
\]
Since there are only finitely many $\bmu\in \cE\setminus\{\0\}$, the Cauchy--Schwarz inequality gives
\[
\mathbb{E}\Big[\Big(\frac{\bar L_n^{(4),+}-\bar L_n^{(5),+}}{\sqrt{n}}\Big)^2\Big]
\le {\rm card}(\cE\setminus\{\0\})\sum_{\bmu \in \cE\setminus\{0\}}\mathbb{E}[\Delta_{n,\bmu}^2],
\] which converges to zero as $n\to\infty.$
Thus, for each fixed $M$,
\begin{equation}\label{eq:trunc_L2_convergence}
\frac{\bar L_n^{(5),+}-\bar L_n^{(4),+}}{\sqrt{n}}
\toL{n}{2}0.
\end{equation}
Write the decomposition
\[
\frac{L_n^{(5),+}-L_n^{(4),+}}{\sqrt{n}}
=
\frac{\bar L_n^{(5),+}-\bar L_n^{(4),+}}{\sqrt{n}}
+
\frac{(L_n^{(5),+}-\bar L_n^{(5),+})-(L_n^{(4),+}-\bar L_n^{(4),+})}{\sqrt{n}}.
\]
Then
\[
\mathbb{E}\Big[\Big(\frac{L_n^{(5),+}-L_n^{(4),+}}{\sqrt{n}}\Big)^2\Big]
\le
2\,\mathbb{E}\Big[\Big(\frac{\bar L_n^{(5),+}-\bar L_n^{(4),+}}{\sqrt{n}}\Big)^2\Big]
+
2\,\mathbb{E}\Big[\Big(\frac{(L_n^{(5),+}-\bar L_n^{(5),+})-(L_n^{(4),+}-\bar L_n^{(4),+})}{\sqrt{n}}\Big)^2\Big].
\]
From \eqref{eq:trunc_error}, 
\[
\sup_{n\in\N}\mathbb{E}\Big[\Big(\frac{(L_n^{(5),+}-\bar L_n^{(5),+})-(L_n^{(4),+}-\bar L_n^{(4),+})}{\sqrt{n}}\Big)^2\Big]
\le
16\pi^2\sum_{k=1}^N\mathbb{E}[\|R_{k,1}\|^2].
\]
Fix $\varepsilon>0$. Choose $M$ so large that $16\pi^2 \sum_k\mathbb{E}[\|R_{k,1}\|^2]\le \varepsilon/2$
(using \eqref{eq:R-second-moment-full}). Then for this $M$, by \eqref{eq:trunc_L2_convergence} choose $n$ so large that
$2 n^{-1} \mathbb{E} [ ( \bar L_n^{(5),+}-\bar L_n^{(4),+} )^2 ]\le \varepsilon/2$.
Hence for all such $n$,
$n^{-1} \mathbb{E} [ ( L_n^{(5),+}-L_n^{(4),+} )^2]\le \varepsilon$.
This proves~\eqref{eq:fourth-approx}.

\subsubsection*{CLT for the drift--dominated approximation $L_n^{(5),+}$}
In this step we show that, in the limit, we can replace the $n$-th steps of random walks appearing under the integrals in the definition of $L_n^{(5),+}$ with the corresponding Gaussian limiting variables.
Recall that
\begin{equation}\label{eq:L5plus_centered}
L_n^{(5),+} - n\,\perim \cC_\bmu = \sum_{\bmu \in \cE \setminus \{0\}}\int_{\Theta_\bmu}\max_{k\in \cI_\bmu}U_n^{(k)}(\theta)\,\ud\theta,
\end{equation}
and $\Sigma_k=\mathrm{Cov}(Z_{k,1})$, and $(X^{(k)})_{k=1}^N$ are independent centred Gaussian vectors with covariance matrices $\Sigma_k$.
Define the limiting Gaussian fields
$G_k(\theta):=\be_\theta^\tra X^{(k)}$, $\theta\in[0,2\pi)$.
Under the standing assumptions, we aim to show that
\begin{equation}\label{eq:lim_for_L_n5+}
    \frac{L_n^{(5),+}-n\,\perim \cC_\bmu}{\sqrt{n}}
\ \tod{n}
\ \sum_{\bmu\in \cE\setminus\{\0\}}\int_{\Theta_\bmu}\max_{k\in \cI_\bmu} G_k(\theta)\,\ud\theta.    
\end{equation}
  For each $n$ define the $\mathbb{R}^{|\cI^+|}$-valued process on $[0,2\pi]$ (identify $0$ and $2\pi$) by
\[
\mathbb{X}_n(\theta):=\big(\mathbb{X}_{n,k}(\theta)\big)_{k\in \cI^+},
\qquad
\mathbb{X}_{n,k}(\theta):=\frac{U_n^{(k)}(\theta)}{\sqrt{n}}.
\]
Also define the $\mathbb{R}^{|\cI^+|}$-valued Gaussian limit process
$\mathbb{G}(\theta):=\big(G_k(\theta)\big)_{k\in \cI^+}$.
Both $\mathbb{X}_n$ and $\mathbb{G}$ have continuous sample paths in $\theta$ (a.s.). Define the functional $\Phi:C([0,2\pi];\mathbb{R}^{|\cI^+|})\to\mathbb{R}$ by
\begin{equation}\label{eq:Phi_def}
\Phi(f)
:=
\sum_{\bmu\in \cE\setminus\{\0\}} \int_{\Theta_\bmu} \max_{k\in \cI_\bmu} f_k(\theta)\,\ud\theta,
\qquad f=(f_k)_{k\in \cI^+}.
\end{equation}
Then \eqref{eq:L5plus_centered} can be rewritten as
\begin{equation}\label{eq:Phi_representation}
\frac{L_n^{(5),+}-n\,\perim \cC_\bmu}{\sqrt{n}}
=
\Phi(\mathbb{X}_n).
\end{equation}
Similarly, the limit in~\eqref{eq:lim_for_L_n5+} equals $\Phi(\mathbb{G})$. Thus~\eqref{eq:lim_for_L_n5+} will follow once we prove:
\begin{itemize}
    \item[(a)] $\mathbb{X}_n \tod{n} \mathbb{G}$ in $C([0,2\pi];\mathbb{R}^{|\cI^+|})$ (in uniform topology);
    \item[(b)] $\Phi$ is continuous on $C([0,2\pi];\mathbb{R}^{|\cI^+|})$.
\end{itemize}
Then the continuous mapping theorem yields $\Phi(\mathbb{X}_n)\tod{n} \Phi(\mathbb{G})$, which gives~\eqref{eq:lim_for_L_n5+}.

We first show that the functional $\Phi$ defined in \eqref{eq:Phi_def} is Lipschitz with respect to the uniform norm
$\|f\|_\infty:=\sup_{\theta\in[0,2\pi]}\max_{k\in \cI^+}|f_k(\theta)|$, i.e.,
\begin{equation}\label{eq:Phi_cont}
    |\Phi(f)-\Phi(g)|
\le 2\pi\,\|f-g\|_\infty,
\qquad f,g\in C([0,2\pi];\mathbb{R}^{|\cI^+|}).
\end{equation}
In particular, $\Phi$ is continuous. Fix $f,g\in C([0,2\pi];\mathbb{R}^{|\cI^+|})$. For each $\bmu$ and each $\theta$,
\[
\Big|\max_{k\in \cI_\bmu} f_k(\theta)-\max_{k\in \cI_\bmu} g_k(\theta)\Big|
\le \max_{k\in \cI_\bmu}|f_k(\theta)-g_k(\theta)|
\le \|f-g\|_\infty.
\]
Integrating over $\Theta_\bmu$ and summing over $\bmu\in \cE\setminus\{\0\}$ gives us
\[
|\Phi(f)-\Phi(g)|
\le \sum_{\bmu \in \cE \setminus \{\0\}} \int_{\Theta_\bmu}\|f-g\|_\infty\,\ud\theta
\le \|f-g\|_\infty \cdot \sum_{\bmu \in \cE \setminus \{0\}}{\rm Leb}(\Theta_\bmu)
\le 2\pi\,\|f-g\|_\infty,
\]
since the disjoint union of the $\Theta_\bmu$ lies in $[0,2\pi)$.

We next prove $\mathbb{X}_n \tod{n} \mathbb{G}$ in $C([0,2\pi];\mathbb{R}^{|\cI^+|})$ by combining finite-dimensional convergence with tightness in $C([0,2\pi];\mathbb{R}^{|\cI^+|})$. For arbitrary $r \in \N$ fix $\theta_1,\dots,\theta_r\in[0,2\pi]$. The $\mathbb{R}^{r\cdot |\cI^+|}$-valued vector
$(\mathbb{X}_n(\theta_1),\dots,\mathbb{X}_n(\theta_r) )$  is a linear transformation of the vector of sums
\[
\Big(\frac{1}{\sqrt{n}}\sum_{i=1}^n (Z_{k,i}-\bmu_k)\Big)_{k\in \cI^+}\in(\mathbb{R}^2)^{|\cI^+|}.
\]
By the multivariate CLT,
\[
\Big(\frac{1}{\sqrt{n}}\sum_{i=1}^n (Z_{k,i}-\bmu_k)\Big)_{k\in \cI^+}
 \tod{n}
\big(X^{(k)}\big)_{k\in \cI^+}.
\]
Applying the continuous mapping theorem for linear maps yields
\[
\big(\mathbb{X}_n(\theta_1),\dots,\mathbb{X}_n(\theta_r)\big)
\ \tod{n}\
\big(\mathbb{G}(\theta_1),\dots,\mathbb{G}(\theta_r)\big).
\]
Thus the finite-dimensional distributions of $\mathbb{X}_n$ converge to those of $\mathbb{G}$. 

To show tightness, we use a standard modulus-of-continuity tightness criterion for processes with continuous sample paths, based on a uniform second-moment increment bound (a Kolmogorov--Chentsov type criterion, see \cite[Chapter VI, Theorem 4.1]{Jacod-Shr}). Since we are in one-dimensional parameter space, it suffices to obtain an estimate of the form
\[
\sup_{n\in\N}\mathbb{E}\big[\|\mathbb{X}_n(\theta)-\mathbb{X}_n(\theta')\|_{\infty}^2\big]
\le C |\theta-\theta'|^2,
\] 
for some constant $C$.
Here $\|x\|_\infty=\max_{k\in \cI^+}|x_k|$ on $\mathbb{R}^{|\cI^+|}$. Fix $k\in \cI^+$. Note that
\[
\mathbb{X}_{n,k}(\theta)-\mathbb{X}_{n,k}(\theta')
=\frac{1}{\sqrt{n}} (\be_\theta-\be_{\theta'})^\tra U_n^{(k)}.
\]
Since $U_n^{(k)}=\sum_{i=1}^n (Z_{k,i}-\bmu_k)$ is a sum of i.i.d.\ centred vectors with covariance $\Sigma_k$, $n^{-1} \Exp [ U_n^{(k)} (U_n^{(k)})^\tra ] = \Sigma_k$. Therefore,
\[
\mathbb{E}[|\mathbb{X}_{n,k}(\theta)-\mathbb{X}_{n,k}(\theta')|^2]
=
\mathbb{E}\Big[\Big(\frac{1}{\sqrt{n}} (\be_\theta-\be_{\theta'})^\tra U_n^{(k)} \Big)^2\Big]
=(\be_\theta-\be_{\theta'})^\top \Sigma_k (\be_\theta-\be_{\theta'}).
\]
Let $\|\Sigma_k\|_{\mathrm{op}}$ denote the operator norm. Then
\[
(\be_\theta-\be_{\theta'})^\top \Sigma_k (\be_\theta-\be_{\theta'})
\le \|\Sigma_k\|_{\mathrm{op}}\ \|\be_\theta-\be_{\theta'}\|^2.
\]
Since $\theta\mapsto \be_\theta$ is $1$-Lipschitz as a map from $[0,2\pi]$ to $\mathbb{R}^2$,
\[
\mathbb{E}[|\mathbb{X}_{n,k}(\theta)-\mathbb{X}_{n,k}(\theta')|^2]
\le \|\Sigma_k\|_{\mathrm{op}}\,|\theta-\theta'|^2.
\]
Therefore
\begin{align*}
\mathbb{E}\Big[\max_{k\in \cI^+}\big|\mathbb{X}_{n,k}(\theta)-\mathbb{X}_{n,k}(\theta')\big|^2\Big]
&\le \sum_{k\in \cI^+} \mathbb{E}\big|\mathbb{X}_{n,k}(\theta)-\mathbb{X}_{n,k}(\theta')\big|^2 
 \le \Big(\sum_{k\in \cI^+}\|\Sigma_k\|_{\mathrm{op}}\Big)\,|\theta-\theta'|^2.
\end{align*}
We infer that
\begin{equation}\label{eq:functional_CLT_conclusion}
\mathbb{X}_n \tod{n} \mathbb{G}\quad\text{in } C([0,2\pi];\mathbb{R}^{|\cI^+|}).
\end{equation}
By equation~\eqref{eq:Phi_cont}, $\Phi$ is continuous on $C([0,2\pi];\mathbb{R}^{|\cI^+|})$.
Thus the continuous mapping theorem applied to \eqref{eq:functional_CLT_conclusion} yields
$\Phi(\mathbb{X}_n)\tod{n} \Phi(\mathbb{G})$.
Using \eqref{eq:Phi_representation}, we finally obtain \eqref{eq:lim_for_L_n5+}.

\subsubsection*{Invariance principle for the zero--drift perimeter contribution}
Here we show that on the set $\Theta_0$, in the limit, the zero-drift random walks can be replaced by the limiting Brownian motions.
Recall that
\[
L_n^{(4),0}=\int_{\Theta_0} h_n^{(4),0}(\theta)\,\ud\theta,
\]
where $\Theta_0=\{\theta\in[0,2\pi):\ \max_{\bmu\in \cE}\be_\theta^\tra \bmu<0\}$, $h_n^{(4),0}(\theta) = \max_{k\in \cI^0}\max_{0\le j\le n} S_j^{(k)}(\theta)$, $\theta\in\Theta_0$, and $\cI^0=\{k:\ \bmu_k=0\}$. Let $(B^{(k)})_{k \in \cI^0}$ be independent planar Brownian motions with covariance matrices $\Sigma_k$. Define the limiting support function on $\Theta_0$ by
\[
h^{0}(\theta)
:= \max_{k\in \cI^0}\sup_{0\le t\le 1}\big(B^{(k)}(t)\cdot \be_\theta\big),
\qquad \theta\in\Theta_0
\]
so that $\Pi^0$ as defined at~\eqref{eq:Pi-defs} satisfies 
\begin{equation}\label{eq:def-Pi-0}
	\Pi^0=\int_{\Theta_0} h^{0}(\theta)\,\ud\theta.
\end{equation}
We now aim to show the following convergence
\begin{equation}
    \label{eq:Pi-0-convergence}
\frac{L_n^{(4),0}}{\sqrt{n}}\ \tod{n}\ \Pi^0.
\end{equation}
We proceed in three steps. First we embed the discrete zero--drift walks into continuous--time processes, next we identify $L_n^{(4),0}/\sqrt{n}$ as a continuous functional of these processes, and finally we apply the functional central limit theorem and the continuous mapping theorem.
We fix $k\in \cI^0$ and define the rescaled, piecewise--linear process
\[
W_n^{(k)}(t)
:=
\frac{1}{\sqrt{n}}\,\left(S_{\lfloor nt\rfloor}^{(k)} +(nt - \lfloor nt\rfloor) (S_{\lfloor nt\rfloor + 1}^{(k)} - S_{\lfloor nt\rfloor}^{(k)}) \right),
\qquad t\in[0,1].
\]
By the multidimensional Donsker invariance principle 
\[
W_n^{(k)} \ \tod{n} B^{(k)}
\quad\text{in } C([0,1];\mathbb{R}^2).
\]
Since the walks $(S^{(k)})_{k\in \cI^0}$ are independent, the joint convergence holds
\begin{equation}\label{eq:fclt}
    \big(W_n^{(k)}\big)_{k\in \cI^0}
\ \tod{n}
\big(B^{(k)}\big)_{k\in \cI^0}
\quad\text{in } C([0,1];(\mathbb{R}^2)^{|\cI^0|}).
\end{equation}
For any collection of continuous paths
$w=(w^{(k)})_{k\in \cI^0}\in C([0,1];(\mathbb{R}^2)^{|\cI^0|})$,
define the functional
\[
\Psi(w)
:=
\int_{\Theta_0}
\max\Big\{0,\ \max_{k\in \cI^0}\sup_{0\le t\le 1}\big(w^{(k)}(t)\cdot \be_\theta\big)\Big\}\,\ud\theta.
\]
By construction,
$
\sup_{0\le t\le 1} \be_\theta^\tra W_n^{(k)}(t)
=
n^{-1/2} \max_{0\le j\le n}{S_j^{(k)}(\theta)}$.
Hence,
\[
\frac{L_n^{(4),0}}{\sqrt{n}}
=
\Psi\big((W_n^{(k)})_{k\in \cI^0}\big),
\qquad
\Pi^0=\Psi\big((B^{(k)})_{k\in \cI^0}\big).
\]
We now show that,  with respect to the uniform norm
$\|w\|_\infty:=\max_{k\in \cI^0}\sup_{t\in[0,1]}\|w^{(k)}(t)\|$,
 $\Psi$ is continuous.
For $w,u\in C([0,1];(\mathbb{R}^2)^{|\cI^0|})$ and each $\theta$,
\begin{align*}
\Big|
\max_{k\in \cI^0}\sup_{t\in[0,1]} (\be_\theta^\tra w^{(k)}(t))
-
\max_{k\in \cI^0}\sup_{t\in[0,1]} (\be_\theta^\tra u^{(k)}(t) )
\Big|
&\le
\max_{k\in \cI^0}\sup_{t\in[0,1]} |\be_\theta^\tra (w^{(k)}(t)-u^{(k)}(t)) |\\
&\le \|w-u\|_\infty.
\end{align*}
Taking positive parts does not increase differences, hence integrating over $\Theta_0$ yields
\[
|\Psi(w)-\Psi(u)|
\le {\rm Leb}(\Theta_0)\,\|w-u\|_\infty
\le 2\pi\,\|w-u\|_\infty.
\]
Thus $\Psi$ is continuous.
By equation \eqref{eq:fclt} and the continuous mapping theorem,
\[
\Psi\big((W_n^{(k)})_{k\in I^0}\big)
\ \tod{n}\
\Psi\big((B^{(k)})_{k\in I^0}\big),
\]
which proves~\eqref{eq:Pi-0-convergence}.

\subsubsection*{The main result}
We are now ready for the proof of our main result. 
\begin{proof}[Proof of Theorem \ref{thm:perim_main_theorem}]
From the approximation steps established earlier,
\[
\frac{L_n^{(1)}-L_n^{(4)}}{\sqrt{n}}\toL{n}{2}0,
\qquad
\frac{L_n^{(4),+}-L_n^{(5),+}}{\sqrt{n}}\toL{n}{2}0,
\]
and
\[
\frac{L_n^{(5),+}-n\,\perim(\cC_\bmu)}{\sqrt{n}}\tod{n} \Pi^+,\qquad
\frac{L_n^{(4),0}}{\sqrt{n}}\tod{n} \Pi^0.
\]
Since $L_n^{(4)}=L_n^{(4),+}+L_n^{(4),0}$,
Slutsky's theorem yields~\eqref{eq:perim-main-limit}, with
$\Pi^+$ and $\Pi^0$   as defined at~\eqref{eq:Pi-defs}.
Finally,
the convergence in distribution combined with the fact that $n^{-1} (L_n - \Exp L_n)^2$ is uniformly integrable, by Corollary~\ref{cor:perim-ui}, implies that $n^{-1} \Var L_n$ converges to $\Var \Pi^0+\Var \Pi^+$.
\end{proof}

Finally, we give the proof of Corollary~\ref{cor:zero-in-interior-clt}.

\begin{proof}[Proof of Corollary~\ref{cor:zero-in-interior-clt}]
Under the hypotheses in the corollary, we can apply Theorem~\ref{thm:perim_main_theorem} with $\Pi^0$ and $\Pi^+$ given by~\eqref{eq:Pi-defs} evaluating to $\Pi^0 = 0$ and
\[ \Pi^+ = \sum_{k=1}^m \int_{\Theta_{\bmu_k}} \be_\theta^\tra X^{(k)} \ud \theta .\]
Thus $\Pi^+$ is a sum of independent Gaussians, and its variance is identified by computing
\[ \Exp \int_{\Theta_{\bmu_k}} \int_{\Theta_{\bmu_k}} \be_\theta^\tra X^{(k)} (X^{(k)})^\tra \be_\phi \ud \theta \ud \phi = \int_{\Theta_{\bmu_k}} \int_{\Theta_{\bmu_k}} \be_\theta^\tra X^{(k)} \Sigma_k \be_\phi  \ud \theta \ud \phi .\]
Now observe that
\[  \int_{\Theta_{\bmu_k}} \be_\theta \ud \theta  = \hbmu_{k+1,k} - \hbmu_{k,k-1}. \]
    This completes the proof.
\end{proof}


\begin{appendix}

\section{Geometric interpretation of Theorem \ref{thm:perim_main_theorem}}
\label{sec:semi-cauchy}

We  give a geometric interpretation of the limit from Theorem \ref{thm:perim_main_theorem}. We first show a general result, and then apply it to our setting. The general result is shown for strictly convex sets, but is easily transferred to our setting (by an approximating sequence, cf.\ Remark \ref{rem:nonsmooth}). Strictly convex set is a convex set whose boundary contains no line segments, i.e.\ we say that a convex set $K\subset\R^2$ is strictly convex if for every two distinct points $x, y \in K$, $x \neq y$, we have
\begin{equation*}
    tx + (1-t)y \in \mathrm{int}(K), \qquad  t \in (0, 1).
\end{equation*}

\begin{lemma}\label{lem:partial_cauchy}
Let $K\subset\mathbb{R}^2$ be a compact, strictly convex set with $C^2$ boundary.
For $\theta\in\mathbb{R}$, denote
$\be_\theta=(\cos\theta,\sin\theta)$, $\be_\theta^\perp=(-\sin\theta,\cos\theta)$,
and the support function of $K$ by $h_K(\theta)=\sup\{\be_\theta^\tra x:\ x\in K\}$.
Fix an interval $[\alpha,\beta]\subset\mathbb{R}$ with $\alpha<\beta<\alpha+2\pi$. Then the following hold.
\begin{thmenumi}[label=(\roman*)]
\item\label{lem:partial_cauchy-i}
For each $\theta\in[\alpha,\beta]$ there is a unique supporting point $x(\theta)\in\partial K$
such that $\be_\theta^\tra x(\theta)=h_K(\theta)$ and the outer unit normal at $x(\theta)$ equals $\be_\theta$.
The map $\theta\mapsto x(\theta)$ is $C^1$ and parametrizes the boundary arc
\[
\Gamma_{[\alpha,\beta]}:=\{x(\theta):\theta\in[\alpha,\beta]\}\subset\partial K
\]
in the anticlockwise direction.
\item\label{lem:partial_cauchy-ii}
If $\rho(\theta)$ denotes the radius of curvature of $\partial K$ at $x(\theta)$, then
$x'(\theta)=\rho(\theta)\,\be_\theta^\perp$,
and hence the length of the arc is
\[
\mathrm{length}(\Gamma_{[\alpha,\beta]})=\int_\alpha^\beta \rho(\theta)\,\ud\theta.
\]
\item\label{lem:partial_cauchy-iii}
One has the identity
\begin{equation}\label{eq:partial_cauchy_identity}
\int_\alpha^\beta h_K(\theta)\,\ud\theta
=
\mathrm{length}(\Gamma_{[\alpha,\beta]})
+
x(\alpha)\cdot \be_\alpha^\perp
-
x(\beta)\cdot \be_\beta^\perp.
\end{equation}
\end{thmenumi}
\end{lemma}

\begin{proof}
Since $K$ is compact and convex, for each $\theta$ the supremum defining $h_K(\theta)$ is attained.
Strict convexity implies that the supporting line with outer normal $\be_\theta$ touches $K$ at a unique point we denote by $x(\theta)$. The assumption $\partial K\in C^2$ implies that the Gauss map
(point $\mapsto$ outer unit normal) is a $C^1$ diffeomorphism from $\partial K$ onto the unit circle,
so $\theta\mapsto x(\theta)$ is $C^1$ and its image is precisely the boundary arc with outer normal angles in $[\alpha,\beta]$.
This proves~\ref{lem:partial_cauchy-i}.

To prove~\ref{lem:partial_cauchy-ii}, first note that a standard identity for $C^2$ strictly convex bodies relates the support function and the boundary point:
\begin{equation}\label{eq:x_of_theta}
x(\theta)=h_K(\theta)\,\be_\theta + h_K'(\theta)\,\be_\theta^\perp.
\end{equation}
Differentiating \eqref{eq:x_of_theta} and using $(\be_\theta)'=\be_\theta^\perp$ and $(\be_\theta^\perp)'=-\be_\theta$ we get
\[
x'(\theta)
=
h_K'(\theta)\be_\theta + h_K(\theta)\be_\theta^\perp + h_K''(\theta)\be_\theta^\perp - h_K'(\theta)\be_\theta
=
\big(h_K(\theta)+h_K''(\theta)\big)\be_\theta^\perp.
\]
Set $\rho(\theta):=h_K(\theta)+h_K''(\theta)$,
which is the radius of curvature (positive by strict convexity). Then $x'(\theta)=\rho(\theta)\be_\theta^\perp$
and $\|x'(\theta)\|=\rho(\theta)$, which gives (ii).

Finally, consider the scalar function
$F(\theta):=x(\theta)\cdot \be_\theta^\perp$.
It holds that
\[
F'(\theta)
=
x'(\theta)\cdot \be_\theta^\perp + x(\theta)\cdot (\be_\theta^\perp)'
=
x'(\theta)\cdot \be_\theta^\perp - x(\theta)\cdot \be_\theta.
\]
Now $x(\theta)\cdot \be_\theta=h_K(\theta)$ by definition of $x(\theta)$, and
$x'(\theta)\cdot \be_\theta^\perp=\rho(\theta)$ because $x'(\theta)=\rho(\theta)\be_\theta^\perp$.
Therefore $F'(\theta)=\rho(\theta)-h_K(\theta)$
and
\[
\int_\alpha^\beta h_K(\theta)\,\ud\theta
=
\int_\alpha^\beta \rho(\theta)\,\ud\theta - (F(\beta)-F(\alpha)).
\]
By (ii), $\int_\alpha^\beta\rho(\theta)\,\ud\theta=\mathrm{length}(\Gamma_{[\alpha,\beta]})$,
and $F(\theta)=x(\theta)\cdot \be_\theta^\perp$, so we obtain \eqref{eq:partial_cauchy_identity}.
\end{proof}

\begin{remark}\label{rem:nonsmooth}
Lemma~\ref{lem:partial_cauchy} was proved under $C^2$ strict convexity to make every step explicit.
However, the identity \eqref{eq:partial_cauchy_identity} has a natural extension to general compact convex sets $K$.
\begin{itemize}
\item For almost every $\theta$, the supporting face of $K$ with outer normal $\be_\theta$ is a single point;
call any measurable selection of such points $x(\theta)$.
\item The set of $\theta$ where the supporting face is an edge is finite for polygons, and negligible for integration.
\item One may approximate any convex $K$ in the Hausdorff metric by a sequence of smooth strictly convex bodies $(K_n)_{n \in \N}$ and pass to the limit in \eqref{eq:partial_cauchy_identity}.
\end{itemize}
In this way, formula \eqref{eq:partial_cauchy_identity} remains valid with $\Gamma_{[\alpha,\beta]}$ interpreted as
the (counterclockwise) boundary chain whose outer normal angles lie in $[\alpha,\beta]$,
and with $x(\alpha),x(\beta)$ interpreted as any supporting points corresponding to the endpoint normals.
\end{remark}

\begin{lemma}[Geometric interpretation of Theorem \ref{thm:perim_main_theorem}]
\label{lem:geometric-perim}
For each vertex $v\in \cE$ let $\Theta_v\subset[0,2\pi)$ denote the interior of the set of directions $\theta$ for which $v$ uniquely maximizes $u\mapsto u\cdot \be_\theta$ over $u\in \cC_\bmu$. Represent each $\Theta_v$ as an interval
\[
\Theta_v=(\alpha_v,\beta_v),
\qquad \alpha_v<\beta_v<\alpha_v+2\pi.
\]

\begin{thmenumi}[label=(\roman*)]
\item\label{lem:geometric-perim-i}
Recall that $\cI^0=\{k:\bmu_k=0\}$ and let $(B^{(k)})_{k\in \cI^0}$ be independent planar Brownian motions with covariances $\Sigma_k$.
Define the random convex body
\[
\mathcal{B}^0:=\hull\Big(\bigcup_{k\in \cI^0}\{B^{(k)}(t):0\le t\le1\}\Big).
\]
Then the zero-drift contribution in the limit can be written as
\[
\Pi^0=\int_{\Theta_0} h_{\mathcal{B}^0}(\theta)\,\ud\theta.
\]
Moreover, by Lemma~\ref{lem:partial_cauchy} and Remark~\ref{rem:nonsmooth},
\begin{equation}\label{eq:Y0_arc}
\Pi^0
=
\mathrm{length}\!\big(\Gamma^0_{[\alpha_0,\beta_0]}\big)
+
x^0(\alpha_0)\cdot \be_{\alpha_0}^\perp
-
x^0(\beta_0)\cdot \be_{\beta_0}^\perp,
\end{equation}
where $\Gamma^0_{[\alpha_0,\beta_0]}$ is the boundary arc of $\partial\mathcal{B}^0$ whose outer normal angles run from
$\alpha_0$ to $\beta_0$, and $x^0(\alpha_0),x^0(\beta_0)$ are supporting points of $\mathcal{B}^0$ at the endpoint normals.
\item
\label{lem:geometric-perim-ii}
Fix $\bmu\in \cE\setminus\{0\}$. Recall that $\cI_\bmu=\{k:\bmu_k=\bmu\}$ and $(X^{(k)})_{k\in \cI_\bmu}$ are independent Gaussian vectors $X^{(k)}\sim\mathcal{N}(0,\Sigma_k)$.
Define the (random) Gaussian polygon
\[
\mathcal{B}_\bmu:=\hull \Big(\{\0\}\cup\{X^{(k)}:k\in \cI_\bmu\}\Big).
\]
Then the corresponding drift-vertex contribution can be written as
\[
\Pi^+_\bmu=\int_{\Theta_\bmu} h_{\mathcal{B}_\bmu}(\theta)\,\ud\theta,
\]
and Lemma~\ref{lem:partial_cauchy} together with Remark~\ref{rem:nonsmooth} give the boundary-arc representation
\begin{equation}\label{eq:Ymu_arc}
\Pi^+_\bmu
=
\mathrm{length}\!\big(\Gamma^\bmu_{[\alpha_\bmu,\beta_\bmu]}\big)
+
x^\bmu(\alpha_\bmu)\cdot \be_{\alpha_\bmu}^\perp
-
x^\bmu(\beta_\bmu)\cdot \be_{\beta_\bmu}^\perp,
\end{equation}
where $\Gamma^\bmu_{[\alpha_\bmu,\beta_\bmu]}$ is the boundary arc of $\partial\mathcal{B}_\bmu$ with outer normal angles in
$[\alpha_\bmu,\beta_\bmu]$, and $x^\bmu(\alpha_\bmu),x^\bmu(\beta_\bmu)$ are corresponding supporting points.
\item
\label{lem:geometric-perim-iii}
The total limiting fluctuation is
\[
\Pi^+ + \Pi^0
=
\sum_{\bmu\in \cE\setminus\{\0\}} \Pi^+_\bmu + \Pi^0,
\]
so the full limit decomposes into a sum of boundary-arc functionals indexed by the vertices of the drift polygon $\cC_\bmu$.
\end{thmenumi}
\end{lemma}

\section{Martingale uniform square-integrability}
 \label{sec:square-ui}
 
For many of our distributional limit results, we  establish associated
$L^2$-approximation statements that yield variance asymptotics directly via the simpler approximant object. In the case of the variance convergence in Theorem~\ref{thm:perim_main_theorem} we use a more abstract uniform square-integrability argument, based on  the following martingale generalization of the uniform square-integrability property of random walk from~\eqref{eq:sum-mean-ui}, which is proved in a similar way.

 \begin{lemma}
 \label{lem:square-ui}
Let $(W_n)_{n \in \N}$ be a sequence of $\R$-valued stochastic processes $W_n := (W_{n,0}, \ldots, W_{n,n})$, with $W_{n,0} =0$. Write $\xi_{n,i} := W_{n,i+1} - W_{n,i}$, $0 \leq i \leq n-1$.
Suppose also:
\begin{thmenumi}[label=(\roman*)]
    \item
    \label{lem:square-ui-i}
For each $n\in\N$ there are nested $\sigma$-algebras  $\cF_{n,0}  \subseteq \cF_{n,1}  \subseteq \cdots \subseteq \cF_{n,n}$ for which $W_{n,i}$ is $\cF_{n,i}$-measurable for all $i \in \{0,\ldots,n\}$.
    \item
    \label{lem:square-ui-ii}
For all $n \in \N$, $\Exp [ \xi_{n,i}  \mid \cF_{n,i} ] =0$ for all $i \in \{0,\ldots, n-1\}$.
    \item
    \label{lem:square-ui-iii}
    It holds that   $\lim_{B \to \infty} \sup_{n \in \N} \sup_{0 \leq i \leq n-1} \Exp [ \xi^2_{n,i} \mathbbm{1}_{\{ | \xi_{n,i} | > B \}}  ] = 0$.
\end{thmenumi}
Then $\Exp [ W_{n,n}^2 ] < \infty$, and $(n^{-1} W_{n,n}^2 )_{n \in \N}$ is uniformly integrable.
 \end{lemma}
 \begin{proof}
It follows from~\ref{lem:square-ui-iii} that
  $\Exp [ \xi^2_{n,i} ] < \infty$ for all $n \in \N$ and all $i \in \{0,\ldots, n-1\}$.
  We adapt the proof for the case of i.i.d.~sums (see e.g.~\cite[pp.~20--21]{gut-stopped}), replacing the classical
Marcinkiewicz--Zygmund inequality by its martingale analogue, due to Burkholder. 
First, fix $\eps>0$ and fix (large enough) $B \in \RP$ such that $\Exp [ \xi_{n,i}^2 \1 { | \xi_{n,i} | > B } ] \leq \eps$ for all $n, i$ (which is possible by condition~\ref{lem:square-ui-iii}). Then define
\begin{align*}
\xi_{n,i}' & := \xi_{n,i} \mathbbm{1}_{\{ | \xi_{n,i} | \leq B \}} - 
\Exp [ \xi_{n,i} \mathbbm{1}_{\{ | \xi_{n,i} | \leq B \}} \mid \cF_{n,i} ], \\
\xi_{n,i}'' & := \xi_{n,i} \mathbbm{1}_{\{ | \xi_{n,i} | > B \}} - 
\Exp [ \xi_{n,i} \mathbbm{1}_{\{ | \xi_{n,i} | > B \}} \mid \cF_{n,i} ] .
\end{align*}
Consider the decomposition $W_{n,n} = W'_{n,n} + W''_{n,n}$, where
\[ W_{n,n}' := \sum_{i=0}^{n-1} \xi_{n,i}', \quad\text{and}\quad
W_{n,n}'' := \sum_{i=0}^{n-1} \xi_{n,i}'' .\]
By construction,  $\Exp [ \xi'_{n,i} \mid \cF_{n,i} ] = \Exp [ \xi''_{n,i} \mid \cF_{n,i} ] = 0$, a.s., $| \xi'_{n,i} | \leq 2B$, a.s.,    and $(\xi'_{n,i})_{0 \leq i \leq n-1}$ and $(\xi''_{n,i})_{0 \leq i \leq n-1}$ 
are martingale difference sequences. A calculation shows that
\[ \Exp [ ( \xi_{n,i}'' )^2 \mid \cF_{n,i} ] = \Exp [ \xi_{n,i}^2 \1 { | \xi_{n,i} | > B} \mid \cF_{n,i} ] - (  \Exp [ \xi_{n,i} \1 { | \xi_{n,i} | > B} ] \mid \cF_{n,i} )^2,\]
which, on taking expectations, yields $\Exp [ ( \xi_{n,i}'' )^2] \leq \eps$.

Observe that,  by orthogonality of martingale differences,
 \[ n^{-1} \Exp \bigl[ ( W''_{n,n})^2 \bigr] = n^{-1} \sum_{i=0}^{n-1} \Exp \bigl[  ( \xi''_{n,i} )^2   \bigr] 
 \leq  \eps   . \]
 Next, recall that Burkholder's inequality (e.g.~\cite[p.~506]{gut-graduate} or~\cite[p.~414]{chow-teicher}) says that, for every $p>1$ there is a constant $C_p < \infty$ such that
 $\Exp [ (W_{n,n}' )^ p] \leq C_p   \Exp [ Q_{n}^{p/2} ] $, 
 where $Q_n := (\xi'_{n,0})^2 + \cdots + (\xi'_{n,n-1})^2 \geq 0$.
 Since $| \xi'_{n,i} | \leq 2B$, we have the crude bound $Q_n \leq 4 B^2 n$.
 In particular, applying the Burkholder inequality for $p=4$ we 
 have
 \begin{align*}
 n^{-1} \Exp \bigl[  ( W_{n,n}')^2 \mathbbm{1}_{\{ | W_{n,n}' | \geq \alpha \sqrt{n} \}} \bigr] 
 & \leq \alpha^{-2} n^{-2} \Exp \bigl[  ( W_{n,n}')^4 \mathbbm{1}_{\{ | W_{n,n}' | \geq \alpha \sqrt{n} \}} \bigr] \\
 & \leq C_4 ( 16 B^4 n^2 ) n^{-2} \alpha ^{-2} = 16 C_4 B^4 \alpha^{-2}.  
 \end{align*}
 Now observe the inequality for $a, b \in \R$ and $y >0$ (proved by considering cases $a\geq b$, $a\leq b$)
 \[ (a + b)^2 \1 { a +b \geq 2y} \leq 4 a^2 \1{ a \geq y} + 4 b^2 \1 { b \geq y }.\]
Since $W_{n,n} = W_{n,n} - \Exp [ W_{n,n} ] = W_{n,n}' + W_{n,n}''$, applying the preceding inequality gives
 \begin{align*}
     \Exp \bigl[  W_{n,n}^2 \mathbbm{1}_{\{ |W_{n,n}| \geq 2\alpha \sqrt{n}\}}  \bigr]
     & \leq 4 \Exp \bigl[  (W'_{n,n} )^2 \mathbbm{1}_{\{ |W'_{n,n}| \geq \alpha \sqrt{n} \}} \bigr] + 4 \Exp \bigl[  (W''_{n,n})^2 \mathbbm{1}_{\{ |W''_{n,n}| \geq \alpha \sqrt{n} \}} \bigr] ,
 \end{align*}
 so combining the above bounds we get
 \[ \limsup_{\alpha \to \infty} \sup_{n \in \N} n^{-1} 
 \Exp [  W_{n,n}^2 \mathbbm{1}_{\{ |W_{n,n}| \geq 2\alpha \sqrt{n} \}} ]
 \leq \lim_{\alpha \to \infty} 64 C_4 B^4 \alpha^{-2} + 4 \eps = 4 \eps.
 \]
 Since $\eps>0$ was arbitrary, the result follows.
 \end{proof}

We will apply the following consequence of Lemma~\ref{lem:square-ui}  to the geometrical random walk functionals of interest.
One views these geometrical functionals (perimeter, diameter) as functions of (multiple) random walk increments, and it is required that they satisfy a uniform smoothness condition with respect to changes in every argument, expressed in~\eqref{eq:smoothness}.

 \begin{corollary}
     \label{cor:lipschitz-ui} Suppose $d \in \N$.
     Let $(F_n)_{n \in \N}$ be a sequence of functions   $F_n : (\R^d)^n \to \R$. Suppose  that there exists a constant $C <\infty$ such that, for all $n \in \N$ and all $i \in \{1,\ldots,n\}$,
     \begin{equation}
     \label{eq:smoothness} \left| F_n ( z_1, \ldots, z_i , \ldots , z_n ) - F_n ( z_1, \ldots, z_i' ,\ldots, z_n ) \right| \leq C \bigl( \| z_i \| + \| z'_i \| \bigr) .\end{equation}
 Suppose that $\theta_1, \theta_2, \ldots$ are i.i.d.~in $\R^d$ with $\Exp [ \| \theta_1 \|^2 ] < \infty$, and let $\zeta_n := F_n (\theta_1, \ldots, \theta_n )$. Then $\Exp [ \zeta_n^2 ] < \infty$, and $( n^{-1} ( \zeta_n - \Exp  \zeta_n  )^2 )_{n \in \N}$ is uniformly integrable.
 \end{corollary}
 \begin{proof}
 First note that by~\eqref{eq:smoothness}, $| F_n (z_1, \ldots, z_n ) | \leq | F_n (0,\ldots, 0) | + C \sum_{i=1}^n \| z_i \|$,
  which, together with the hypothesis $\Exp [ \| \theta_1\|^2 ] < \infty$, implies that  $\Exp [ \zeta_n^2 ] < \infty$.
     Denote by $\theta_1', \theta_2', \ldots$ an independent copy of the sequence $\theta_1, \theta_2, \ldots$, and let $\cF_n := \sigma ( \theta_1, \ldots, \theta_n)$. 
     We use a standard resampling construction of martingale differences to write 
     \[ \zeta_n - \Exp [ \zeta_n] = \sum_{i=1}^{n} \bigl( \Exp[ \zeta_{n} \mid \cF_{i} ] - \Exp [ \zeta_n \mid  \cF_{i-1} ]\bigr)
     = \sum_{i=1}^{n} \xi_{n,i}, \text{ where } \xi_{n,i} := \Exp[ \zeta_n - \resample{\zeta_n}{i} \mid \cF_{i} ] ,
     \]
and $\resample{\zeta_n}{i}$ is defined as $\zeta_n$ but using $\theta_i'$ in place of $\theta_i$. Writing $W_{n,i} := \sum_{j=0}^{i-1} \xi_{n,j}$, then $\xi_{n,i} = W_{n,i+1} - W_{n,i}$ (as in Lemma~\ref{lem:square-ui}) and setting $\cF_{n,i} = \cF_{i-1}$ for $i \geq 1$, conditions~\ref{lem:square-ui-i} and~\ref{lem:square-ui-ii} of Lemma~\ref{lem:square-ui} are readily checked (for the latter, note that $\Exp [ \zeta_n \mid \cF_{i-1} ] = \Exp [\resample{\zeta_n}{i} \mid \cF_{i-1}]$). For condition~\ref{lem:square-ui-iii}, note that the smoothness property~\eqref{eq:smoothness} implies that $| \zeta_n - \resample{\zeta_n}{i} | \leq C ( \| \theta_i \| + \| \theta'_i \| )$, a.s. It follows that, for some constant $C' < \infty$ (not depending on $n$ or $i$), 
\[ \| \xi_{n,i} \| \leq C \Exp [ \| \theta_i \| + \| \theta'_i \| \mid \cF_i ]
\leq C' ( 1 + \| \theta_i \| ) . \]
 Since the $\theta_i$ are i.i.d., and $\Exp [ \| \theta_1 \|^2] < \infty$, condition~\ref{lem:square-ui-iii} is verified by dominated convergence. Hence we can apply Lemma~\ref{lem:square-ui} to obtain the result, noting that $W_{n,n} = \zeta_{n-1} - \Exp [ \zeta_{n-1}]$.
 \end{proof}

Here is the result applied to perimeter; there is a corresponding statement for diameter.
 
 \begin{corollary}
     \label{cor:perim-ui} 
If~\eqref{ass:many-walks} holds then $(n^{-1} (L_n - \Exp L_n)^2 )_{n \in \N}$
is uniformly integrable.
\end{corollary}
\begin{proof}
 The result follows from Corollary~\ref{cor:lipschitz-ui} applied with 
 \[ F_n ( z_1, \ldots, z_n ) = \perim \hull \biggl\{  \sum_{j=1}^i z_{k,j} : k \in \{1,\ldots , N\}, 0 \leq i \leq n \biggr\}, \]
 where $z_i = (z_{1,i}, \ldots, z_{N,i}) \in \R^{2N}$ and $\theta_i = ( Z_{1,i}, \ldots , Z_{N,i} )$ are the joint random walk increments, so that
 $L_n = F_n (\theta_1, \ldots, \theta_n) = \zeta_n$, with the notation of Corollary~\ref{cor:lipschitz-ui}. For the smoothness hypotheses~\eqref{eq:smoothness},
 note that, by Cauchy's formula for perimeter
 \begin{align*}
     | F_n (z_1,   \ldots , z_n ) - F_n (w_1, \ldots, w_n) | & = \int_0^{2\pi} \Bigl( \max_{1 \leq j \leq n} \be_\theta^\tra z_j  - \max_{1 \leq j \leq n} \be_\theta^\tra w_j \Bigr) \ud \theta,
 \end{align*}
 where $w_j = z_j$ for $j \neq i$ and $w_i = z_i'$. Here
 $| \max_{1 \leq j \leq n} \be_\theta^\tra z_j - \max_{1 \leq j \leq n} \be_\theta^\tra w_j | \leq \| z_i - z_i' \|$, so we verify~\eqref{eq:smoothness} with $C = 2\pi$.
\end{proof}

\section{Multiple zero-drift walks: Weak convergence}
\label{sec:donsker}

Consider the case of  $N \in \N$ random walks all with zero drift, so that~\eqref{ass:many-walks} holds with  $\bmu_1 = \cdots = \bmu_N = \0$. Since the walks are independent, Donsker's theorem implies that, in the sense of weak convergence on the space $\cD^{2N}_0$ of c\`adl\`ag functions  $f : [0,1] \to \R^{2N}$, with $f(0) = \0$, endowed with the Skorokhod metric, 
\[ \frac{ \bigl( S^{(1)}_{\lfloor n t\rfloor} , \ldots, S^{(N)}_{\lfloor nt \rfloor} \bigr)_{t \in [0,1]}}{\sqrt{n}} \tod{n} \bigl(\Sigma_1^{1/2} B^{(1)}_t, \ldots, \Sigma_N^{1/2} B^{(N)}_t \bigr)_{t \in [0,1]} ,\]
where $\Sigma_k^{1/2}$ is the symmetric square root of~$\Sigma_k$ in~\eqref{ass:many-walks},
and $(B^{(1)}_t)_{t\in\RP}, \ldots, (B^{(N)}_t)_{t\in\RP}$ are independent standard Brownian motions on $\R^2$. The functional
from $\cD_0^2$ to $\cK_0$ defined by
$f \mapsto \cl \hull f [0,1]$ (`$\cl$' denotes closure) is continuous  (see e.g.~Lemma~6.2 of~\cite{lmw}), and hence, by continuous mapping, it is an immediate extension of e.g.~Theorem 6.4 of~\cite{lmw} that
\[ \frac{ \bigl( \cH^{(1)}_n, \ldots, \cH^{(N)}_n \bigr) }{\sqrt{n}} 
\tod{n} \bigl( \Sigma_1^{1/2} \hull B^{(1)} [0,1] , \ldots,  \Sigma_N^{1/2} \hull B^{(N)} [0,1] \bigr), \]
in the sense of weak convergence on $\cK_0^{\otimes N}$. On $\cK_0^{\otimes N}$, with the product-Hausdorff topology,  $(A_1, \ldots, A_N) \mapsto \hull ( A_1 \cup \cdots \cup A_N )$ is continuous, and since $\cH_n = \hull ( \cH_n^{(1)} \cup \cdots \cup \cH_n^{(N)} )$, another application of continuous mapping shows,
in the sense of weak convergence on $\cK_0$, 
\begin{equation}
    \label{eq:N-donsker}
\frac{ \cH_n}{\sqrt{n}}  \tod{n} \hull \bigl( \Sigma^{1/2}_{k} B^{(k)} [0,1] : 1 \leq k \leq N \bigr),
\end{equation}
the joint convex hull of $N$ Brownian motions, with appropriate linear transformations.  Now~\eqref{eq:N-donsker} together with continuity   of the diameter and perimeter functionals yield
\begin{align*}
    \frac{ D_n}{\sqrt{n}}  & \tod{n} \diam 
    \hull \bigl( \Sigma^{1/2}_{k} B^{(k)} [0,1] : 1 \leq k \leq N \bigr); \\
     \frac{ L_n}{\sqrt{n}}  & \tod{n} \perim 
    \hull \bigl( \Sigma^{1/2}_{k} B^{(k)} [0,1] : 1 \leq k \leq N \bigr).
\end{align*}
This is the natural extension of the case of $N=1$ walk examined in~\cite{wx-scaling}; the limit random variables in the last display are non-Gaussian, since they are a.s.~positive, and variations on the arguments in~\cite{wx-scaling} can be used to show they have positive variances.

\section{Two auxiliary lemmas}\label{sec:two_aux_lemmas}
In this part of the appendix we state and prove two auxiliary lemmas used in the proof of Theorem \ref{thm:perim_main_theorem}. The first gives Gaussian-type tail bounds.
\begin{lemma}\label{lem:overshoot}
Let $X$ be a  random variable such that for some $\sigma^2>0$, $\mathbb{E}[\re^{\lambda X}]\le \re^{\lambda^2\sigma^2/2}$ for all $\lambda>0$. Then for every $a>0$,
\[
\mathbb{E}\big[ ((X-a)^+)^2\big]\le 2\sigma^2\exp\Big(-\frac{a^2}{2\sigma^2}\Big).
\]
\end{lemma}

\begin{proof}
By Chernoff's bound \cite[page 21]{BLM-concentration}, for all $t>0$,
\[
\mathbb{P}(X\ge t)\le \inf_{\lambda>0}\exp\Big(-\lambda t+\frac{\lambda^2\sigma^2}{2}\Big)
=\exp\Big(-\frac{t^2}{2\sigma^2}\Big),
\]
with the infimum attained at $\lambda=t/\sigma^2$.
Recall that for any  random variable $Y \geq 0$ we have
\[
\mathbb{E}[Y^2]=\int_0^\infty 2u\,\mathbb{P}(Y\ge u)\,\ud u.
\]
With $Y=(X-a)^+$, we have $\{(X-a)^+\ge u\}=\{X\ge a+u\}$. Therefore
\[
\mathbb{E}\big[ ((X-a)^+)^2\big]
=\int_0^\infty 2u\,\mathbb{P}(X\ge a+u)\,\ud u
\le \int_0^\infty 2u \exp\Big(-\frac{(a+u)^2}{2\sigma^2}\Big)\,\ud u.
\]
Let $t=a+u$ so $u=t-a$ and $t\in[a,\infty)$. Since $t-a\le t$ for $t\ge0$,
\[
\mathbb{E}\big[ ((X-a)^+)^2\big]
\le \int_a^\infty 2t \exp\Big(-\frac{t^2}{2\sigma^2}\Big)\,\ud t
= 2\sigma^2\exp\Big(-\frac{a^2}{2\sigma^2}\Big),
\]
where the last equality is obtained by differentiating $\exp(-t^2/(2\sigma^2))$.
\end{proof}

\begin{lemma}\label{lem:geom_int}
Let $v\in\mathbb{R}^2\setminus\{0\}$ and $c>0$. Recall that $\be_\theta = (\cos \theta, \sin \theta)$. Then for all $n\ge 1$, and for any measurable set $A\subseteq[0,2\pi)$,
\begin{equation}
    \label{eq:geom-int-bound}
\int_A\exp\big(-c\,n\,(v\cdot \be_\theta)^2\big)\,\ud\theta
\le \frac{\pi^{3/2}}{\|v\|\sqrt{c\,n}}.
\end{equation}
\end{lemma}

\begin{proof}
Write $v=\|v\|(\cos\phi,\sin\phi)$ for some $\phi\in[0,2\pi)$. Then
$
v\cdot \be_\theta=\|v\|\cos(\theta-\phi)
$.
By the change of variables $u=\theta-\phi$ and $2\pi$-periodicity,
\[
\int_0^{2\pi}\exp\big(-c\,n\,(v\cdot \be_\theta)^2\big)\,\ud\theta
=
\int_0^{2\pi}\exp\big(-c\,n\,\|v\|^2\cos^2 u\big)\,\ud u.
\]
Now split the integral into four intervals of length $\pi/2$:
\[
\int_0^{2\pi} \re^{-c n\|v\|^2\cos^2 u}\,\ud u
=4\int_0^{\pi/2} \re^{-c n\|v\|^2\cos^2 u}\,\ud u.
\]
On $[0,\pi/2]$ we have $\cos u \ge 2(\frac{\pi}{2}-u)/\pi$, hence
$\cos^2 u \ge \frac{4}{\pi^2} (\frac{\pi}{2}-u )^2$.
Let $x=\frac{\pi}{2}-u$. Then
\[
\int_0^{\pi/2} \re^{-c n\|v\|^2\cos^2 u}\,\ud u
\le \int_0^{\pi/2}\exp\Big(-c n\|v\|^2\cdot \frac{4x^2}{\pi^2}\Big)\,\ud x
\le \int_0^{\infty}\exp\Big(-\frac{4c}{\pi^2} n\|v\|^2 x^2\Big)\,\ud x.
\]
Using $\int_0^\infty \re^{-ax^2}\,\ud x=\frac{1}{2}\sqrt{\pi/a}$, we obtain
\[
\int_0^{\pi/2} \re^{-c n\|v\|^2\cos^2 u}\,\ud u
\le 
 \frac{\pi^{3/2}}{4\|v\|\sqrt{c\,n}},
\]
from which we get the $A = [0,2\pi)$ case of~\eqref{eq:geom-int-bound},
and the case of arbitrary $A \subseteq [0,2\pi)$ follows immediately.
\end{proof}

\end{appendix}

\section*{Acknowledgments}
TK was a visiting PhD student at Durham University, April--June 2024, with support of a  British Scholarship Trust award.
NS and S\v S were supported by  the European Union – NextGenerationEU through the National Recovery and Resilience Plan 2021-2026 Institutional grant of University of Zagreb Faculty of Science (IK IA 1.1.3. Impact4Math), and Institutional grant of University of Zagreb Faculty of Electrical Engineering and Computing (VALOR). NS and S\v{S} also received support from grant DIGIT.2.1.02.016 funded by the Digital, Innovation, and Green Technology Project – DIGIT Project (IBRD Loan No. 9558‑HR).
TK, NS and S\v S were supported by  Croatian Science Foundation grant no.~2277.
AW was supported by EPSRC grant EP/W00657X/1,
and by the IDUB Visiting Professors Programme at University of Wroc\l aw. 
Some of this work was carried
 out during the programme ``Stochastic systems for anomalous diffusion'' (July--December 2024) hosted by the  Isaac Newton Institute, under EPSRC grant EP/Z000580/1. 
An early version of the $L^2$ approximation for the Wald theorem (\S\ref{sec:wald}), 
 based on the argument in~\cite{wx-drift}, 
 was worked out in discussions with 
Jianshan Zeng during an undergraduate summer project supported by the Department of Mathematical Sciences, Durham University, in June--July 2017.  
 
\printbibliography

\end{document}